\documentclass[11pt, reqno]{amsart}
\usepackage{coisotropics}
\usepackage[margin=1in]{geometry}
\usetikzlibrary{decorations.pathreplacing}

\newcommand{\defn}[1]{\textbf{\textit{#1}}}

\newcommand{\reg}{^{\mathrm{reg}}}

\newcommand{\pr}{\operatorname{pr}}

\newcommand{\Id}{\mathrm{Id}}

\newcommand{\coiso}{\operatorname{coiso}^0}
\newcommand{\coisoeq}{\operatorname{coiso}}
\newcommand{\transfer}{\mathsf{T}^0}
\newcommand{\transfereq}{\mathsf{T}}
\newcommand{\htimes}{\times^h}
\newcommand{\stimes}{\times^s}
\renewcommand{\B}{\mathbf{B}}
\newcommand{\obj}{^{0}}
\newcommand{\arr}{^{1}}
\newcommand{\comp}[1]{^{#1}}
\newcommand{\mmm}{\mathsf{m}}
\newcommand{\iii}{\mathsf{i}}
\newcommand{\uuu}{\mathsf{1}}

\renewcommand{\C}{\mathcal{C}}
\renewcommand{\c}{c}
\newcommand{\IM}{\sigma}
\newcommand{\aaa}{\rho}
\newcommand{\conn}{\tau}
\newcommand{\dconn}{\check{\conn}}
\newcommand{\qad}{\operatorname{Ad}}
\newcommand{\curv}{K}
\newcommand{\ME}{\mathcal{M}}
\newcommand{\moment}{\mu}

\newcommand{\act}{\mathsf{a}}

\title[Shifted coisotropic structures for differentible stacks]{Shifted coisotropic structures for\\ differentiable stacks}
\author{Maxence Mayrand}
\address[Maxence Mayrand]{D\'epartement de math\'ematiques, Universit\'e de Sherbrooke, 2500 Bd de l'Universit\'e, Sherbrooke, QC, J1K 2R1, Canada}
\email{maxence.mayrand@usherbrooke.ca}
\date{\today}

\begin{document}

%We give a differential-geometric definition of coisotropics in 1-shifted symplectic stacks using twisted Dirac structures and show that it satisfies the expected properties analogous to the algebraic situation.
%In particular, intersections of 1-shifted coisotropics are 0-shifted Poisson, and 1-shifted coisotropic structures transfer through Morita equivalences.
%Examples of 1-shifted coisotropics that are not Lagrangians include actions of quasi-symplectic groupoids on Dirac manifolds.
%Their associated reduction theory is recovered via intersection and Morita transfer.
%Most results are formulated with clean-intersection conditions weaker than transversality, while still avoiding derived geometry.
%As a special case, we also obtain results about 1-shifted Lagrangians on differentiable stacks that have not appeared elsewhere.
%
\begin{abstract}

We introduce a notion of coisotropics on 1-shifted symplectic Lie groupoids (i.e.\ quasi-symplectic groupoids) using twisted Dirac structures and show that it satisfies properties analogous to the corresponding derived-algebraic notion in shifted Poisson geometry.
In particular, intersections of 1-coisotropics are 0-shifted Poisson.
We also show that 1-shifted coisotropic structures transfer through Morita equivalences, giving a well-defined notion for differentiable stacks.
Most results are formulated with clean-intersection conditions weaker than transversality while avoiding derived geometry.
Examples of 1-coisotropics that are not necessarily Lagrangians include Hamiltonian actions of quasi-symplectic groupoids on Dirac manifolds, and this recovers several generalizations of Marsden--Weinstein--Meyer's symplectic reduction via intersection and Morita transfer.
%Their intersection and Morita transfer recover . %and introduces new ones.

%In the special case of 1-shifted Lagrangians, most of our results are expected, but did not appear elsewhere to the author's knowledge.
%Their reduction theory is recovered via coisotropic intersection and Morita transfer.
%Most results are formulated with clean-intersection conditions weaker than transversality, while still avoiding derived geometry.
%As a special case, we also obtain results about 1-shifted Lagrangians on differentiable stacks that have not appeared elsewhere.
%We also give an infinitesimal version of these results giving a notion of 1-shifted coisotropics in twisted Dirac manifolds, with the expected intersection properties.
\end{abstract}

\maketitle
\tableofcontents

\section{Introduction}

%\subsection{Overview}
Differential geometry has recently benefited tremendously from interacting with the theory of derived algebraic stacks.
The root of this %fruitful 
interaction is that differentiable stacks can be viewed as Lie groupoids up to Morita equivalences \cite{beh-xu:03}, which have been at the core of much of differential geometry for several decades, either explicitly or implicitly.
Significant advances have thus been made by translating concepts related to Lie groupoids in the language of stacks, providing new insights \cite{tse-zhu:05,tse-zhu:06,ler:10,saf:16,saf:17a,saf:17b,pri:20,saf:21b,bai-gua:23}, powerful machinery \cite{zhu:09,cue-zhu:23,ber-ler:20,rog-zhu:20,pym-saf:20}, and far-reaching generalizations \cite{blo:08,cab-hoy-puj:20,bur-nos-zhu:20,hof:20,ban-che-sti-xu:20,hof-sja:21}. % of those concepts.
%For instance, the integration problem of Lie algebroid has been clarified \cite{tse-zhu:06} \cite{tse-zhu:05} using higher stacks, and symplectic reduction \cite{mar-wei:74} can be viewed as an instance of a general construction in derived symplectic geometry \cite{ptvv:13,cal:14}.
%We will come back to the latter in a moment, as it forms the basis of this paper.
A considerable body of work has therefore been devoted to constructing ``stacky'' versions of differential-geometric notions, such as 
volume \cite{wei:09},
vector fields \cite{hep:09}, %,ber-ler:20},
representations \cite{aba-cra:13}, 
Riemannian metrics \cite{hoy-fer:19},
measures \cite{cra-mes:19},
vector bundles \cite{hoy-ort:20},
equivariant cohomology \cite{bar-neu:21}, 
Poisson structures \cite{bon-cic-lau-xu:22}, 
contact structures \cite{mag-tor-vit:23}, 
gerbes \cite{beh-xu:03,lau-sti-xu:09,beh-xu:11}, %cha-kou:20},
$K$-theory \cite{tu-xu-lau:04},
%symplectic structures \cite{cue-zhu:23}, 
%differential forms \cite{get:14},
%Lusternik-Schnirelmann category \cite{als-col-neu:17},
and more.

Most relevant to us is the differential-geometric version of Pantev--To\"en--Vaqui\'e--Vezzosi's shifted symplectic structures \cite{ptvv:13} --- a powerful generalization of symplectic geometry on derived algebraic stacks with applications to quantum field theory.
The ``shift'' is an integer $n$ that recovers ordinary symplectic geometry when set to $0$ and the derived stack is a manifold.
In the case of shift $n = 1$ on a differentiable stack, this recovers the notion of quasi-symplectic groupoids up to gauge transformations \cite{get:14,cal:21,cue-zhu:23}, which was introduced around a decade earlier by Xu \cite{xu:04} and Bursztyn--Crainic--Weinstein--Zhu \cite{bur-cra-wei-zhu:04}.
These groupoids are the integrations of Dirac structures, thus placing Dirac geometry and its related geometries, such as Poisson \cite{bur-wei:05}, quasi-Poisson \cite{ale-kos-mei:02,bur-cra:05}, and generalized complex geometry \cite{gua:11,bai-gua:23}, firmly into the realm of differentiable stacks. % \cite{bai-gua:23}.
%a vast generalization of Poisson and quasi-Poisson structures, at the center of much modern differentiable geometry, such as generalized complex geometry \cite{} and ...
%This thus places Dirac geometry, and related geometries, firmly into the realm of shifted symplectic stacks.
%Higher shifts $n > 1$ have been explored more recently \cite{cue-zhu:23}.
% and lower shifts $n < 0$ require a version of ``derived differential geometry'' which is still in investigation.

One of the main existence results for shifted symplectic structures is the notion of Lagrangian intersection \cite{ptvv:13}.
Given two \emph{$n$-shifted Lagrangians}
\[
\begin{tikzcd}[row sep={2em,between origins},column sep={4em,between origins}]
\mathbf{L}_1 \arrow{dr} & & \mathbf{L}_2 \arrow{dl} \\
& \mathbf{X} &
\end{tikzcd}
\]
in an $n$-shifted symplectic stack $\mathbf{X}$, the derived fibre product $\mathbf{L}_1 \times_{\mathbf{X}} \mathbf{L}_2$ has an $(n-1)$-shifted symplectic structure.
In particular, the intersection of two $1$-shifted Lagrangians is $0$-shifted symplectic.
When translated into differential-geometric terms, this produces methods for constructing symplectic manifolds.
In particular, Marsden--Weinstein--Meyer symplectic reduction \cite{mar-wei:74,mey:73} and quasi-Hamiltonian reduction \cite{ale-mal-mei:98} arise in this way \cite{cal:14,cal:15,saf:16,cal:21,ane-cal:22}.
%, as well as its generalization to quasi-symplectic groupoids \cite{xu:04}, thus also including quasi-Hamiltonian reduction \cite{ale-mal-mei:98} and Mikami--Weinstein reduction \cite{mik-wei:88}.
As shown by P.\ Crooks, A.\ B\u{a}libanu, and the author \cite{cro-may:22,bal-may:22}, many other more exotic constructions also fall under this scheme, such as symplectic cutting \cite{ler:95,ler-mei-tol-woo:98} and symplectic implosion \cite{gui-jef-sja:02,hur-jef-sja:06}.
%, preimages of Poisson transversals \cite{fre-mar:17}, the Ginzburg--Kazhdan's construction \cite{gin.kaz:23} of the Moore-Tachikawa topological quantum field theory \cite{moo-tac:12}, and several others in geometric representation theory.
%many oin their study of the Moore--Tachikawa conjectural topological quantum field theory.
%Similarly, as shown in \cite{bal-may:22} multiplicative analogues of some of these examples can also be seen as intersection of 1-shifted Lagrangians, such as quasi-Hamiltonian implosion, multiplicative Whittaker reduction, and several other constructions in geometric representation theory.

On the other hand, a well-established construction that is \emph{not} an intersection of shifted Lagrangians is the reduction theory of Dirac manifolds \cite{bur-cra:05}.
Instead, as shown in this paper, it can be interpreted as an intersection of shifted \emph{coisotropics} ---
%\footnote{This idea was kindly suggested to the author by Henrique Bursztyn.}
a concept from shifted Poisson geometry \cite{cptvv,mel-saf:18a,mel-saf:18b,hau-mel-saf:22}.
%One also has the derived-algebraic notion of shifted Poisson structures \cite{cptvv} and shifted coisotropics \cite{}.
As for Lagrangians, derived fibre products of $n$-shifted coisotropics are $(n - 1)$-shifted Poisson.
However, although the differential-geometric version of 1-shifted Lagrangians is evident from the original paper \cite{ptvv:13}, the same cannot be said for shifted coisotropics, which are more deeply rooted in algebro-geometric concepts.
One of the goals of this paper is to introduce a definition of 1-shifted coisotropics on 1-shifted symplectic stacks in the language of Lie groupoids.
In short, a 1-shifted coisotropic is a Lie groupoid morphism $\c : \C \to \G$, where $\G$ is a quasi-symplectic groupoid, together with a twisted Dirac structure on the space of objects $\C\obj$ satisfying a compatibility and a non-degeneracy condition (Definition \ref{gm52z93m}).
We show that this notion has the expected properties analogous to the derived-algebraic case:

\begin{theorem}[Properties of 1-coisotropics]
Under suitable non-degeneracy conditions,
\begin{enumerate}[label={\textup{(\arabic*)}}]
\item \label{g7rgs6l8}
\textit{intersections of 1-shifted coisotropics are 0-shifted Poisson (Theorem \ref{2jy36i7i} and Theorem \ref{rt0qkos3}),}
\item \label{xea3qq2x}
\textit{1-shifted coisotropic structures transfer through Morita equivalences (Theorem \ref{vazibeld}), and}
\item \label{lwnzxn8d}
\textit{1-shifted coisotropics have an induced 0-shifted Poisson structure (Corollary \ref{t61dwbsa}).}
%\item \label{fcl2b568}
%\textit{the identity morphism of a quasi-symplectic groupoid is a 1-shifted coisotropic.}
\end{enumerate}
\end{theorem}

Here, a 0-shifted Poisson structure on a Lie groupoid $\G$ is a Dirac structure $L$ on $\G\obj$ such that $\sss^*L = \ttt^*L$ and $\im \aaa_\G = \ker L$, where $\aaa_\G$ is the anchor map.
It can be interpreted as a ``transverse Poisson structure.''
In particular, if the orbit space $\G\obj/\G\arr$ is a manifold, then $0$-shifted Poisson structures on $\G$ are in one-to-one correspondence with ordinary Poisson structures on $\G\obj/\G\arr$.
%A 0-shifted Poisson structure can also be seen as a 1-shifted coisotropic structure over a point.

Property \ref{g7rgs6l8} will follow from a stronger result about the composition of 1-shifted coisotropic correspondences, i.e.\ a shifted version of Weinstein's coisotropic calculus \cite{wei:88} and a differential-geometric version of \cite{hau-mel-saf:22}.
Property \ref{xea3qq2x} will be used to define 1-shifted coisotropic structures on morphisms of differentiable stacks using the equivalence of bicategories from Lie groupoids localized at essential equivalences to differentiable stacks \cite{beh-xu:11,pro:96}. %\cite{beh-xu:11,hoy:13,ler:10}
%Property \ref{xea3qq2x} is the most involved result of this paper.
The proof uses the adjoint representation up to homotopy of Abad--Crainic \cite{aba-cra:13} to construct chain homotopies between maps on tangent complexes.% (\S\ref{f73gy0ci}).

The most basic examples of 1-shifted coisotropics are similar to the two ``extreme cases'' of coisotropics in ordinary symplectic geometry, i.e.\ Lagrangians and the whole space:

\begin{proposition}[Examples of 1-coisotropics]
The following are 1-coisotropics:
\begin{itemize} %\begin{enumerate}[label={\textup{(\arabic*)}}]
\item 1-shifted Lagrangians in the sense of \cite{ptvv:13}, and
\item the identity morphism $\G \to \G$ of a quasi-symplectic groupoid (Proposition \ref{2kzh0or3}).
\end{itemize} %\end{enumerate}
\end{proposition}

Non-trivial examples come from Hamiltonian actions of quasi-symplectic groupoids on twisted Dirac manifolds:

\begin{proposition}[Hamiltonian actions as 1-coisotropics (Proposition \ref{ww8x8m9h})]\label{yrsedx4p}
Let $\G$ be a quasi-symplectic groupoid acting on a manifold $M$.
Then a 1-shifted coisotropic structure on the projection $\G \ltimes M \to \G$ is the same as a twisted Dirac structure on $M$ making the $\G$-action Hamiltonian.
\end{proposition}

Choosing an orbit $\O \s \G\obj$ of $\G$ induces another 1-coisotropic $\G|_\O \to \G$ (Proposition \ref{3pl9nrs8}), and we will see (Corollary \ref{2wut6us3}) that its intersection with the one in Proposition \ref{yrsedx4p} recovers Dirac reduction \cite{bur-cra:05,xu:04} via Properties \ref{g7rgs6l8} and \ref{xea3qq2x}.
Changing $\G|_\O \to \G$ to other 1-coisotropics gives new notions of reduction (see Theorem \ref{nrztzlyn} below).

We will also formulate versions of Properties \ref{g7rgs6l8} and \ref{xea3qq2x} under weaker assumptions that go outside the corresponding derived-algebraic properties but are more useful for differential-geometric applications. % as they don't require free actions nor transverse intersections.
%These results are also new (to the authors knowledge) in the special case of shifted Lagrangians.
To explain this, consider the basic example of a Lie group $G$ acting on a symplectic manifold $M$ in a Hamiltonian way with moment map $\mu : M \to \g^*$.
If the action is free, we get two transverse 1-shifted Lagrangians whose intersection is the action groupoid $G \ltimes \mu^{-1}(0)$, which is therefore 0-shifted symplectic.
%Then the cotangent groupoid $T^*G \tto \g^*$ is a symplectic groupoid taking the role of the 1-shifted symplectic stack $\mathbf{X}$ above.
%We get two 1-shifted Lagrangians: the inclusion $G \to T^*G$ over the $0$-fibre and the action groupoid $G \ltimes M \to T^*G$. % = G \times \g^* : (g, p) \mto (g, \mu(p))$.
%If the action of $G$ on $M$ is free and proper, they intersect transversely into the action groupoid $G \ltimes \mu^{-1}(0)$, which is therefore 0-shifted symplectic.
If the action is also proper, $G \ltimes \mu^{-1}(0)$ is Morita equivalent to the manifold $\mu^{-1}(0)/G$, so we get an ordinary symplectic structure on $\mu^{-1}(0)/G$, recovering Marsden--Weinstein--Meyer reduction \cite{mar-wei:74,mey:73}.
On the other hand, symplectic reduction holds under more generality, i.e.\ it suffices to assume that $\mu^{-1}(0)$ and $\mu^{-1}(0)/G$ are manifolds and the quotient map is a submersion (but $G$ does not necessarily act freely nor properly).
Although very simple, this generalization cannot be obtained as a Lagrangian intersection since the intersection is a derived stack not presentable by a Lie groupoid and the manifold $\mu^{-1}(0)/G$ is not Morita equivalent to the action groupoid $G \ltimes \mu^{-1}(0)$ due to non-trivial isotropy groups.
For these reasons, we prove versions of Properties \ref{g7rgs6l8} and \ref{xea3qq2x} using clean intersections and a weaker notion of Morita equivalence.
As a corollary, we obtain the following.

\begin{theorem}[Dirac reduction along a 1-coisotropic]\label{nrztzlyn}
Let $\G$ and $\H$ be quasi-symplectic groupoids such that $\G \times \H$ acts on a twisted Dirac manifold $M$ in a Hamiltonian way with moment map
\[
(\mu, \nu) : M \too \G\obj \times \H\obj.
\]
Let $\c : \C \to \G$ be a 1-coisotropic.
If the intersection $\C\obj \times_{\G\obj} M$ is clean and the quotient
\[
M \sll{\c} \G \coloneqq (\C\obj \times_{\G\obj} M) / \C
\]
is a manifold, then $M \sll{\c} \G$ has the structure of a Hamiltonian $\H$-space.
If $\C$ acts locally freely on $\C\obj \times_{\G\obj} M$, then the fibre product $\C\obj \times_{\G\obj} M$ is transverse.
\end{theorem}

%In particular, this recovers symplectic reduction along a submanifold \cite{cro-may:22} and a global version of the notion of reduction along strong Dirac maps introduced in \cite{bal-may:22}.
We thus think of 1-coisotropics as generalizations of the ``level'' in ordinary symplectic reduction.
For instance:

\begin{example}\
\begin{itemize}
\item
Taking $\c : \G|_\O \hookrightarrow \G$ for an orbit $\O$ recovers Dirac reduction \cite{bur-cra:05,xu:04} and hence also Marsden--Weinstein--Meyer reduction \cite{mar-wei:74,mey:73} and quasi-Hamiltonian reduction \cite{ale-mal-mei:98}.

\item
If $\c: \C \hookrightarrow \G$ is a 1-shifted Lagrangian subgroupoid, this recovers symplectic reduction along a submanifold \cite{cro-may:22} and a global version of the notion of reduction along strong Dirac maps introduced in \cite{bal-may:22}.
In particular, it includes symplectic cutting \cite{ler:95,ler-mei-tol-woo:98}, symplectic implosion \cite{gui-jef-sja:02,hur-jef-sja:06}, preimages of Poisson transversals \cite{fre-mar:17}, and the Ginzburg--Kazhdan's construction \cite{gin.kaz:23} of the Moore--Tachikawa topological quantum field theory \cite{moo-tac:12}. %, and several other constructions in geometric representation theory.
\end{itemize}
\end{example}

Finally, we will explain a conjectural differentiation/integration process for 1-shifted coisotropics.
It can be seen as a generalization of Cattaneo's correspondence between coisotropic submanifolds in Poisson geometry and Lagrangian subgroupoids \cite{cat:04}.

\begin{conjecture}
There is a version of Lie's third theorem whose global objects are 1-coisotropics and whose infinitesimal objects are pairs of twisted Dirac manifolds $(M, L_M, \phi)$ and $(N, L_N, \psi)$ together with a smooth map $\c : N \to M$ such that $\c^*\phi = \psi$ and
\[
L_N \times_\c L_M \coloneqq \{((v, \alpha), (w, \beta)) \in L_N \times L_M : \c_*v = w \text{ and } \alpha = \c^*\beta\}
\]
has constant rank over $N$.
\end{conjecture}

%(Note that $\c$ is not necessarily forward Dirac nor backward Dirac.)
Several facts support the conjecture; see \S\ref{xlzo381e}.

%***We should cite more papers on symplectic reduction interpreted with Lagrangian intersection, e.g. \cite{cal:15} and some of Safronov's papers.***

\subsection*{Organization of the paper}
%The paper is organized as follows.
In \S\ref{c7u27pg5}, we review basic notions of Lie groupoids and Dirac geometry.
In \S\ref{beux4a3x}, we define 1-coisotropics and provide some examples.
Then \S\ref{o2caey5r} proves that 1-coisotropic correspondences can be composed, obtaining Properties \ref{g7rgs6l8} and \ref{lwnzxn8d} as corollaries.
In \S\ref{eph7yukn}, we obtain a few general theorems of independent interest about transferring Dirac structures through Morita morphisms.
We then review in \S\ref{1vfh5k59} the Morita transfer of quasi-symplectic groupoids, mainly due to Xu \cite{xu:04}.
Then \S\ref{4pvyvpb2} proves Property \ref{xea3qq2x}, i.e.\ 1-coisotropics can be transferred in a similar way.
In \S\ref{r9gv7nax}, we show that this transfer of 1-coisotropics is compatible with the composition of Morita equivalences.
We then show in \S\ref{616chpl6} how to get 1-coisotropics from Hamiltonian actions of quasi-symplectic groupoids and obtain Theorem \ref{nrztzlyn} as a corollary.
Finally, we define 1-shifted coisotropic structures on morphisms of differentiable stacks in \S\ref{dy26su68}.

\subsection*{Acknowledgements}

We thank Henrique Bursztyn for fruitful discussions and, in particular, for suggesting that Hamiltonian actions of quasi-symplectic groupoids on Dirac manifolds are 1-coisotropics.
We are also grateful to Ana B\u{a}libanu and Peter Crooks for helpful conversations, to Miquel Cueca for insightful comments on an earlier draft, and to the referee for their valuable suggestions, which have significantly improved the paper. 
This work was supported by an NSERC Discovery grant.

\section{Conventions and basic facts}% on Lie groupoids and Dirac structures}
\label{c7u27pg5}

%***We could make it shorter; see \cite{fer-mar:23}. That's how you write a (long) paper on Lie groupoids in 2023.***

We summarize basic notions on Lie groupoids and Dirac structures to explain our notation and conventions.
For more on Lie groupoids, see e.g.\ \cite{sil-wei:99,moe-mrc:03,mac:05,cra-fer:11,bur-hoy:23}, and for Dirac structures, see e.g.\ \cite{cou:90,sev-wei:01,bur-cra:05,bur-cra-wei-zhu:04,bur:13}.

\subsection{Lie groupoids}

We denote a Lie groupoid by $\G = (\G\arr \tto \G\obj)$ with source and target maps $\sss, \ttt$, multiplication $\mmm : \G\comp{2} \coloneqq \G\arr \times_{\sss,\ttt} \G\arr \to \G\arr$, unit $\uuu$, and inversion $\iii$.
We consider a Lie groupoid as a category with objects $\G\obj$ and arrows $\G\arr$. %, and composition $\mmm$.
The Lie algebroid of $\G$ is denoted $A_\G \coloneqq \ker \sss_*|_{\G\obj}$ with anchor map $\aaa_\G : A_\G \to T\G\obj : a \mto \ttt_*a$, or simply $\aaa$ if $\G$ is clear from the context.
For $a \in A_\G$, its left and right translations at $g \in \G$ are denoted $a^L_g$ and $a^R_g$, respectively.
We have $\ttt_*(a^R_g) = \aaa(a)$ and $\sss_*(a^L_g) = -\aaa(a)$.
It will be useful to observe that
\begin{equation}\label{4ksoptjs}
\mmm_*(v, a) = v + a^L_g,
\end{equation}
for all $v \in T_g\G$ and $a \in (A_\G)_{\ttt(g)}$ such that $\sss_*v = \aaa a$.
(Indeed, $\sss_*(v + a^L_g) = 0$ and $\iii_*(a^R_{g^{-1}}) = -a^L_g$, so $\mmm_*(v, a) = \mmm_*(v + a^L_g, 0) + \mmm_*(\iii_* a^R_{g^{-1}}, a) = v + a^L_g$, since $\mmm\circ(\iii \circ R_{g^{-1}}, \Id)$ is constant.)

A \defn{morphism of Lie groupoids} is a functor $f : \H \to \G$ whose maps $f\obj : \H\obj \to \G\obj$ and $f\arr : \H\arr \to \G\arr$ are smooth.
We omit the indices $\obj$ and $\arr$ on $f$ if they are clear from the context.
The induced morphism of Lie algebroids is denoted $f_* : A_\H \to A_\G$.
A \defn{natural transformation} % (or \defn{homotopy})
\begin{equation}\label{xb8qzdsb}
\begin{tikzcd}[row sep=0pt]
&\null \arrow[Rightarrow,shorten=2pt]{dd}{\theta} & \\
\H \arrow[bend left]{rr}{f} \arrow[bend right, swap]{rr}{g} & & \G \\
&\null &
\end{tikzcd}
\end{equation}
between $f$ and $g$ is a smooth map $\theta : \H\obj \to \G\arr$ that is a natural transformation from $f$ to $g$ when viewed as functors, i.e.\
\begin{equation}\label{nq5g7v03}
\sss \circ \theta = f, \quad \ttt \circ \theta = g, \quad\text{and}\quad
\theta(\ttt(h)) \cdot f(h) = g(h) \cdot \theta(\sss(h)),
\end{equation}
for all $h \in \H$.
In that case, we say that $f$ and $g$ are \defn{homotopic}.
Vertical composition of natural transformations $\theta : f \Rightarrow g$ and $\eta : g \Rightarrow h$ is denoted $\eta \star \theta$, and horizontal composition with $\circ$ or juxtaposition.
We say that a diagram of Lie groupoid morphisms is \defn{2-commutative} if it commutes up to natural transformations.

A \defn{multiplicative form} on a Lie groupoid $\G$ is a differential form $\omega$ on $\G\arr$ such that $\mmm^*\omega = \pr_1^*\omega + \pr_2^*\omega$, where $\pr_i : \G\comp{2} \to \G$ are the two projections.
The following basic facts will be used repeatedly. % in this paper.

\begin{lemma}\label{wip6t5yh}
Let $\G$ be a Lie groupoid endowed with a multiplicative form $\omega$.
\begin{enumerate}[label=\textup{(\arabic*)}]
\item \label{u0vnf6px}
For a natural transformation $\theta : f \Rightarrow g$ as in \eqref{xb8qzdsb}, we have
\[
g^*\omega - f^*\omega = \ttt^*\theta^*\omega - \sss^*\theta^*\omega.
\]

%\item
%\label{n3zv04a7}
%$\IM_{g^*\omega}(b) - \IM_{f^*\omega}(b) = i_{\aaa(b)} \theta^*\omega$, for all $b \in A_\H$;
%
\item
\label{8096zsbm}
If $\eta : g \Rightarrow h$ is another natural transformation, then $(\eta \star \theta)^*\omega = \eta^*\omega + \theta^*\omega$.
\end{enumerate}
\end{lemma}

\begin{proof}
\ref{u0vnf6px}
%Note: this probably already proved in something on cohomology of stacks, since it is an intermediate step to prove that $f^* = g^* : H^3\G \to H^3\H$.
By \eqref{nq5g7v03}, we have $\mmm(\theta \ttt_* v, f_*v) = \mmm(g_*v, \theta \sss_* v)$ for all $v \in T\H$.
By multiplicativity of $\omega$, this implies that for all $u, v \in T\G$, we have $\omega(\theta \ttt_* u, \theta \ttt_* v) + \omega(f_*u, f_*v) = \omega(g_*u, g_*v) + \omega(\theta \sss_* u, \theta \sss_* v)$.
%\ref{n3zv04a7}
%This follows from \ref{u0vnf6px} using that $\aaa = \ttt_*$ and the definition of $\IM_\omega$.
\ref{8096zsbm}
We have $\eta \star \theta = \mmm \circ (\eta \times \theta)$, so $(\eta \star \beta)^*\omega = (\eta \times \theta)^* \mmm^*\omega = (\eta \times \theta)^* (\pr_1^*\omega + \pr_2^*\omega) = \eta^*\omega + \theta^*\omega$.
\end{proof}

The \defn{restriction} of a groupoid $\G$ to a subset $S \s \G\obj$ is denoted $\G|_S \coloneqq \sss^{-1}(S) \cap \ttt^{-1}(S)$.

For an action of a Lie groupoid $\G$ on smooth manifold $M$ with moment map $\mu : M \to \G\obj$, the \defn{action groupoid} is denoted $\G \ltimes M$, with source map $\sss(g, p) = p$ and target map $\ttt(g, p) = g \cdot p$.
%is a smooth map $\act : \G\arr \ltimes M \coloneqq \{(g, x) \in \G\arr \times M : \sss(g) = \mu(x)\} \to M : (g, x) \mto g \cdot x$ satisfying the compatilibity conditions in \ref{}.
%In that case, the \defn{action groupoid} is the Lie groupoid $\G \ltimes M$ with objects $M$, arrows $\G\arr \ltimes M$, and structure maps $\sss(g, x) = x$, $\ttt = \act$, $\uuu(x) = (\uuu(\mu(x)), x)$, $\iii(g, x) = (g^{-1}, g \cdot x)$, and $\mmm((g, x), (h, y)) = (gh, y)$.

If $f : N \to \G\obj$ is a smooth map such that $\ttt \circ \pr_{\G\arr} : N \times_{f,\sss} \G\arr \to \G\obj$ is a surjective submersion, the \defn{pullback} of $\G$ by $f$, denoted $f^*\G$, is the Lie groupoid with objects $N$, arrows $N \times_{\G\obj} \G\obj \times_{\G\obj} N \coloneqq \{(x, g, y) \in N \times \G\arr \times N : f(x) = \sss(g), f(y) = \ttt(g)\}$, and structure maps $\sss(x, g, y) = x$, $\ttt(x, g, y) = y$, $\uuu(x) = (x, \uuu_{f(x)}, x)$, $\iii(x, g, y) = (y, g^{-1}, x)$, $\mmm((x_1,g_1,y_1),(x_2,g_2,y_2)) = (x_2, g_1g_2, y_1)$.
The condition on $\ttt \circ \pr_{\G\arr}$ holds, in particular, if $f$ is a surjective submersion.

\subsection{Twisted Dirac structures}\label{1n8d3ppc}
Let $M$ be a smooth manifold.
A \defn{Dirac structure} on $M$ is a subbundle $L \s TM \oplus T^*M$ together with a 3-form $\phi$, called the \defn{background 3-form}, such that $L$ is Lagrangian, i.e.\ $L^\perp = L$ with respect to the symmetric pairing $\ip{(v, \alpha), (w, \beta)} = \alpha(w) + \beta(v)$, and involutive, i.e.\ smooth sections of $L$ are preserved under the Dorfman bracket
\begin{equation}\label{qq4huasw}
\llbracket(v, \alpha), (w, \beta)\rrbracket_\phi
\coloneqq
([v, w], i_{v}d\beta + d i_{v} \beta - i_{w}d\alpha + i_{w} i_{v} \phi).
\end{equation}
In that case, $L$ is a Lie algebroid whose anchor map is the projection $p_T : L \to TM$.
We denote the other projection by $p_{T^*} : L \to T^*M$.
The \defn{kernel} of $L$ is $\ker L \coloneqq L \cap TM$.
We say that $L$ is \defn{untwisted} if $\phi = 0$.
We will make frequent use of the following four constructions.

\subsubsection{Graph}
The \defn{graph} of a 2-form $\omega$ on $M$ is the Dirac structure
\[
\Gamma_\omega \coloneqq \{(v, i_v\omega) : v \in TM\}
\]
with background 3-form $-d\omega$.
A Dirac structure $L$ is of this form if and only if $L \cap T^*M = 0$.
In that case, we say that $L$ is \defn{non-degenerate}.

Similarly, if $\pi : T^*M \to TM$ is a Poisson structure, then its graph $\{(\pi(\alpha), \alpha) : \alpha \in T^*M\}$ is an untwisted Dirac structure.
A Dirac structure $L$ is of this form if and only if $\ker L = 0$.

\subsubsection{Sum}\label{e5tv6dsz}
Let $(L_1, \phi_1)$ and $(L_2, \phi_2)$ be Dirac structures on $M$.
The \defn{sum} \cite{gua:18} (or \defn{tensor product} \cite{gua:11}) of $L_1$ and $L_2$ is the set
\[
L_1 + L_2 \coloneqq \{(v, \alpha_1 + \alpha_2) : (v, \alpha_1) \in L_1 \text{ and } (v, \alpha_2) \in L_2\}.
\]
If $L_1 + L_2$ is smooth, it is a Dirac structure with background 3-form $\phi_1 + \phi_2$.
In particular, the \defn{gauge transformation} of a Dirac structure $(L, \phi)$ by a 2-form $\omega$ is the Dirac structure $L + \Gamma_\omega$ with background 3-form $\phi - d\omega$.

\subsubsection{Pullback}\label{e37p3coq}
Let $f : M \to N$ be a smooth map.
The \defn{pullback} of a Dirac structure $(L, \phi)$ on $N$ is the set
\[
f^*L \coloneqq \{(v, f^*\alpha) \in TM \oplus T^*M : (f_*v, \alpha) \in L\}.
\]
If $f^*L$ is smooth, then it is a Dirac structure twisted by $f^*\phi$.
This is the case, in particular, if $f$ is a submersion (see e.g.\ \cite[Proposition 1.10]{bur:13}).

\subsubsection{Pushforward}
Let $f : M \to N$ be a surjective smooth map and $(L, \phi)$ a Dirac structure on $M$.
We say that $L$ is \defn{$f$-invariant} if
\[
(f_*L)_p \coloneqq \{(f_* v, \alpha) \in T_{f(p)}N \oplus T^*_{f(p)}N : (v, f^*\alpha) \in L_p\}
\]
is constant along fibres of $f$.
In that case, the sets $(f_*L)_p$ define a subbundle $f_*L$ of $TN \oplus T^*N$.
If $f_*L$ is smooth and $\phi = f^*\eta$ for some $3$-form $\eta$ on $N$, then $(f_*L, \eta)$ is a Dirac structure on $N$, called the \defn{pushforward} of $(L, \phi)$ by $f$ (see e.g.\ \cite[Proposition 1.13]{bur:13}).

%\begin{remark}\label{f9qlvv5f}
%Note that if $f : M \to N$ is a smooth submersion between Dirac manifolds, then $f_* f^*L_N = L_N$.
%Moreover, $f^* f_* L_M = L_M$ if and only if $\ker f \s \ker L_M$.
%\end{remark}

\subsection{Clean and transverse intersections}\label{fj51c8fw}

Two smooth maps
\begin{equation}\label{7qm8van1}
\begin{tikzcd}
N_1 \arrow{r}{f_1} & M & N_2 \arrow[swap]{l}{f_2}
\end{tikzcd}
\end{equation}
\defn{intersect cleanly} if the fibre product $N_1 \times_M N_2 \coloneqq \{(n_1, n_2) \in N_1 \times N_2 : f_1(n_1) = f_2(n_2)\}$ is a smooth submanifold of $N_1 \times N_2$ with the expected tangent spaces, i.e.\ $T_{(n_1, n_2)}(N_1 \times_M N_2) = T_{n_1}N_1 \times_{T_mM} T_{n_2}N_2$.
They \defn{intersect transversely} if $f_{1*}(T_{p_1}N_1) + f_{2*}(T_{p_2}N_2) = T_{q} M$ for all $(p_1, p_2) \in N_1 \times_M N_2$.

Two vector bundle homomorphisms
\[
\begin{tikzcd}
W_1 \arrow{r}{F_1} & V & W_2 \arrow[swap]{l}{F_2}
\end{tikzcd}
\]
covering \eqref{7qm8van1} \defn{intersect cleanly} if $\im(F_1)_{p_1} + \im(F_2)_{p_2}$ has constant rank for $(p_1, p_2) \in N_1 \times_M N_2$.
We say that they are \defn{transverse} if $\im(F_1)_{p_1} + \im(F_2)_{p_2} = V_q$ for all $(p_1, p_2) \in N_1 \times_M N_2$.
%In either case, the fibre product
%\[
%E_1 \times_F E_2 = \{(v_1, v_2) \in E_1 \times E_2 : f_1(v_1) = f_2(v_2)\}
%\]
%is a smooth subbundle of $E_1 \times E_2$.

%We say that the smooth maps \eqref{7qm8van1} \defn{intersect cleanly} if $N_1 \times_M N_2$ is a smooth submanifold of $N_1 \times N_2$ with the expected tangent spaces, i.e.\
%\[
%T_{(n_1, n_2)}(N_1 \times_M N_2) = T_{n_1}N_1 \times_{T_mM} T_{n_2}N_2.
%\]
%In that case, the homomorphisms
%\begin{equation}\label{b79lvtot}
%\begin{tikzcd}
%TN_1 \arrow{r}{f_{1*}} & TM & TN_2 \arrow[swap]{l}{f_{2*}}
%\end{tikzcd}
%\end{equation}
%intersect cleanly.
%We say that \eqref{7qm8van1} \defn{intersect transversely} if \eqref{b79lvtot} intersect transversely.
%%Transverse intersections are clean.
%

\section{1-shifted coisotropic structures}
\label{beux4a3x}

\subsection{Definition of 1-shifted coisotropic structures}
Our starting point is the notion of quasi-symplectic groupoids introduced by Xu \cite{xu:04} and Bursztyn--Crainic--Weinstein--Zhu \cite{bur-cra-wei-zhu:04}, which will be interpreted as presentations of 1-shifted symplectic stacks \cite{get:14,cal:21,cue-zhu:23} (see \S\ref{dy26su68} for more details).
Recall that a \defn{quasi-symplectic groupoid} \cite{xu:04} (also known as a \emph{twisted presymplectic groupoid} \cite{bur-cra-wei-zhu:04}) is a Lie groupoid $\G$ endowed with a multiplicative $2$-form $\omega$ on $\G\arr$ and a closed $3$-form $\phi$ on $\G\obj$ such that $d\omega = \sss^*\phi - \ttt^*\phi$, $\dim \G\arr = 2 \dim \G\obj$, and $\ker \omega_x \cap \ker \sss_* \cap \ker \ttt_*  = 0$ for all $x \in \G\obj$.
We denote the corresponding infinitesimally multiplicative form by
\[
\IM_\omega : A_\G \too T^*\G\obj, \quad \IM_\omega(a) \coloneqq \uuu^*(i_a\omega).
\]
Quasi-symplectic groupoids are the integrations of (twisted) Dirac structures \cite{bur-cra-wei-zhu:04}.
That is, given a quasi-symplectic groupoid $(\G, \omega, \phi)$, the set $L \coloneqq \im(\aaa, \IM_\omega)$ is a Dirac structure on $\G\obj$ with background 3-form $\phi$.
Conversely, any Dirac structure $(L, \phi)$ whose Lie algebroid $L$ is integrable arises in this way. % for some quasi-symplectic structure.
The following basic facts will be used repeatedly throughout the paper.

\begin{lemma}[{\cite{bur-cra-wei-zhu:04,xu:04}}]\label{7wn1m20c}
For a quasi-symplectic groupoid $(\G, \omega, \phi)$, we have
\begin{enumerate}[label=\textup{(\arabic*)}]
\item
\label{dmfylnua}
$\aaa^* \IM_\omega = -\IM_\omega^*\aaa$,
\item
\label{rwp7208a}
$\ker \aaa \cap \ker \IM_\omega = 0$,
\item
\label{r6sdh90a}
$\ker(\aaa^* + \IM_\omega^*) = \im(\aaa, \IM_\omega)$, and %, i.e.\ for $(v, \alpha) \in T\G\obj \oplus T^*\G\obj$, we have
%\[
%\IM_\omega^*v + \aaa^*\alpha = 0 \iff (v, \alpha) = (\aaa(a), \IM_\omega(a)) \text{ for some }a \in A_\G.
%\]
\item
\label{xul5itsk}
$\sss^*\IM_\omega a = i_{a^L}\omega$ and $\ttt^*\IM_\omega a = i_{a^R}\omega$ for all $a \in A_\G$.
\end{enumerate}
\end{lemma}

\begin{proof}
Part \ref{dmfylnua} is \cite[Proposition 3.5(i)]{bur-cra-wei-zhu:04}.
The map $(\aaa, \IM_\omega) : A_\G \to L \coloneqq \im(\aaa, \IM_\omega)$ is an isomorphism \cite[Corollary 4.8]{bur-cra-wei-zhu:04}.
Then Part \ref{rwp7208a} is injectivity of this map, and Part \ref{r6sdh90a} is surjectivity together with the fact that $L$ is Lagrangian.
To show Part \ref{xul5itsk}, note that for all $v \in T_g\G$, the vectors $(v, \uuu\sss_* v, v)$ and $(0, a - \uuu \ttt_* a, a^L_g)$ are tangent to the graph of the multiplication of $\G$ at $(g, \uuu_{\sss(g)}, g)$.
Hence, $\omega(a - \uuu\ttt_* a, \uuu\sss_* v) = \omega(a^L_g, v)$.
Since $\uuu^*\omega = 0$ \cite[Lemma 3.1]{bur-cra-wei-zhu:04}, this implies that $\omega(a, \uuu\sss_* v) = \omega(a^L_g, v)$, i.e.\ $\sss^*\IM_\omega a = i_{a^L_g}\omega$.
Similarly, $(\uuu\ttt_* v, v, v)$ and $(a, 0, a^R_g)$ are tangent to the graph at $(\uuu_{\ttt(g)}, g, g)$ so $\omega(a, \uuu\ttt_* v) = \omega(a^R_g, v)$.%, i.e.\ $\ttt^* \IM_\omega a = i_{a^R_g}\omega$.
%For the latter, note that $\ker \aaa^* = (\im \aaa)^\circ = L^\perp \cap T^*M = L \cap T^*M = \IM_\omega(\ker \aaa)$, where the last step follows from surjectivity of $(\aaa, \IM_\omega) : A_\G \to L$.
\end{proof}

By thinking of quasi-symplectic groupoids as presentations of 1-shifted symplectic stacks as in \cite{get:14,cal:21,cue-zhu:23}, the original paper of Pantev--To\"en--Vaqui\'e--Vezzosi \cite{ptvv:13} provides a definition of 1-shifted Lagrangians on $(\G, \omega, \phi)$.
Namely, a \defn{1-shifted Lagrangian structure} on a Lie groupoid morphism $\c : \C \to \G$ is a 2-form $\gamma$ on $\C\obj$ satisfying the following two conditions.
\begin{itemize}
\item
\defn{Compatibility.} We have
\[
\c^*\phi = -d\gamma \quad\text{and}\quad \c^*\omega = \ttt^*\gamma - \sss^*\gamma.
\]
\item
\defn{Non-degeneracy.}
The chain map
\begin{equation}\label{fkoguq0d}
\begin{tikzcd}[column sep=50pt]
A_\C \arrow{d} \arrow{r}{(\aaa_\C, \c_*)} & T\C\obj \oplus \c^*A_\G \arrow{r}{\c_* - \aaa_\G} \arrow{d}{-\gamma + c^*\IM_\omega} & \c^*T\G\obj \arrow{d}{\c^*\IM_\omega^*} \\
0 \arrow{r} & T^*\C\obj \arrow{r}{\aaa_\C^*} & A_\C^*
\end{tikzcd}
\end{equation}
is a quasi-isomorphism.
\end{itemize}
%Here $\c^*\IM_\omega$ denotes the composition $\c^*A_\G \overset{\IM_\omega}{\to} \c^*T^*\G\obj \overset{\c^*}{\to} T^*\C^0$ and $\c^* \IM_\omega^*$ the dual of $A_\C \overset{\c_*}{\to} \c^*A_\G \overset{\IM_\omega}{\to} \c^*T^*\G\obj$.
%(***Explain the general principle behind \eqref{fkoguq0d} in terms of mapping cones and the likes.***

%\subsection{1-shifted coisotropics}
The notion of 1-shifted coisotropics will generalize this by replacing the 2-form $\gamma$ on $\C\obj$ with a Dirac structure $L$ suitably compatible with $(\omega, \phi)$, and modifying the chain map \eqref{fkoguq0d} appropriately with $L$ in place of $T\C\obj$.
Let us first construct the chain map:

\begin{lemma}\label{18e8u16h}
Let $\c : \C \to \G$ be a morphism of Lie groupoids, $\omega$ a multiplicative 2-form on $\G$, and $L$ a Dirac structure on $\C\obj$ such that
\[
\ttt^*L = \sss^*L + \Gamma_{\c^*\omega}.
\]
%(i.e.\ the gauge transformation (\S\ref{e5tv6dsz}) of the pullback $\sss^*L$ (\S\ref{e37p3coq}) by $\c^*\omega$ is $\ttt^*L$).
Then the image of the map
\begin{equation}\label{013fffem}
(\aaa_\C, \c^*\IM_\omega \c_*) : A_\C \too T\C\obj \oplus T^*\C\obj
\end{equation}
is contained in $L$.
Moreover, the diagram of vector bundle homomorphisms over $\C\obj$ given by
\begin{equation}\label{r7ynptgu}
\begin{tikzcd}[column sep=50pt]
A_\C \arrow{d} \arrow{r}{(\aaa_\C, \c^* \IM_\omega \c_*, \c_*)} & L \oplus \c^*A_\G \arrow{r}{\c_* p_T - \aaa_\G} \arrow{d}{-p_{T^*} + \c^*\IM_\omega} & \c^*T\G\obj \arrow{d}{\c^*\IM_\omega^*} \\
0 \arrow{r} & T^*\C\obj \arrow{r}{\aaa_\C^*} & A_\C^*
\end{tikzcd}
\end{equation}
is a chain map.
\end{lemma}

\begin{proof}
Since $L^\perp = L$, to show that \eqref{013fffem} has its image in $L$, it suffices to show that
\begin{equation}\label{4dohrg5v}
\ip{\alpha, \aaa_\C b} + \ip{\IM_\omega \c_* b, \c_* v} = 0, \quad \text{for all $b \in A_\C$ and $(v, \alpha) \in L$.}
\end{equation}
We have $(\uuu_*v, \ttt^*\alpha) \in \ttt^*L = \sss^*L + \Gamma_{\c^*\omega}$, so $(\uuu_*v, \ttt^*\alpha) = (\uuu_*v, \sss^*\beta + i_{\uuu_*v}\c^*\omega)$ for some $\beta \in T^*\C\obj$ such that $(v, \beta) \in L$.
Then $\ip{\IM_\omega \c_*b, \c_*v} = \omega(\c_*b, 1 \c_*v) = (\c^*\omega)(b, \uuu_*v) = (\sss^*\beta - \ttt^*\alpha)(b) = -\alpha(\aaa_\C b)$, which proves \eqref{4dohrg5v}.
The commutativity of the first square of \eqref{r7ynptgu} is automatic, and the commutativity of the second square follows from \eqref{4dohrg5v} and Lemma \ref{7wn1m20c}\ref{dmfylnua}.
\end{proof}

The following observation greatly simplifies the definition of 1-Lagrangians and 1-coisotropics.

\begin{lemma}\label{t7rjqxe5}
The chain map \eqref{r7ynptgu} is a quasi-isomorphism if and only if the map
\begin{equation}\label{lajk1793}
(\aaa_\C, \c^* \IM_\omega \c_*, \c_*) : A_\C \too L \times_\c A_\G
\end{equation}
is an isomorphism of vector bundles over $\C\obj$, where
\[
L \times_\c A_\G \coloneqq \{((v, \alpha), a) \in L \oplus \c^*A_\G : \c_* v = \aaa_\G a \text{ and } \alpha = \c^* \IM_\omega a\}.
\]
More generally, \eqref{r7ynptgu} induces an isomorphism on the middle cohomology groups if and only if \eqref{lajk1793} is surjective.
\end{lemma}

\begin{proof}
Denote the three maps on cohomology groups induced by \eqref{r7ynptgu} by
\[
\begin{tikzcd}[column sep=1em]
\ker(\aaa_\C, \c^* \IM_\omega \c_*, \c_*) \arrow{d}{\varphi_1}
& \ker(\c_* p_T - \aaa_\G) / \im(\aaa_\C, \c^* \IM_\omega \c_*, \c_*) \arrow{d}{\varphi_2}
& \c^*T\G\obj / \im(\c_* p_T - \aaa_\G) \arrow{d}{\varphi_3} \\
0 & \ker \aaa_\C^* & A_\C^* / \im \aaa_\C^*.
\end{tikzcd}
\]
Injectivity of \eqref{lajk1793} is equivalent to $\varphi_1$ being an isomorphism, and surjectivity of \eqref{lajk1793} is equivalent to $\varphi_2$ being injective.
Hence, it remains to show that if \eqref{lajk1793} is surjective, then so is $\varphi_2$, and if \eqref{lajk1793} is an isomorphism, then so is $\varphi_3$.

For surjectivity of $\varphi_2$, it suffices to show that $\ker \aaa_\C^* \s \im(-p_{T^*} + \c^*\IM_\omega : L \times_{T\G\obj} \c^*A_\G \to T^*\C\obj)$.
Taking the annihilator on both sides, this reduces to $\ker((-p_{T^*} + \c^*\IM_\omega)^*) \s \im \aaa_\C$, i.e.\ we need to show that for all $v \in T\C\obj$ such that
\begin{equation}\label{bj0fnjk8}
-p_{T^*}^* v = p_T^* \c^*\alpha
\quad\text{and}\quad
\IM_\omega^*\c_* v = -\aaa_{\G}^* \alpha
\end{equation}
for some $\alpha \in T^*\G\obj$, then $v \in \im \aaa_\C$.
The first part of \eqref{bj0fnjk8} implies that $(v, \c^*\alpha) \in L^\perp = L$.
By Lemma \ref{7wn1m20c}\ref{r6sdh90a}, the second part of \eqref{bj0fnjk8} implies that $(\c_* v, \alpha) = (\aaa_\G a, \IM_\omega a)$ for some $a \in A_\G$.
Then $((v, \c^*\alpha), a) \in L \times_\c A_\G$, so $v \in \im \aaa_\C$ by surjectivity of \eqref{lajk1793}.

To show that $\varphi_3$ is an isomorphism, consider its dual
\begin{equation}\label{sr630ygf}
\ker \aaa_\C \too \ker((\c_* p_T - \aaa_\G)^*) = \ker p_T^* \c^* \cap \ker \aaa_\G^*, \quad b \mtoo \IM_\omega \c_* b,
\end{equation}
where $p_T^*\c^* : \c^*T^*\G\obj \to T^*\C\obj \to L^*$.
To show surjectivity of \eqref{sr630ygf}, let $a \in \ker p_T^*\c^* \cap \ker \aaa_\G^*$.
Since $\aaa_\G^*\alpha = 0$, Lemma \ref{7wn1m20c}\ref{r6sdh90a} implies that $\alpha = \IM_\omega a$ for some $a \in A_\G$ such that $\aaa_\G a = 0$.
It follows that $((0, \c^*\alpha), a) \in L \times_\c A_\G$, and hence $((0, \c^*\alpha), a) = ((\aaa_\C b, \c^*\IM_\omega \c_* b), \c_* b)$ for some $b \in A_\C$.
In particular, $b \in \ker \aaa_\C$ and $\IM_\omega \c_* b = \IM_\omega a = \alpha$.
For injectivity of \eqref{sr630ygf}, note that if $\aaa_\C b = 0$ and $\IM_\omega \c_* b = 0$, then also $\aaa_\G \c_* b = \c_* \aaa_\C b = 0$.
It follows that $\c_* b = 0$ (Lemma \ref{7wn1m20c}\ref{rwp7208a}) so $b = 0$ by injectivity of \eqref{lajk1793}.
\end{proof}

We are now ready to define the main object of study of this paper.

\begin{definition}\label{gm52z93m}
Let $(\G, \omega, \phi)$ be a quasi-symplectic groupoid, $\C$ a Lie groupoid, and $\c : \C \to \G$ a morphism of Lie groupoids.
A \defn{1-shifted coisotropic structure} on $\c$ is a Dirac structure $(L, \eta)$ on $\C\obj$ satisfying the following two conditions.
\begin{enumerate}[label={(\arabic*)}]
\item\label{0k2j1uwv}
\defn{Compatibility.}
We have $\eta = \c^*\phi$ and $\ttt^*L = \sss^*L + \Gamma_{\c^*\omega}$.
\item\label{3bw7gcnj}
\defn{Non-degeneracy.}
%The chain map \eqref{r7ynptgu} is a quasi-isomorphism, or equivalently (by Lemma \ref{t7rjqxe5}), 
The map
\begin{equation}\label{0v1fz7h6}
(\aaa_\C, \c^* \IM_\omega \c_*, \c_*) : A_\C \too L \times_\c A_\G
\end{equation}
is surjective, where $L \times_\c A_\G \coloneqq \{((v, \alpha), a) \in L \oplus \c^*A_\G : \c_* v = \aaa a \text{ and } \alpha = \c^* \IM_\omega a\}$.
\end{enumerate}
Since the background 3-form $\eta$ is determined by $\phi$, we denote a 1-shifted coisotropic structure simply by $L$.
The triple $(\C, \c, L)$ is called a \defn{1-coisotropic} on $\G$.
A 1-shifted coisotropic structure is \defn{strong} if \eqref{0v1fz7h6} is also injective, i.e.\ $\ker \rho_\C \cap \ker \c_* = 0$, or equivalently (by Lemma \ref{t7rjqxe5}), the chain map \eqref{r7ynptgu} is a quasi-isomorphism.
%A \defn{1-shifted Lagrangian structure} is a strong 1-shifted coisotropic structure given by the graph of a 2-form.
\end{definition}

\begin{remark}
Due to the properties of 1-shifted coisotropics shown in this paper, we expect that this definition is equivalent to the corresponding derived-algebraic notion \cite{cptvv,mel-saf:18a,mel-saf:18b,hau-mel-saf:22} %,saf:17a}
when $\c : \C \to \G$ is a morphism of algebraic groupoids.
\end{remark}

%It will be useful in our discussion on reduction to weaken this definition in the following way.
%
%\begin{definition}
%A \defn{weak 1-shifted coisotropic structure} on $\c : \C \to (\G, \omega, \phi)$ is a Dirac structure on $\C\obj$ satisfying the compatibility condition \ref{0k2j1uwv} but such that \eqref{0v1fz7h6} is only surjective.
%A weak 1-shifted coisotropic structure is \defn{strong} if it satisfies the full non-degeneracy condition, i.e.\ $\ker \rho_\C \cap \ker \c_* = 0$.
%We sometimes say that a 1-shifted coisotropic structures is \defn{strong} to emphasis that it is not weak.
%%Similarly, a \defn{weak 1-shifted Lagrangian structure} is a weak 1-shifted coisotropic structure of the form $L = \Gamma_\gamma$.
%\end{definition}
%
%\begin{remark}
%As can be seen from the proof of Lemma \ref{t7rjqxe5}, a weak 1-shifted coisotropic structure can also be described as a compatible Dirac structure such that the chain map \eqref{r7ynptgu} induces an isomorphism only on the middle cohomology groups.
%\end{remark}
%

\subsection{Examples of 1-coisotropics}

Our first example is a shifted version of the fact that for an ordinary symplectic manifold, the whole space is a coisotropic submanifold.

\begin{proposition}\label{2kzh0or3}
Let $(\G, \omega, \phi)$ be a quasi-symplectic groupoid and $L$ the induced Dirac structure on $\G\obj$.
Then $L$ is a 1-shifted coisotropic structure on the identity morphism $\G \to \G$.
\end{proposition}

\begin{proof}
To show that
%\begin{equation}\label{l5pa6a8h}
$\ttt^*L = \sss^*L + \Gamma_\omega$,
%\end{equation}
let $(v, \ttt^*\alpha) \in \ttt^*L$, where $(\ttt_* v, \alpha) \in L$.
We claim that $\ttt^*\alpha - i_v\omega \in \im \sss^* = (\ker \sss_*)^\circ$.
Let $w \in \ker \sss_*$ so that $w = a^R_g$ for some $a \in A_\G$ and $g \in \G$.
%Since $(a, 0, a^R_g)$ and $(\uuu\ttt_* v, v, v)$ are tangent to the graph of multiplication at $(\uuu_{\ttt(g)}, g, g)$, we have $\omega(a, \uuu\ttt_* v) = \omega(a^R_g, v)$.
By Lemma \ref{7wn1m20c}\ref{xul5itsk} and the isotropy of $L$, we have $(\ttt^*\alpha - i_v\omega)(w) = \ip{\alpha, \aaa a} + \ip{\IM_\omega a, \ttt_* v} = 0$.
It follows that $\ttt^*\alpha - i_v\omega = \sss^*\beta$ for some $\beta \in T^*\G\obj$, so $(v, \ttt^*\alpha) = (v, \sss^*\beta + i_v\omega)$.
It then suffices to show that $(\sss_* v, \beta) \in L = L^\perp$.
Note that every element of $L$ can be written $(-\sss_* b, \IM_\omega b)$ for some $b \in \ker \ttt_*$.
Since $\sss_*(v - 1\sss_* v) = 0$ and $\ker \ttt_*$ and $\ker \sss_*$ are $\omega$-orthogonal \cite[Lemma 3.1(ii)]{bur-cra-wei-zhu:04}, we have $\omega(b, \uuu\sss_* v) = \omega(b, v)$.
It follows that $\ip{(\sss_* v, \beta), (-\sss_* b, \IM_\omega b)} = -\sss^*\beta(b) + \omega(b, \uuu \sss_* v) = (i_v\omega - \ttt^*\alpha)(b) + \omega(b, v) = 0$, and hence $(\sss_* v, \beta) \in L$.
Finally, \eqref{0v1fz7h6} is surjective since every $((v, \alpha), a) \in L \times_{\Id} A_\G$ is the image of $a \in A_\G$.
%It is also clearly injective.
\end{proof}

At the other extreme, 1-shifted Lagrangian structures are 1-shifted coisotropic structures.
Indeed, by Lemma \ref{t7rjqxe5}, they are precisely the strong 1-shifted coisotropic structures given by the graph of a 2-form.
In particular, the following example will be useful in our discussion on Hamiltonian reduction (\S\ref{616chpl6}).

\begin{proposition}\label{3pl9nrs8}
Let $(\G, \omega, \phi)$ be a quasi-symplectic groupoid and $\O \s \G\obj$ a $\G$-orbit.
Let $\gamma$ be the canonical presymplectic form on $\O$, i.e.\ $\gamma(\aaa a, \aaa b) = \omega(a, b)$ for all $a, b \in A_\G|_\O$ \cite{bur-cra:05}.
Then $\gamma$ is a 1-shifted Lagrangian structure on the inclusion $\c : \G|_\O \hookrightarrow \G$.
\end{proposition}

\begin{proof}
We first show that $\c^*\phi = -d\gamma$.
Let $L$ be the induced Dirac structure on $\G\obj$ so that $\gamma(u, v) = \alpha(v) = -\beta(u)$ for all $(u, \alpha), (v, \beta) \in L|_\O$.
Then, for all local sections $(v_i, \alpha_i)$ of $L$, we have
\begin{align*}
d\gamma(v_1, v_2, v_3) &= v_1 \gamma(v_2, v_3) + \gamma(v_1, [v_2, v_3]) + c.p. \\
&= v_1 \alpha_2(v_3) + \alpha_1([v_2, v_3]) + c.p.,
\end{align*}
where $c.p.$ denotes cyclic permutations of the first terms.
Since sections of $L$ are closed under the Dorfman bracket \eqref{qq4huasw} and $L$ is Lagrangian, we have
\begin{align*}
0 &= \ip{([v_1, v_2], i_{v_1}d\alpha_2 + d i_{v_1} \alpha_2 - i_{v_2}d\alpha_1 + i_{v_2} i_{v_1} \phi), (v_3, \alpha_3)} \\
&= \alpha_3([v_1, v_2]) + d\alpha_2(v_1, v_3) + v_3 \alpha_2(v_1) - d\alpha_1(v_2, v_3) + \phi(v_1, v_2, v_3).
\end{align*}
Note also that $d\alpha_2(v_1, v_3) = v_1 \alpha_2(v_3) - v_3 \alpha_2(v_1) - \alpha_2([v_1, v_3])$ and
\begin{align*}
d\alpha_1(v_2, v_3) &= v_2 \alpha_1(v_3) - v_3 \alpha_1(v_2) - \alpha_1([v_2, v_3]) \\
&= -v_2 \alpha_3(v_1) - v_3 \alpha_1(v_2) - \alpha_1([v_2, v_3]),
\end{align*}
since $\ip{(v_1, \alpha_1), (v_3, \alpha_3)} = 0$.
All together, this gives
\begin{align*}
0 &= \alpha_3([v_1, v_2]) + v_1 \alpha_2(v_3) - \alpha_2([v_1, v_3]) + v_2\alpha_3(v_1) + v_3 \alpha_1(v_2) + \alpha_1([v_2, v_3]) + \phi(v_1, v_2, v_3) \\
&= d\gamma(v_1, v_2, v_3) + \phi(v_1, v_2, v_3),
\end{align*}
as desired.
We now show that $\c^*\omega = \ttt^*\gamma - \sss^*\gamma$.
Let $u, v \in T\G|_\O$.
Then $\sss_* u = \ttt_* a$ and $\sss_* v = \ttt_* b$ for some $a, b \in A_\G|_\O$.
By multiplicativity of $\omega$, we have
\begin{equation}\label{erwyepcs}
\omega(u, v) + \omega(a, b) = \omega(\mmm_*(u, a), \mmm_*(v, b)).
\end{equation}
The second term can be written $\omega(a, b) = \gamma(\ttt_* a, \ttt_* b) = \sss^*\gamma(u, v)$.
For the right-hand side, first note that for all $w_1, w_2 \in \ker \sss_*|_{\G|_\O}$ we have $w_i = (z_i)^R_g$ for some $z_i \in A_\G|_\O$, so $\omega(w_1, w_2) = \omega(z_1, z_2) = \gamma(\aaa z_1, \aaa z_2) = \gamma(\ttt_* w_1, \ttt_* w_2)$.
In particular, $\omega(\mmm_*(u, a), \mmm_*(v, b)) = \gamma(\ttt_* \mmm_*(u, a), \ttt_* \mmm_*(v, b)) = \ttt^*\gamma(u, v)$.
By \eqref{erwyepcs}, we have $\c^*\omega + \sss^*\gamma = \ttt^*\gamma$, proving the compatibility condition.
The non-degeneracy condition amounts to the bijectivity of the map
\[
A_\G|_\O \too \{(v, a) \in T\O \times A_\G|_\O : v = \aaa a,  i_v\gamma = \c^*\IM_\omega a\}, \quad
a \mtoo (\aaa a, a).
\]
Let $(v, a)$ be in the codomain.
Then for all $b \in A_\G|_\O$ we have $(i_{\aaa a}\gamma)(\aaa b) = \omega(a, b) = \omega(a, \uuu_* \aaa b) = \IM_\omega(\aaa b)$, so $i_{\aaa a} \gamma = \c^*\IM_\omega a$, i.e.\ $(v, a)$ is the image of $a \in A_\G|_\O$.
\end{proof}

Other useful examples are as follows.

\begin{example}\label{s2rkhg7b}
Let $\G$ be a symplectic groupoid \cite{wei:87,cos-daz-wei:87} (i.e.\ a quasi-symplectic groupoid with $\phi = 0$ and $\omega$ symplectic) integrating the Poisson manifold $\G\obj$ and let $\C\obj \s \G\obj$ be a coisotropic submanifold.
Then $A_\G \cong T^*\G\obj$ and the conormal bundle $(T\C\obj)^\circ$ is a Lie subalgebroid integrating to a Lagrangian subgroupoid $\C \to \G$ \cite{cat:04}.
It follows that the trivial $2$-form is a 1-shifted Lagrangian structure on $\C \to \G$.
Indeed, the compatibility condition is automatic and the non-degeneracy condition is the bijectivity of the map $(T\C\obj)^\circ \to \{(v, \xi) \in T\C\obj \times T^*\G\obj : v = \pi(\xi) \text{ and } \xi|_{T\C\obj} = 0\}$, $\xi \mto (\pi(\xi), \xi)$, where $\pi$ is the Poisson bivector field.
\end{example}

\begin{example}\label{li2xxdqm}
Generalizing the previous example, let $\G$ be a symplectic groupoid and $\C\obj \s \G\obj$ a pre-Poisson submanifold \cite{cat-zam:07,cat-zam:09}, i.e.\ a submanifold such that $A_\C \coloneqq \pi^{-1}(T\C\obj) \cap (T\C\obj)^\circ$ has constant rank (see also \cite[\S2.2]{cro-may:22}).
By \cite[Proposition 7.2]{cat-zam:09}, $A_\C$ is a Lie subalgebroid of $A_\G \cong T^*\G\obj$ integrating to an isotropic subgroupoid $\C \to \G$.
It follows as in Example \ref{s2rkhg7b} that the trivial 2-form on $\C\obj$ is a 1-shifted Lagrangian structure on $\C \to \G$ (see also \cite[Proposition 6.1]{cro-may:22}).
These 1-Lagrangians were studied in \cite{cro-may:22} under the name \emph{stabilizer subgroupoids}, where they generalize reduction levels in symplectic reduction; see \S\ref{616chpl6}.
\end{example}

%\begin{example}
%Let $\G$ be a quasi-symplectic groupoid.
%
%By \cite{bal-may:22}, a 1-shifted Lagrangian structure on a subgroupoid $\c : \C \hookrightarrow \G$ is precisely a 2-form $\gamma$ on $\C\obj$ such that $\c^*\phi = -d\gamma$, $\c^*\omega = \ttt^*\omega - \sss^*\omega$, and $A_\C \to \{(v, a) \in T\C\obj \times A_\G$
%Let $\G$ be a symplectic groupoid (i.e.\ a quasi-symplectic groupoid with $\phi = 0$ and $\omega$ symplectic) and $S \s \G\obj$ a submanifold.
%\end{example}

Another important source of 1-coisotropics comes from Hamiltonian actions of quasi-symplectic groupoids on Dirac manifolds, but this will be treated in \S\ref{616chpl6}.

%\begin{remark}
%The purpose of the weak version is that clean intersections will give weak 0-shifted Poisson structures, and hence ordinary Poisson structures on the quotients.
%\end{remark}

\subsection{0-shifted Poisson structures}
\label{zaduohbe}
One of the \emph{raisons d'\^etre} of 1-shifted coisotropics is to intersect them to produce 0-shifted Poisson structures.
We thus need the following.
%Shifted Poisson structures have not yet (to the authors knowledge) found a general definition in the language of Lie groupoids.
%Following the differential-geometric notion of 0-shifted symplectic structures (see e.g. \cite[\S5.5]{hof-sja:21}) % or \cite[Example 2.20]{cue-zhu:23})
%it is natural to define $0$-shifted Poisson structures as follows.
%make the following definition.

\begin{definition}\label{3thr311z}
A \defn{$0$-shifted Poisson structure} on a Lie groupoid $\G$ is an untwisted Dirac structure $L$ on $\G\obj$ such that $\sss^*L = \ttt^*L$ and $\im \aaa = \ker L$.
%\cite[Definition 4.1]{mag-tor-vit:23} for the symplectic case).
\end{definition}

Note that the first condition $\sss^*L = \ttt^*L$ implies that $\im \aaa \s \ker L$, so the second condition is a non-degeneracy one.

For example, a $0$-shifted Poisson structure given by the graph of a 2-form is a $0$-symplectic form as in \cite[\S5.5]{hof-sja:21}.
%This definition is natural following the litterature on shifted Poisson and shifted symplectic structures.
As we will see later (Corollary \ref{s7wgomy8}), if the orbit space $\G\obj/\G\arr$ is a manifold, then 0-shifted Poisson structures on $\G \tto M$ are in one-to-one correspondence with Poisson structures on $\G\obj / \G\arr$.
We thus think of 0-shifted Poisson structures as Poisson structures on the corresponding differentiable stack.
The following is immediate from the definitions (cf.\ \cite[Example 2.3]{cal:15} for a similar result in derived algebraic geometry).
%However, for the differential geometer whose main purpose is to obtain Poisson structures by passing to the quotient, this definition can be relaxed:
%
%\begin{definition}
%If all of the above hold except that $\ker \aaa$ is not necessarily trivial, we say that it is a \defn{weak $0$-shifted Poisson structure}.
%A (weak) $0$-shifted Poisson structure is \defn{symplectic} if $L$ is non-degenerate.
%\end{definition}

%\begin{remark}
%In other words, a 0-shifted symplectic structure is ...
%See also \cite[Example 2.20]{cue-zhu:23} and \cite{hof-sja:21} where it is called a 0-symplectic groupoid.
%\end{remark}

\begin{proposition}\label{q2gecljb}
Let $\G$ be a Lie groupoid.
A 1-shifted coisotropic structure on $\G \to \{*\}$ is the same as a 0-shifted Poisson structure on $\G$.
%Similarly, a weak 1-shifted coisotropic structure on $\G \to \{*\}$ is the same as a weak 0-shifted Poisson structure on $\G$.
\qed
\end{proposition}

\begin{remark}
As explained by Pridham in \cite[Remark 2.6]{pri:20}, a $0$-shifted Poisson structure in the sense of \cite{pri:17} can be viewed as a $P_\infty$ algebras in the sense of Cattaneo--Felder \cite{cat-fel:07}.
We should therefore expect a relationship between Definition \ref{3thr311z} and $P_\infty$ algebras.
We leave this for future work.
\end{remark}

\subsection{Differentiation and integration for 1-coisotropics}
\label{xlzo381e}

We now explain a conjectural version of Lie's third theorem for 1-shifted coisotropic structures.
Let $(\G, \omega, \phi)$ be a quasi-symplectic groupoid integrating a Dirac structure $L_M$ on $M \coloneqq \G\obj$.
Suppose that $\C \to \G$ is a strong 1-coisotropic over $N \coloneqq \C\obj$ with Dirac structure $L_N$.
Since \eqref{0v1fz7h6} is an isomorphism, the Lie algebroid $A_\C$ of $\C$ is entirely determined by the map $N \to M$ and the Dirac structures $L_N$ and $L_M$.
Hence, the infinitesimal version of a 1-coisotropic is as follows.

\begin{definition}\label{3v9jbw09}
An \defn{infinitesimal 1-coisotropic} is a pair of Dirac manifolds $(M, L_M, \phi)$ and $(N, L_N, \psi)$ together with a smooth map $\c : N \to M$ such that $\c^*\phi = \psi$ and
\begin{equation}\label{x75vblco}
L_N \times_\c L_M \coloneqq \{((v, \alpha), (w, \beta)) \in L_N \times L_M : \c_*v = w \text{ and } \alpha = \c^*\beta\}
\end{equation}
has constant rank over $N$.
\end{definition}

It follows that $L_N \times_\c L_M$ is a Lie algebroid and the projection $L_N \times_\c L_M \to L_M$ is a morphism of Lie algebroids.
%A first step towards the reverse process of integration is the following observation.
%Note that if $(N, L_N, \eta)$ is the base of a strong 1-coisotropic $\C \to \G$, then $L_N \times_\c L_M$ is the right-hand side of \eqref{0v1fz7h6} and hence is a Lie algebroid isomorphic to $A_\C$.
%More generally:
%\begin{proposition}
%If \eqref{x75vblco} has constant rank, then it is a Lie algebroid over $N$ and the projection $L_N \times_\c L_M \to L_M$ is a morphism of Lie algebroids.
%\end{proposition}
%
%\begin{proof}
%Being the kernel of a vector bundle homomorphism of constant rank, $L_N \times_\c L_M$ is a smooth vector bundle.
%It naturally embeds in the Lie algebroid $L_N \oplus (TN \times_{TM} L_M)$, where $TN \times_{TM} L_M$ is the Lie algebroid pullback of $L_M$ on $N$.
%One readily sees that $L_N \times_\c L_M$ is closed under the Lie bracket of $L_N \oplus (TN \times_{TM} L_M)$ and hence forms a Lie subalgebroid.
%The composition $L_N \times_\c L_M \to L_N \oplus (TN \times_{TM} L_M) \to L_M$ is therefore a morphism of Lie algebroids.
%\end{proof}
%
Hence, if $L_M$ and $L_N \times_\c L_M$ are integrable, there is a source-simply-connected quasi-symplectic groupoid $\G$ integrating $(M, L_M, \phi)$, a source-simply-connected Lie groupoid $\C$ integrating $L_N \times_\c L_M$, and a Lie groupoid morphism $\C \to \G$.
It is then natural to expect that $L_N$ is a 1-shifted coisotropic structure on $\C \to \G$.
This amounts to the compatibility condition $\ttt^*L_N = \sss^*L_N + \Gamma_{\c^*\omega}$.

For example, this holds if $M$ is a Poisson manifold and $N \s M$ is a submanifold with a trivial Dirac structure (Example \ref{li2xxdqm}).
In particular, a special case is given by the correspondence between coisotropic submanifolds and Lagrangian subgroupoids \cite{cat:04}.

Another special case is when $\c$ is a \defn{strong Dirac map} \cite{ale-bur-mei:09} (also known as a \emph{Dirac realization} \cite{bur-cra:05}), i.e.\ $\c_* L_N = L_M$ and $\ker \c_* \cap \ker L_N = 0$.
In that case, $L_N \times_\c L_M$ is the vector bundle pullback $\c^*L_M$ with anchor map given by the action of $L_M$ on $N$ \cite[Corollary 3.12]{bur-cra:05}.
If $(M, L_M, \phi)$ integrates to a source-simply-connected quasi-symplectic groupoid $\G$, this action integrates to an action of $\G$ on $N$ \cite[Theorem 4.7]{bur-cra:05}.
It follows that the integration of $L_N \times_\c L_M \to L_M$ is the projection $\G \ltimes N \to \G$, and $L_N$ is indeed a 1-shifted coisotropic structure by \S\ref{616chpl6}.

Similarly, it holds if $\c$ is a submersion and is backward Dirac, i.e.\ $L_N = \c^*L_M$.
Indeed, if $(M, L_M, \phi)$ integrates to $\G$, then $L_N \times_\c L_M \to L_M$ integrates to the pullback $\c^*\G \to \G$.
Then $L_N = \c^*L_M$ is a 1-shifted coisotropic structure on $\c^*\G \to \G$ by Proposition \ref{2kzh0or3} and Theorem \ref{vazibeld}.

An indication for the general case comes from \cite{bal-may:22}, where a corollary of the conjecture is proven directly.
In that paper, an infinitesimal version of the Dirac reduction in Theorem \ref{nrztzlyn} is developed, where the Hamiltonian action is replaced by a strong Dirac map and the 1-coisotropic by an infinitesimal 1-coisotropic $\c : (N, L_N, \psi) \to (M, L_M, \phi)$ such that $N \s M$ and $L_N$ is non-degenerate.
Then \cite{bal-may:22} together with \cite[Theorem 4.7]{bur-cra:05} show that Theorem \ref{nrztzlyn} holds in the special case where $\C \to \G$ is the integration of $L_N \times_\c L_M \to L_M$ as above, without \emph{a priori} assuming that $\C \to \G$ is 1-coisotropic.
More generally, one can show with the methods of this paper that there is an infinitesimal version of the results on intersection of 1-coisotropics in \S\ref{o2caey5r} using Definition \ref{3v9jbw09}.
%Since this theorem is a corollary of the intersection theory of 1-coisotropics, it is natural to expect that $\C \to \G$ was 1-coisotropic to begin with.

\begin{remark}
As observed by Pym--Safronov \cite{pym-saf:20}, the infinitesimal version of a 1-shifted symplectic structure is a Dirac structure in an exact Courant algebroid. 
In the $C^\infty$ context, exact Courant algebroids are equivalent to Courant algebroids of the form $TM \oplus T^*M$ with the Dorfman bracket twisted by a closed 3-form $\phi$ \cite{sev:99}, and Dirac structures correspond to $\phi$-twisted Dirac structures.
%Definition~\ref{3v9jbw09} may thus be viewed as a proposal for the infinitesimal version of 1-shifted coisotropic structures in the $C^\infty$ case.
When $c \colon N \hookrightarrow M$ is a closed submanifold and the 1-shifted coisotropic structure is Lagrangian (i.e.\ $L_N$ is the graph of a 2-form), then Definition~\ref{3v9jbw09} coincides with the notion of 1-shifted Lagrangian morphisms considered by Pym--Safronov in \cite[Theorem~7.4]{pym-saf:20}.
Indeed, identifying the exact Courant algebroid with $E = TM \oplus T^*M$, we can view the Dirac structure as $L_M \s E$, the generalized tangent bundle as $F = TN \oplus TN^\circ$, and we have a canonical isomorphism $c^!E \cong TN \oplus T^*N$.
A compatible Courant trivialization in the sense of \cite[Definition 7.3]{pym-saf:20} (see also \cite[Definition 6.1]{gua:11}) then corresponds to a $B$-field transformation \cite{gua:11,sev-wei:01}, i.e.\ a 2-form $\gamma$ on $N$ acting by $(v, \alpha) \mto (v, \alpha + i_v\gamma)$, such that $c^*\phi + d\gamma = 0$ and $\{(v, \alpha) \in c^*L_M : v \in TN \text{ and } \iota_v\gamma = c^*\alpha\}$ is a smooth subbundle.
This is exactly the compatibility condition appearing in Definition \ref{3v9jbw09} after setting $L_N = \Gamma_\gamma$ and $\psi = -d\gamma$.
\end{remark}

\section{Intersection of 1-coisotropics}\label{o2caey5r}

Let $(\G_i, \omega_i, \phi_i)$, for $i = 1, 2, 3$, be quasi-symplectic groupoids together with 1-coisotropics
\[
\c_1 = (\c_{11}, \c_{12}) : (\C_1, L_1) \too \G_1 \times \G_2^-,
\quad
\c_2 = (\c_{22}, \c_{23}) : (\C_2, L_2) \too \G_2 \times \G_3^-,
\]
where $\G_i^-$ denotes the same groupoid $\G_i$ with the opposite quasi-symplectic structure.
In other words, we have a zigzag of 1-shifted coisotropic correspondences
\[
\begin{tikzcd}[column sep={4em,between origins},row sep={4em,between origins}]
& \C_1 \arrow[swap]{dl}{\c_{11}} \arrow{dr}{\c_{12}} & & \C_2 \arrow[swap]{dl}{\c_{22}} \arrow{dr}{\c_{23}} & \\
\G_1 & & \G_2 & & \G_3,
\end{tikzcd}
\]
We wish to define the composition of these correspondences with a fibre product
\begin{equation}\label{ar6bxy7s}
\begin{tikzcd}[column sep={4em,between origins},row sep={4em,between origins}]
& & \C_1 \times_{\G_2} \C_2 \arrow{dl} \arrow{dr} & & \\
& \C_1 \arrow[swap]{dl}{\c_{11}} \arrow{dr}{\c_{12}} & & \C_2 \arrow[swap]{dl}{\c_{22}} \arrow{dr}{\c_{23}} & \\
\G_1 & & \G_2 & & \G_3,
\end{tikzcd}
\end{equation}
such that the induced map to $\G_1 \times \G_3^-$ has a canonical 1-shifted coisotropic structure.

There are two types of fibre products one can consider.
First, the \defn{strong fibre product} $\C_1 \stimes_{\G_2} \C_2$ is the usual set-theoretic fibre product on both the set of objects and arrows.
We also have the \defn{homotopy fibre product} $\C_1 \htimes_{\G_2} \C_2$, which is more natural in the context of differentiable stacks, as it presents the fibre product of the corresponding quotient stacks. % (we will recall its definition in more details below).
%We say that they are \defn{clean} if they are Lie groupoids with the expected tangent spaces.
The goal of this section is to show that, under suitable cleanness conditions, both types of fibre products have a canonical Dirac structure making \eqref{ar6bxy7s} into a 1-shifted coisotropic correspondence.
In particular, if $\G_1 = \G_3 = \{*\}$, we obtain a $0$-shifted Poisson structure (Proposition \ref{q2gecljb}).

\subsection{Strong fibre products}

The \defn{strong fibre product} of Lie groupoid morphisms  $\c_1 : \C_1 \to \G$ and $\c_2 : \C_2 \to \G$ is the groupoid
\[
\C_1 \stimes_\G \C_2 \coloneqq (\C_1\arr \times_{\G\arr} \C_2\arr \tto \C_1\obj \times_{\G\obj} \C_2\obj),
\]
with the obvious structure maps.
If the maps $\c_i\arr$ on arrows intersect cleanly, then so do the maps $\c_i\obj$ on objects \cite[Lemma A.1.3]{bur-cab-hoy:16}.
In that case, the strong fibre product is a Lie groupoid \cite[Proposition A.1.4]{bur-cab-hoy:16}, and we say that $\C_1 \stimes_\G \C_2$ is \defn{clean}.
It satisfies the universal property of fibre products for the category of Lie groupoids \cite[Proposition A.1.4]{bur-cab-hoy:16}.

\begin{theorem}%[Strong fibre products of 1-shifted coisotropic correspondences]
\label{2jy36i7i}
Let $(\G_i, \omega_i, \phi_i)$, for $i = 1, 2, 3$, be quasi-symplectic groupoids together with 1-coisotropics
\[
\c_1 = (\c_{11}, \c_{12}) : (\C_1, L_1) \too \G_1 \times \G_2^-
\quad\text{and}\quad
\c_2 = (\c_{22}, \c_{23}) : (\C_2, L_2) \too \G_2 \times \G_3^-,
\]
and consider the strong fibre product $\C \coloneqq \C_1 \stimes_{\G_2} \C_2$.
Suppose that the vector bundle homomorphisms
\begin{equation}\label{tlkiiyxl}
\begin{tikzcd}[column sep=4em]
A_{\C_1} \arrow{r}{\c_{12*}} & A_{\G_2} & A_{\C_2} \arrow[swap]{l}{\c_{22*}}
\end{tikzcd}
\end{equation}
are transverse (\S\ref{fj51c8fw}).

\begin{enumerate}[label={\textup{(\arabic*)}}]
\item \label{ukvnir14}
If $\C$ is clean and the sum of pullbacks of Dirac structures (\S\ref{1n8d3ppc})
\[
L \coloneqq p_1^*L_1 + p_2^*L_2
\]
is smooth, where $(p_1, p_2) : \C\obj \to \C\obj_1 \times \C\obj_2$ are the natural projections, then $L$ is a 1-shifted coisotropic structure on $\c : \C \to \G_1 \times \G_3^-$.

\item \label{zheuq0qk}
If $\C$ is clean and the vector bundle homomorphisms
\begin{equation}\label{cqrik68i}
\begin{tikzcd}[column sep=4em]
L_1 \arrow{r}{\c_{12*} p_T} & T\G\obj_2 & L_2 \arrow[swap]{l}{\c_{22*} p_T}
\end{tikzcd}
\end{equation}
intersect cleanly, then $L$ is smooth.
If \eqref{cqrik68i} is transverse and $L_i$ are strong, then $L$ is strong.

\item \label{u56ix79e}
Let $A_\C \coloneqq A_{\C_1} \times_{A_{\G_2}} A_{\C_2}$, viewed as a family of vector spaces over $\C\obj$, let $\aaa_\C : A_\C \to T\C\obj_1 \times_{T\G_2} T\C\obj_2 : (b_1, b_2) \mto (\aaa_{\C_1} b_1, \aaa_{\C_2} b_2)$, let $\c_* : A_\C \to A_{\G_1} \times A_{\G_3} : (b_1, b_2) \mto (\c_{11*}b_1, \c_{23*}b_2)$, and let $R$ be the sum of the images of \eqref{cqrik68i}.
Then there is a short exact sequence
\[
\begin{tikzcd}[column sep=1.5em]
\qquad\quad 
0 \arrow{r} & 
p_1^*(\ker \aaa_{\C_1} \cap \ker \c_{1*}) \oplus p_2^*(\ker \aaa_{\C_2} \cap \ker \c_{2*}) \arrow{r} & \ker \aaa_\C \cap \ker \c_* \arrow{r} & R^\circ
\arrow{r} & 0
\end{tikzcd}
\]
of families of vector spaces over $\C\obj$.
In particular, if $\ker \aaa_\C = 0$ then \eqref{cqrik68i} is transverse, and hence so is $\C\obj$.
\end{enumerate}
\end{theorem}

\begin{remark}\
\begin{enumerate}[label={(\roman*)}]
\item
In most applications, such as in the context of Hamiltonian reduction (\S\ref{616chpl6}), the condition that \eqref{tlkiiyxl} is transverse is automatic, even if the fibre product $\C_1 \stimes_{\G_2} \C_2$ is only a clean intersection.
\item
Part \ref{u56ix79e} can be viewed as a generalization of the familiar fact in symplectic geometry that if a Hamiltonian action is locally free, then $0$ is a regular value of the moment map.
\end{enumerate}
\end{remark}

In particular, taking $\G_1$ and $\G_3$ to be trivial, Theorem \ref{2jy36i7i} together with Proposition \ref{q2gecljb} imply that the intersection of two 1-coisotropics is 0-shifted Poisson:

\begin{corollary}
Let 
\begin{equation}
\begin{tikzcd}[column sep=3em]
(\C_1, L_1) \arrow{r}{\c_1} & (\G, \omega, \phi) & (\C_2, L_2) \arrow[swap]{l}{\c_2}
\end{tikzcd}
\end{equation}
be two 1-coisotropics over a common quasi-symplectic groupoid.
Suppose that the strong fibre product $\C \coloneqq \C_1 \stimes_{\G} \C_2$ is clean and that the vector bundle homomorphisms $A_{\C_1} \to A_{\G} \leftarrow A_{\C_2}$ are transverse.
\begin{enumerate}[label={\textup{(\arabic*)}}]
\item
If $L \coloneqq p_2^*L_2 - p_1^*L_1$ is smooth, where $(p_1, p_2) : \C\obj \to \C\obj_1 \times \C\obj_2$ are the natural projections, then it is a 0-shifted Poisson structure on $\C$.

\item
If the vector bundle homomorphisms $L_1 \overset{\c_{1*}p_T}{\too} T\G\obj \overset{\c_{2*} p_T}{\longleftarrow} L_2$ intersect cleanly, then $L$ is smooth.
\end{enumerate}
\end{corollary}

The rest of this subsection is devoted to the proof of Theorem \ref{2jy36i7i}, which is divided into several lemmas.
%Let $(\G_i, \omega_i, \phi_i)$, for $i = 1, 2, 3$, be quasi-symplectic groupoids together with Lie groupoid morphisms $\c_1$ and $\c_2$ as in \eqref{btcyn5j6}.
%Let $\C \coloneqq \C_1 \stimes_{\G_2} \C_2$.
We denote the natural projections by
\begin{align*}
(p_1, p_2) &: \C\obj% \coloneqq \C\obj_1 \times_{\G\obj_2} \C\obj_2 
\too \C\obj_1 \times \C\obj_2 \\
(q_1, q_2) &: \C\arr% \coloneqq \C\arr_1 \times_{\G\arr_2} \C\arr_2
\too \C\arr_1 \times \C\arr_2.
\end{align*}
It is useful to keep in mind the following commutative diagrams.
\[
\begin{tikzcd}[column sep={3.5em,between origins},row sep={3.5em,between origins}]
& & \C\obj
\arrow[swap]{dl}{p_1} \arrow{dr}{p_2} & & \\
&\C\obj_1 \arrow[swap]{dl}{\c_{11}} \arrow{dr}{\c_{12}}
& & \C\obj_2 \arrow[swap]{dl}{\c_{22}} \arrow{dr}{\c_{23}} & \\
\G\obj_1 & & \G\obj_2 & & \G\obj_3,
\end{tikzcd}
\quad
\begin{tikzcd}[column sep={3.5em,between origins},row sep={3.5em,between origins}]
& & \C\arr
\arrow[swap]{dl}{q_1}  \arrow{dr}{q_2} & & \\
&\C\arr_1 \arrow[swap]{dl}{\c_{11}} \arrow{dr}{\c_{12}}
& & \C\arr_2 \arrow[swap]{dl}{\c_{22}} \arrow{dr}{\c_{23}} & \\
\G\arr_1 & & \G\arr_2 & & \G\arr_3.
\end{tikzcd}
\]
Suppose that \eqref{tlkiiyxl} is transverse and let $L \coloneqq p_1^*L_1 + p_2^*L_2$.

\begin{lemma}
Part \ref{ukvnir14} of Theorem \ref{2jy36i7i} holds.
%If $\C$ is clean and $L$ is smooth, then $L$ is a 1-shifted coisotropic structure on $\C \to \G_1 \times \G_3^-$.
\end{lemma}

\begin{proof}%[Proof of \ref{ukvnir14}]

%Note that
%\[
%\c_{12} \circ p_1 = \c_{22} \circ p_2 
%\quad \text{and} \quad
%\c_{12} \circ q_1 = \c_{22} \circ q_2.
%\]

Since $L$ is smooth, it is a Dirac structure with background 3-form
\begin{align*}
&p_1^*\eta_1 + p_2^*\eta_2 = p_1^*(\c_{11}^*\phi_1 - \c_{12}^*\phi_2) + p_2^*(\c_{22}^*\phi_2 - \c_{23}^*\phi_3) = p_1^*\c_{11}^*\phi_1 - p_2^*\c_{23}^*\phi_3.
\end{align*}
This proves the first part of the compatibility condition in Definition \ref{gm52z93m}.
Now,
\begin{align*}
\ttt^*L &= \ttt^*p_1^*L_1 + \ttt^* p_2^*L_2\\
&= q_1^*\ttt^*L_1 + q_2^*\ttt^*L_2 \\
&= q_1^*(\sss^*L_1 + \Gamma_{\c_{11}^*\omega_1 - \c_{12}^*\omega_2}) + q_2^*(\sss^*L_2 + \Gamma_{\c_{22}^*\omega_2 - \c_{23}^*\omega_3}) \\
&= \sss^*L + \Gamma_{q_1^*\c_{11}^*\omega_1 - q_2^*\c_{23}^*\omega_3},
\end{align*}
which proves the second part of the compatibility condition.
To check the non-degeneracy condition, let $(((v_1, v_2), (\alpha_1, \alpha_2)), (a_1, a_3)) \in L \times_\c (A_{\G_1} \times A_{\G_3})$, i.e.\ $(v_1, v_2) \in T\C\obj$, $(\alpha_1, \alpha_2) \in T^*\C\obj$, and $(a_1, a_3) \in A_{\G_1} \times A_{\G_3}$ are such that $(v_i, \alpha_i) \in L_i$, $(\c_{11*}v_1, \c_{23*}v_2) = (\aaa_{\G_1}a_1, \aaa_{\G_3}a_3)$, and
\begin{equation}\label{5986lcaf}
(\alpha_1, \alpha_2) = p_1^*\c_{11}^* \IM_{\omega_1} a_1 - p_2^*\c_{23}^* \IM_{\omega_3} a_3.
\end{equation}
We need to show that there exists $(b_1, b_2) \in A_{\C}$ such that $\aaa_\C(b_1, b_2) = (v_1, v_2)$ and $(\c_{11*}b_1, \c_{23*}b_2) = (a_1, a_3)$.
We identify $T^*\C\obj$ with the quotient of $T^*\C\obj_1 \times T^*\C\obj_2$ by the image of the map $\xi \mto (-\c_{12}^*\xi, \c_{22}^*\xi)$ for $\xi \in T^*\G\obj_2$.
By \eqref{5986lcaf}, we have
\[
(\alpha_1 - \c_{11}^*\IM_{\omega_1}a_1, \alpha_2 + \c_{23}^*\IM_{\omega_3}a_3) = (-\c_{12}^*\xi, \c_{22}^*\xi),
\]
for some $\xi \in T^*\G\obj_2$.
We claim that
\begin{equation}\label{eimn63fp}
-\aaa_{\G_2}^*\xi = \IM_{\omega_2}^*\c_{12*}v_1 = \IM_{\omega_2}^*\c_{22*}v_2.
\end{equation}
For all $b_1 \in A_{\C_1}$ we have $(\aaa_{\C_1}b_1, \c_{11}^*\IM_{\omega_1}\c_{11*}b_1 - \c_{12}^*\IM_{\omega_2}\c_{12*}b_1) \in L_1$ by Lemma \ref{18e8u16h}, so
\[
0 = \ip{(v_1, \alpha_1), (\aaa_{\C_1}b_1, \c_{11}^*\IM_{\omega_1}\c_{11}b_1 - \c_{12}^*\IM_{\omega_2}\c_{12*}b_1)} = \ip{\alpha_1, \aaa_{\C_1}b_1} + \ip{\c_{11}^*\IM_{\omega_1}\c_{11*}b_1 - \c_{12}^*\IM_{\omega_2}\c_{12*}b_1, v_1}.
\]
It follows that
\begin{align*}
\ip{-\aaa_{\G_2}^*\xi, \c_{12*}b_1}
&= -\ip{\c_{12}^*\xi, \aaa_{\C_1}b_1} \\
&= \ip{\alpha_1 - \c_{11}^*\IM_{\omega_1}a_1, \aaa_{\C_1}b_1} \\
&= \ip{\alpha_1, \aaa_{\C_1}b_1} - \ip{\IM_{\omega_1}a_1, \c_{11*} \aaa_{\C_1} b_1} \\
&= \ip{\c_{12}^*\IM_{\omega_2}\c_{12*}b_1 - \c_{11}^*\IM_{\omega_1}\c_{11*}b_1, v_1} - \ip{\IM_{\omega_1}a_1, \aaa_{\G_1} \c_{11*} b_1} \\
&= \ip{\IM_{\omega_2}\c_{12*}b_1, \c_{12*}v_1} - \ip{\IM_{\omega_1}\c_{11*}b_1, \c_{11*}v_1} + \ip{\IM_{\omega_1}\c_{11*}b_1, \aaa_{\G_1} a_1} \\
&= \ip{\IM_{\omega_2}^*\c_{12*}v_1, \c_{12*}b_1}.
\end{align*}
Similarly,
\[
\ip{-\aaa_{\G_2}^*\xi, \c_{22*}b_2} = \ip{\IM_{\omega_2}^*\c_{22*}v_2, \c_{22*}b_2}.
\]
for all $b_2 \in A_{\C_2}$.
By transversality of \eqref{tlkiiyxl} and the fact that $\c_{12*}v_1 = \c_{22*}v_2$, this proves \eqref{eimn63fp}.
Lemma \ref{7wn1m20c}\ref{r6sdh90a} then implies that $(\c_{12*}v_1, \xi) = (\aaa_{\G_2}a_2, \IM_{\omega_2}a_2)$ for some $a_2 \in A_{\G_2}$.
It follows that $((v_1, \alpha_1), (a_1, a_2)) \in L_1 \times_{\c_1} (A_{\G_1} \times A_{\G_2})$ and $((v_2, \alpha_2), (a_2, a_3)) \in L_2 \times_{\c_2} (A_{\G_2} \times A_{\G_3})$.
By non-degeneracy of $L_1$ and $L_2$, there exist $b_1 \in A_{\C_1}$ and $b_2 \in A_{\C_2}$ such that $((v_1, \alpha_1), (a_1, a_2)) = ((\aaa_{\C_1}b_1, \c_{11}^*\IM_{\omega_1}\c_{11*}b_1 - \c_{12}^*\IM_{\omega_2}\c_{12*}b_1), (\c_{11*}b_1, \c_{12*}b_1))$ and $((v_2, \alpha_2), (a_2, a_3)) = ((\aaa_{\C_2}b_2, \c_{22}^*\IM_{\omega_2}\c_{22*}b_2 - \c_{23}^*\IM_{\omega_3}\c_{23*}b_2), (\c_{22*}b_2, \c_{23*}b_2))$.
In particular, $\c_{12*}b_1 = a_2 = \c_{22*}b_2$, so $(b_1, b_2) \in A_\C$.
\end{proof}

Let $R$ be the sum of the images of \eqref{cqrik68i}, viewed as a family of vector spaces over $\C\obj$.
The cleanness of \eqref{cqrik68i} is the statement that $R$ has constant rank.
Note that
\[
R^\circ = \{\xi \in T^*\G\obj_2 : (0, \c_{12}^*\xi) \in L_1 \text{ and } (0, \c_{22}^*\xi) \in L_2\}.
\]

\begin{lemma}\label{ej8aeupa}
If $\C$ is clean and $R$ has constant rank, then $L$ is smooth.
\end{lemma}

\begin{proof}
Since $\C\obj$ is a clean intersection, the set
\[
U \coloneqq \im(\c_{12*}) + \im(\c_{22*}) \s T\G_2\obj
\]
is a smooth vector bundle over $\C\obj$.
Similarly, $R$ is a smooth vector bundle over $\C\obj$.
We temporarily use the notation $p_i^!L_i$ for the vector bundle pullback of $L_i$ by $p_i$ to distinguish it from the pullback of Dirac structure.
We then have an exact sequence
\[
\begin{tikzcd}[row sep=0pt]
0 \arrow{r} & U^\circ \arrow[hook]{r} &  R^\circ \arrow{r} & p_1^!L_1 \oplus_{T\G\obj_2} p_2^!L_2 \arrow{r} & p_1^*L_1 + p_2^*L_2 \arrow{r} & 0 \\
& & \xi \arrow[maps to]{r} & ((0, \c_{12}^*\xi), (0, -\c_{22}^*\xi)) & & \\
& & & 
((v_1, \alpha_1), (v_2, \alpha_2)) \arrow[maps to]{r} &
((v_1, v_2), p_1^*\alpha_1 + p_2^*\alpha_2). &
\end{tikzcd}
\]
Since $U$, $R^\circ$, and $p_1^*L_1 + p_2^*L_2$ have constant rank, so does $p_1^!L_1 \oplus_{T\G\obj_2} p_2^!L_2$.
Since the latter can be expressed as the kernel of a vector bundle homomorphism, it is smooth, and hence so is its image.
\end{proof}

Without assuming that $\C$ is clean, we define its Lie algebroid as a family of vector spaces
\[
A_\C \coloneqq A_{\C_1} \times_{A_{\G_2}} A_{\C_2}
\]
over $\C\obj$, with anchor map $\aaa_\C(b_1, b_2) = (\aaa_{\C_1}b_1, \aaa_{\C_2}b_2)$.

\begin{lemma}\label{oe2uqo50}
Part \ref{u56ix79e} of Theorem \ref{2jy36i7i} holds, i.e.\ we have a short exact sequence
\[
\begin{tikzcd}[column sep=2em]
0 \arrow{r} & p_1^*(\ker \aaa_{\C_1} \cap \ker \c_{1*}) \oplus p_2^*(\ker \aaa_{\C_2} \cap \ker \c_{2*}) \arrow{r} & \ker \aaa_\C \cap \ker \c_* \arrow{r} & R^\circ \arrow{r} & 0
\end{tikzcd}
\]
of families of vector spaces over $\C\obj$.
\end{lemma}

\begin{proof}
By Lemma \ref{18e8u16h}, we have a well-defined map
\begin{equation}\label{qvio0coz}
\ker \aaa_\C \cap \ker \c_* \too R^\circ, \quad
(b_1, b_2) \mtoo \IM_{\omega_2} \c_{12*} b_1 = \IM_{\omega_2} \c_{22*} b_2.
\end{equation}
To show that this map is surjective, let $\xi \in R^\circ$.
By \eqref{eimn63fp} with $v_i = 0$ and $a_i = 0$, we have $\aaa_{\G_2}^*\xi = 0$ and hence $(0, \xi) = (\aaa_{\G_2}a_2, \IM_{\omega_2}a_2)$ for some $a_2 \in A_{\G_2}$ (Lemma \ref{7wn1m20c}\ref{r6sdh90a}).
In particular, $((0, -\c_{12}^*\xi), (0, a_2)) \in L_1 \times_{\c_1} (A_{\G_1} \times A_{\G_2})$.
By the non-degeneracy of $L_1$, there exists $b_1 \in A_{\C_1}$ such that $((0, -\c_{12}^*\xi), (0, a_2)) = ((\aaa_{\C_1} b_1, \c_{11}^*\IM_{\omega_1} \c_{11*}b_1 - \c_{12}^*\IM_{\omega_2} \c_{12*} b_1), (\c_{11*}b_1, \c_{12*} b_1))$.
It follows that $\xi = \IM_{\omega_2} \c_{12*} b_1$ for some $b_1 \in A_{\C_1}$ such that $\aaa_{\C_1} b_1 = 0$ and $\c_{11*}b_1 = 0$.
Similarly, $\xi = \IM_{\omega_2} \c_{22*} b_2$ for some $b_2 \in A_{\C_2}$ such that $\aaa_{\C_2} b_2 = 0$ and $\c_{23*} b_2 = 0$.
Since $\IM_{\omega_2} \c_{12*}b_1 = \IM_{\omega_2} \c_{22*}b_2$ and $\aaa_{\G_2} \c_{12*} b_1 = 0 = \aaa_{\G_2} \c_{22*}b_2$, we have $\c_{12*}b_1 = \c_{22*} b_2$ (Lemma \ref{7wn1m20c}\ref{rwp7208a}), so $(b_1, b_2) \in A_\C$ and $\aaa_{\C}(b_1, b_2) = 0$.
It follows that $\xi$ is the image of $(b_1, b_2) \in \ker \aaa_\C \cap \ker \c_*$ under \eqref{qvio0coz}, so the map is surjective.
The other map in the sequence is simply inclusion.
To show exactness at the middle term, let $(b_1, b_2) \in \ker \aaa_\C \cap \ker \c_*$ be such that $\IM_{\omega_2} \c_{12*}b_1 = \IM_{\omega_2} \c_{22*} b_2 = 0$.
Then $\aaa_{\G_2} \c_{12*}b_1 = 0 = \aaa_{\G_2} \c_{22*}b_2$, so $\c_{12*}b_1 = 0 = \c_{22*}b_2$ (Lemma \ref{7wn1m20c}\ref{rwp7208a}).
We then have $b_i \in \ker \aaa_{\C_i} \cap \ker \c_{i*}$ for $i = 1, 2$.
\end{proof}

\begin{lemma}
Part \ref{zheuq0qk} of Theorem \ref{2jy36i7i} holds.
\end{lemma}

\begin{proof}
The fact that $L$ is smooth follows from Lemma \ref{ej8aeupa}.
If $L_i$ are strong then $\ker \aaa_{\C_i} \cap \ker \c_{i*} = 0$, so if $R^\circ = 0$, then $L$ is also strong by Lemma \ref{oe2uqo50}.
\end{proof}

In particular, we obtain a differential-geometric version of the fact that $n$-shifted coisotropics have an $(n - 1)$-shifted Poisson structure \cite{cptvv,mel-saf:18a,mel-saf:18b}:

\begin{corollary}\label{t61dwbsa}
Let $\c : (\C, L) \to (\G, \omega, \phi)$ be a 1-coisotropic and let $L_\G$ be the Dirac structure on $\G\obj$ induced by $(\omega, \phi)$.
If $L - \c^*L_\G$ is smooth, then it is a 0-shifted Poisson structure on $\C$.
This holds, in particular, if $\C$ has discrete isotropy groups, or if $L \overset{\c_*p_T}{\too} T\G \overset{p_T}{\longleftarrow} L_\G$ intersect cleanly.
\end{corollary}

\begin{proof}
By Proposition \ref{2kzh0or3}, we have coisotropic correspondences
\[
\begin{tikzcd}[column sep={4em,between origins},row sep={2em,between origins}]
& \C \arrow[swap]{dl} \arrow{dr}{\c} & & \G^- \arrow[swap]{dl}{\Id} \arrow{dr} & \\
\{*\} & & \G^- & & \{*\}.
\end{tikzcd}
\]
The strong fibre product is $\C$, so Part \ref{ukvnir14} of Theorem \ref{2jy36i7i} shows that if $L - \c^*L_\G$ is smooth, then it is a 1-shifted coisotropic structure over a point, i.e.\ a 0-shifted Poisson structure (Proposition \ref{q2gecljb}).
The short exact sequence in Part \ref{u56ix79e} of Theorem \ref{2jy36i7i} reduces to
\[
0 \too \ker \aaa_\C \cap \ker \c_* \too \ker \aaa_\C \too R^\circ \too 0.
\]
In particular, if $\C$ has discrete isotropy groups, then $\ker \aaa_\C = 0$, so $R^\circ = 0$.
The result then follows by Part \ref{zheuq0qk} of Theorem \ref{2jy36i7i}.
\end{proof}

%\begin{example}
%Let $\G \tto M$ be a symplectic groupoid and $N \s M$ a coisotropic submanifold.
%Then the conormal bundle $TN^\circ$ intergrates to a Lagrangian subgroupoid $\H \to \G$.
%The trivial Dirac structure $TN$ on $N$ gives $\H \to \G$ the structure of a 1-shifted Lagrangian.
%Let $L_N = \{(v, -\alpha|_N) \in TN \oplus T^*N : v = \sigma(\alpha)\}$, where $\sigma : T^*M \to TM$ is the Poisson structure which $\G$ integrates.
%If $L_N$ is smooth, it is a weak 0-shifted Poisson structure.
%If $TN^\circ \cap \ker \sigma = 0$, it is a strong one.
%\end{example}
%

\begin{remark}
The bundle $L - \c^*L_\G$ in Corollary \ref{t61dwbsa} may fail to be smooth.
For example, let $\G$ be a symplectic groupoid and $\C\obj \s \G\obj$ a coisotropic submanifold as in Example \ref{s2rkhg7b}.
Then we have a 1-Lagrangian $\C \to \G$, where $\C$ integrates the conormal bundle and $L = T\C\obj$ is trivial.
It follows that $L - \c^*L_\G = \c^*\Gamma_{-\pi}$, where $\pi$ is the Poisson structure on $\G\obj$.
However, the pullback of the graph of Poisson structure to a coisotropic submanifold may not be smooth. % (but see \cite[Example 1.11]{bur:13} for a sufficient condition).
For a specific example, take $\R^2$ with $\pi = x \frac{\d}{\d x} \wedge \frac{\d}{\d y}$ and $\{(x, y) : y = 0\}$.
Then $(\c^*L_\pi)_{(x, 0)}$ is equal to $\Span(\frac{\d}{\d x})$ if $x \ne 0$ and to $\Span(dx)$ if $x = 0$.
\end{remark}

\subsection{Homotopy fibre product}

The \defn{homotopy fibre product} (also known as the \emph{weak fibre product} \cite[Section 5.3]{moe-mrc:03} or \emph{weak homotopy pullback} \cite{hoy:13}), % or weak 2-pullback \cite{car:11}), 
denoted $\C_1 \htimes_\G \C_2$, is the groupoid whose spaces of objects and arrows are respectively
\begin{align}
(\C_1 \htimes_\G \C_2)\obj &\coloneqq \C\obj_1 \times_{\G\obj} \G\arr \times_{\G\obj} \C\obj_2 = \{(x_1, g, x_2) : \c_1(x_1) = \sss(g) \text{ and } \c_2(x_2) = \ttt(g)\} \label{bmv5v7op} \\
(\C_1 \htimes_\G \C_2)\arr &\coloneqq \C\arr_1 \times_{\G\obj} \G\arr \times_{\G\obj} \C\arr_2 = \{(h_1, g, h_2) : \c_1(\sss(h_1)) = \sss(g) \text{ and } \c_2(\sss(h_2)) = \ttt(g)\},\label{ocdak67i}
\end{align}
with structure maps
\begin{align*}
\sss(h_1, g, h_2) &= (\sss(h_1), g, \sss(h_2)) \\
\ttt(h_1, g, h_2) &= (\ttt(h_1), \c_2(h_2) \cdot g \cdot \c_1(h_1)^{-1}, \ttt(h_2)) \\
\mmm((h_1, g, h_2), (\tilde{h}_1, \tilde{g}, \tilde{h}_2)) &= (h_1 \tilde{h}_1, \tilde{g}, h_2 \tilde{h}_2) \\
\uuu(x_1, g, x_2) &= (\uuu_{x_1}, g, \uuu_{x_2}).
\end{align*}
If the intersection \eqref{bmv5v7op} is clean, % i.e.\ the maps $(\c_1, \c_2) : \C\obj_1 \times \C\obj_2 \to \G\obj \times \G\obj$ and $(\sss, \ttt) : \G\arr \to \G\obj \times \G\obj$ intersect cleanly, 
then \eqref{ocdak67i} is also clean and $\C_1 \htimes_\G \C_2$ is Lie groupoid \cite[Section 5.3]{moe-mrc:03}.
In that case, we say that the homotopy fibre product is \defn{clean}.
This happens, for example, if either $\c_1\obj$ or $\c_2\obj$ is a submersion.
The homotopy fibre product satisfies a universal property for 2-commutative squares \cite[Section 5.3]{moe-mrc:03} and hence presents the fibre product of the corresponding stacks when it exists.
In particular, the square
\[
\begin{tikzcd}[row sep={3em,between origins},column sep={3em,between origins}]
& \C_1 \htimes_\G \C_2 \arrow{dl} \arrow{dr} & \\
\C_1 \arrow{dr} \arrow[Rightarrow,shorten=16pt]{rr} & & \C_2 \arrow{dl} \\
& \G &
\end{tikzcd}
\]
is 2-commutative, where the natural transformation is the projection $\pr_\G : (\C_1 \htimes_\G \C_2)\obj \to \G\arr$.
%This is a better notion of fibre products in the context of differentiable stacks, since it presents the fibre product of stacks $\B\C_1 \times_{\B\G} \B\C_2$; see, for example, Fulton's book on stacks, Example 4.21.

\begin{remark}
In many situations, the weak and strong fibre products are Morita equivalent.
This happens, for example, if either $\C_1$ or $\C_2$ is an action groupoid of $\G$.
See also \cite[Proposition 3.2.5]{lac:22}.
\end{remark}

\begin{theorem}%[Homotopy fibre products of 1-shifted coisotropic correspondences]
\label{rt0qkos3}
Let $(\G_i, \omega_i, \phi_i)$, for $i = 1, 2, 3$, be quasi-symplectic groupoids together with 1-coisotropics
\[
\c_1 = (\c_{11}, \c_{12}) : (\C_1, L_1) \too \G_1 \times \G_2^-
\quad\text{and}\quad
\c_2 = (\c_{22}, \c_{23}) : (\C_2, L_2) \too \G_2 \times \G_3^-.
\]
Consider the homotopy fibre product $\C \coloneqq \C_1 \htimes_{\G_2} \C_2$.

\begin{enumerate}[label=\textup{(\arabic*)}]
\item \label{lpfuwmeg}
If $\C$ is clean and
\begin{equation}\label{uk28lq9e}
L \coloneqq p_1^*L_1 + p_2^*L_2 - p_0^*\Gamma_{\omega_2}
\end{equation}
is smooth, where $(p_1, p_0, p_2) : \C\obj \to \C\obj_1 \times \G\arr_2 \times \C\obj_2$ are the natural projections, then $L$ is a 1-shifted coisotropic structure on $\c : \C \to \G_1 \times \G_3^-$.

\item \label{k9kpqiad}
%If the vector bundle on $N$ defined by
%\begin{align*}
%R \coloneqq \im(\mu_{12} \aaa_{N_1}) \times \im(\mu_{22} \aaa_{N_2}) + \im((\sss, \ttt) : T\G_2 \to TM_2 \times TM_2)
%\end{align*}
%has constant rank, i.e.\ the vector bundles
If $\C$ is clean and the vector bundle homomorphisms
\begin{equation}\label{rf57it5e}
\begin{tikzcd}[column sep=6em]
L_1 \times L_2 \arrow{r}{(\c_{12} p_T, \c_{22} p_T)} & T\G\obj_2 \times T\G\obj_2 & T\G\arr_2 \arrow[swap]{l}{(\sss, \ttt)}
\end{tikzcd}
\end{equation}
intersect cleanly, then $L$ is smooth.
If \eqref{rf57it5e} is transverse and $L_i$ are strong, then $L$ is strong.

\item \label{a09rl9m8}
Let $A_\C \coloneqq (p_1, p_2)^*(A_{\C_1} \times A_{\C_2})$, viewed as a family of vector spaces over $\C\obj$, let $\aaa_\C : A_\C \to T\C\obj_1 \times_{T\G\obj_2} T\G\arr_2 \times_{T\G\obj_2} T\C\obj_2 : (b_1, b_2) \mto (\aaa_{\C_1} b_1, (\c_{22*}b_2)^R - (\c_{12*} b_1)^L, \aaa_{\C_2} b_2)$, let $\c_* : A_\C \to A_{\G_1} \times A_{\G_3} : (b_1, b_2) \mto (\c_{11*}b_1, \c_{23*}b_2)$, and let $R$ be the sum of the images of \eqref{rf57it5e}.
Then there is a sequence
\[
\begin{tikzcd}[column sep=1.5em]
\qquad\quad 0 \arrow{r} & 
(\ker \aaa_{\C_1} \cap \ker \c_{1*}) \times (\ker \aaa_{\C_2} \cap \ker \c_{2*}) \arrow{r} & \ker \aaa_\C \cap \ker \c_* \arrow{r} & R^\circ
\arrow{r} & 0,
\end{tikzcd}
\]
of families of vector spaces over $\C\obj$ which is exact at the first two terms.
If $A_{\C_1} \to A_{\G_2} \leftarrow A_{\C_2}$ are transverse, then it is also exact at the third term $R^\circ$.
In particular, if in addition $\ker \aaa_\C \cap \ker \c_* = 0$, then \eqref{rf57it5e} is transverse and hence $\C$ is clean.
%$L$ is 1-shifted coisotropic structure (***Note: we could state something stronger since in that case the homotopy fibre product exists as a consequence***).

\end{enumerate}
\end{theorem}

As in the previous section, the special case where $\G_1$ and $\G_3$ are trivial implies that the homotopy fibre product of two 1-coisotropics is 0-shifted Poisson.

\begin{corollary}
Let 
\begin{equation}
\begin{tikzcd}[column sep=3em]
(\C_1, L_1) \arrow{r}{\c_1} & (\G, \omega, \phi) & (\C_2, L_2) \arrow[swap]{l}{\c_2}
\end{tikzcd}
\end{equation}
be two 1-coisotropics over a common quasi-symplectic groupoid and consider the homotopy fibre product $\C \coloneqq \C_1 \htimes_{\G_2} \C_2$.

\begin{enumerate}[label=\textup{(\arabic*)}]
\item
If $\C$ is clean and $L \coloneqq -p_1^*L_1 + p_2^*L_2 - p_0^*\Gamma_{\omega}$ is smooth, where $(p_1, p_0, p_2) : \C\obj \to \C\obj_1 \times \G\arr \times \C\obj_2$ are the natural projections, then $L$ is a 0-shifted Poisson structure on $\C$.

\item
If $\C$ is clean and the vector bundle homomorphisms
\[
\begin{tikzcd}[column sep=6em]
L_1 \times L_2 \arrow{r}{(\c_1 p_T, \c_2 p_T)} & T\G\obj \times T\G\obj_2 & T\G\arr \arrow[swap]{l}{(\sss, \ttt)}
\end{tikzcd}
\]
intersect cleanly, then $L$ is smooth.
%\item
%Let $\aaa_\C : (p_1, p_2)^*(A_{\C_1} \times A_{\C_2}) \to T\C\obj_1 \times_{T\G\obj} T\G\arr \times_{T\G\obj} T\C\obj_2 : (b_1, b_2) \mto (\aaa_{\C_1} b_1, (\c_{2*}b_2)^R - (\c_{1*} b_1)^L, \aaa_{\C_2} b_2)$.
%If $\ker \aaa_\C = 0$ and $A_{\C_1} \to A_{\G} \leftarrow A_{\C_2}$ are transverse, then $\C$ is clean.
\end{enumerate}
\end{corollary}

%\begin{remark}
%Part \ref{a09rl9m8} is a generalization of the familiar fact from Hamiltonian geometry that if we have a moment map $\mu : M \to \g^*$ with respect to a $G$-action, then $0$ is a regular value of $\mu$ if $G$ acts locally freely on $\mu^{-1}(0)$.
%In that case, the transversality condition is automatically true, and the discreteness condition corresponds to the condition that $G$ acts locally freely.
%\end{remark}

The rest of this subsection is devoted to the proof of Theorem \ref{rt0qkos3}.
It is useful to keep in mind the following commutative diagram.
\[
\begin{tikzcd}[column sep={3em,between origins},row sep={3em,between origins}]
& & & \C\obj
\arrow[swap]{dll}{p_1} \arrow{d}{p_0} \arrow{drr}{p_2} & & & \\
&\C\obj_1 \arrow[swap]{dl}{\c_{11}} \arrow{dr}{\c_{12}}
& & \G_2\arr \arrow[swap]{dl}{\sss} \arrow{dr}{\ttt} 
& & \C\obj_2 \arrow[swap]{dl}{\c_{22}} \arrow{dr}{\c_{23}} \\
\G\obj_1 & & \G\obj_2 & & \G\obj_2 & & \G\obj_3
\end{tikzcd}
\]
Consider also the natural projections
\[
\begin{tikzcd}[column sep={3em,between origins},row sep={3em,between origins}]
& & & \C\arr
\arrow[swap]{dll}{q_1} \arrow{d}{q_0} \arrow{drr}{q_2} & & & \\
&\C\arr_1 
& & \G_2\arr
& & \C\arr_2
\end{tikzcd}
\]
%\[
%\begin{tikzcd}[column sep={3em,between origins},row sep={3em,between origins}]
%& & & \C\arr
%\arrow[swap]{dll}{q_1} \arrow{d}{q_0} \arrow{drr}{q_2} & & & \\
%&\C\arr_1 \arrow[swap]{dl}{\c_{11}} \arrow{dr}{\c_{12}}
%& & \G_2\arr
%& & \C\arr_2 \arrow[swap]{dl}{\c_{22}} \arrow{dr}{\c_{23}} \\
%\G\arr_1 & & \G\arr_2 & & \G\arr_2 & & \G\arr_3.
%\end{tikzcd}
%\]
We have
\[
p_1 \circ \sss = \sss \circ q_1, \quad
p_2 \circ \sss = \sss \circ q_2, \quad
p_0 \circ \sss = q_0, \quad
p_1 \circ \ttt = \ttt \circ q_1, \quad
p_2 \circ \ttt = \ttt \circ q_2,
\]
and
\begin{equation}\label{1to6s83b}
p_0 \circ \ttt = \mmm^3 \circ \tau \circ (\iii, \Id, \Id) \circ (\c_{12}, \Id, \c_{22}),
\end{equation}
where $\mmm^3 : \G_2^3 \to \G\arr_2$ is the multiplication of 3-composable arrows, and $\tau$ exchanges the first and third factors.

\begin{lemma}
Part \ref{lpfuwmeg} of Theorem \ref{rt0qkos3} holds.
\end{lemma}

\begin{proof}
Since $L \coloneqq p_1^*L_{1} + p_2^*L_{2} - p_0^*\Gamma_{\omega_2}$ is smooth, it is a Dirac structure with background 3-form
\begin{align*}
p_1^*\eta_1 + p_2^*\eta_2 + p_0^*d\omega_2
&= p_1^*(\c_{11}^*\phi_1 - \c_{12}^*\phi_2) + p_2^*(\c_{22}^*\phi_2 - \c_{23}^*\phi_3) + p_0^*(\sss^*\phi_2 - \ttt^*\phi_2) \\
&= p_1^*\c_{11}^*\phi_1 - p_2^*\c_{23}^*\phi_3.
\end{align*}
This proves the first part of the compatibility condition in Definition \ref{gm52z93m}.
By the multiplicativity of $\omega_2$ and \eqref{1to6s83b} we have
\[
\ttt^*p_0^*\omega_2 = q_0^*\omega_2 - q_1^*\c_{12}^*\omega_2 + q_2^*\c_{22}^*\omega_2.
\]
Hence,
\begin{align*}
\ttt^*L &= \ttt^*p_1^*L_1 + \ttt^* p_2^*L_2 - \ttt^* p_0^*\Gamma_{\omega_2} \\
&= q_1^*\ttt^*L_1 + q_2^*\ttt^*L_2 - q_0^*\Gamma_{\omega_2} + q_1^*\c_{12}^*\Gamma_{\omega_2} - q_2^*\c_{22}^*\Gamma_{\omega_2} \\
&= q_1^*(\sss^*L_1 + \c_{11}^*\Gamma_{\omega_1}) + q_2^*(\sss^*L_2 - \c_{23}^*\Gamma_{\omega_3}) - q_0^*\Gamma_{\omega_2}\\
&= \sss^*(p_1^*L_1 + p_2^*L_2 - p_0^*\Gamma_{\omega_2}) + q_1^*\c_{11}^*\Gamma_{\omega_1} - q_2^*\c_{23}^*\Gamma_{\omega_3} \\
&= \sss^*L + \Gamma_{\c^*(\omega_1, -\omega_3)}.
\end{align*}
%where
%\begin{equation}\label{qs4rx11f}
%\c \coloneqq (\c_{11} q_1, \c_{23} q_2) : \C \too \G_1 \times \G_3^-.
%\end{equation}
This proves the second compatibility condition.
It remains to check the non-degeneracy condition.
First, it is useful to observe that the homotopy fibre product is the action groupoid
\[
(\C_1 \times \C_2) \ltimes \C\obj
\]
where $(h_1, h_2) \cdot (x_1, g, x_2) = (\ttt(h_1), \c_{22}(h_2) g \c_{12}(h_2)^{-1}, \ttt(h_2))$ if $(\sss(h_1), \sss(h_2)) = (x_1, x_2)$.
Hence, its Lie algebroid is the vector bundle pullback
\[
A_\C = p^*(A_{\C_1} \times A_{\C_2}),
\]
where
\[
p \coloneqq (p_1, p_2) : \C\obj \too \C\obj_1 \times \C\obj_2,
\]
with anchor map
\begin{equation}\label{n0nqav1s}
\aaa_\C(b_1, b_2) = (\aaa_{\C_1}b_1, (\c_{22*}b_2)^R - (\c_{12*} b_1)^L, \aaa_{\C_2}b_2).
\end{equation}
We need to show that
\[
(\aaa_\C, \c^* \IM \c_*, \c_*) : A_\C \too L \times_\c (A_{\G_1} \times A_{\G_3})
\]
is surjective, where $\IM \coloneqq (\IM_{\omega_1}, -\IM_{\omega_3}) : A_{\G_1} \times A_{\G_3} \to T^*\G\obj_1 \times T^*\G\obj_3$.
To this end, consider
\[
(((v_1, v_0, v_2), (\alpha_1, \alpha_0, \alpha_2)), (a_1, a_3)) \in L \times_c (A_{\G_1} \times A_{\G_3}).
\]
In other words, $(v_1, v_0, v_2) \in T\C\obj$, $(\alpha_1, \alpha_0, \alpha_2) \in T^*\C\obj$, and $(a_1, a_3) \in A_{\G_1} \times A_{\G_3}$ are such that
\begin{equation}\label{sxhtxotg}
(v_1, \alpha_1) \in L_1, \quad
(v_2, \alpha_2) \in L_2, \quad
(v_0, \alpha_0) \in -\Gamma_{\omega_2},
\end{equation}
\[
(\c_{11*}v_1, \c_{23*}v_3) = (\aaa a_1, \aaa a_3)
\]
and
\begin{equation}\label{p9whn1xh}
(\alpha_1, \alpha_0, \alpha_2) = p_1^*\c_{11}^* \IM_{\omega_1} a_1 - p_2^*\c_{23}^* \IM_{\omega_3} a_3.
\end{equation}
We need to show that there exists $(b_1, b_2) \in A_\C$ such that $\aaa_\C(b_1, b_2) = (v_1, v_0, v_2)$ and $(\c_{11*}b_1, \c_{23*}b_2) = (a_1, a_3)$.
By \eqref{p9whn1xh}, we have
\begin{equation}\label{u0750hx1}
(\alpha_1 - \c_{11}^*\IM_{\omega_1}a_1, \alpha_0, \alpha_2 + \c_{23}^*\IM_{\omega_3}a_3) = (-\c_{12}^*\xi, \sss^*\xi + \ttt^*\eta, -\c_{22}^*\eta),
\end{equation}
for some $\xi, \eta \in T^*\G\obj_2$.
Let $(x_1, g, x_2) \in \C\obj$ be the point over which $(v_1, v_0, v_2)$ lives.
Note that for all $u \in A_{\G_2}|_{\sss(g)}$, we have
\begin{equation}\label{7kh1m4w1}
(\aaa_{\G_2}^*\xi)(u) = -\xi(\sss_* u^L_g) = -(\sss^*\xi + \ttt^*\eta)(u^L_g) = - \alpha_0(u^L_g) = \omega_2(v_0, u^L_g).
\end{equation}
By Lemma \ref{7wn1m20c}\ref{xul5itsk}, it follows that
\begin{equation}\label{pnv6d3wd}
(\IM_{\omega_2}^* \sss_* v_0)(u) = \omega_2(u, \uuu \sss_* v_0) = (\sss^* \IM_{\omega_2} u)(v_0) = (i_{u^L_g}\omega_2)(v_0) = -\omega_2(v_0, u^L_g) = -(\aaa_{\G_2}^*\xi)(u).
\end{equation}
Hence, $(\c_{12*} v_1, \xi) = (\aaa_{\G_2} a_2, \IM_{\omega_2} a_2)$ for some $a_2 \in A_{\G_2}$ (Lemma \ref{7wn1m20c}\ref{r6sdh90a}).
We then have an element
\[
((v_1, \alpha_1), (a_1, a_2)) \in L_1 \times_{\c_1} (A_{\G_1} \times A_{\G_2}).
\]
By the non-degeneracy of the 1-shifted coiostropic structure $L_1$, there exists $b_1 \in A_{\C_1}$ such that $v_1 = \aaa_{\C_1} b_1$, $a_1 = \c_{11*} b_1$, and $a_2 = \c_{12*} b_1$.
By the same argument, $(\c_{22*}v_2, \eta) = (\aaa_{\G_2} \tilde{a}_2, -\IM_{\omega_2} \tilde{a}_2)$ for some $\tilde{a}_2 \in A_{\G_2}$, and there exists $b_2 \in A_{\C_2}$ such that $v_2 = \aaa_{\C_2} b_2$, $\tilde{a}_2 = \c_{22*} b_2$, and $a_3 = \c_{23*} b_2$.
It then suffices to show that $\aaa_\C(b_1, b_2) = (v_1, v_0, v_2)$.
According to \eqref{n0nqav1s}, this amounts to the identity
\[
(\tilde{a}_2)^R_g - (a_2)^L_g = v_0.
\]
%To simplify the notation, let $u \coloneqq a_2$ and $w \coloneqq \tilde{a}_2$; we want to show that $w^R - u^L - v_0 = 0$.
We have $\sss_*((\tilde{a}_2)^R_g - (a_2)^L_g - v_0) = \aaa a_2 - \sss_* v_0 = 0$ and $\ttt_*((\tilde{a}_2)^R_g - (a_2)^L_g - v_0) = \aaa\tilde{a}_2 - \ttt_* v_0 = 0$.
Hence, by the definition of quasi-symplectic structures, it suffices to show that $(\tilde{a}_2)^R_g - (a_2)^L_g - v_0 \in \ker \omega_2$.
By Lemma \ref{7wn1m20c}\ref{xul5itsk}, $i_{(a_2)^L_g}\omega_2 = \sss^*\IM_{\omega_2} a_2 = \sss^*\xi$ and $i_{(\tilde{a}_2)^R_g}\omega_2 = \ttt^* \IM_{\omega_2} \tilde{a}_2 = -\ttt^*\eta$.
It follows from \eqref{sxhtxotg} and \eqref{u0750hx1} that $i_{(\tilde{a}_2)^R_g - (a_2)^L_g - v_0}\omega_2 = -\ttt^*\eta -\sss^*\xi + \alpha_0 = 0$.
\end{proof}

%The proofs of \ref{k9kpqiad} and \ref{a09rl9m8} will follow from the following result.
Let $R$ is the sum of the images of \eqref{rf57it5e}.
Note that
\[
R^\circ = \{(\xi, \eta) \in T^*\G\obj_2 \times T^*\G\obj_2 : (0, \c_{12}^*\xi) \in L_1, (0, \c_{22}^*\eta) \in L_2, \sss^*\xi + \ttt^*\eta = 0\}.
\]

\begin{lemma}\label{jbecqi1g}
If $R$ has constant rank, then $L \coloneqq p_1^*L_{1} + p_2^*L_{2} - p_0^*\Gamma_{\omega_2}$ is smooth.
\end{lemma}

\begin{proof}
It suffices to show that $p_1^*L_1 + p_2^*L_2$ is smooth.
Since the intersection $\C\obj \coloneqq \C\obj_1 \times_{\G\obj_2} \G\arr_2 \times_{\G\obj_2} \C\obj_2$ is clean, then
\[
U = \im(\c_{12*}) \times \im(\c_{22*}) + \im(\sss_*, \ttt_*).
\]
has constant rank over $\C\obj$.
Note that we have an exact sequence
\[
\begin{tikzcd}[row sep=0pt, column sep=20pt]
0 \arrow{r} & U^\circ \arrow[hook]{r} &  R^\circ \arrow{r} & p_1^!L_1 \oplus_{T\G\obj_2} T\G\arr_2 \oplus_{T\G\obj_2} p_2^!L_2 \arrow{r} & p_1^*L_1 + p_2^*L_2 \arrow{r} & 0 \\
& & (\xi, \eta) \arrow[mapsto]{r} & ((0, \c_{12}^*\xi), 0, (0, \c_{22}^*\eta)) \\
& & &((v_1, \alpha_1), v_0, (v_2, \alpha_2)) \arrow[mapsto]{r} & ((v_1, v_0, v_2), p_1^*\alpha_1 + p_2^*\alpha_2),
\end{tikzcd}
\]
where, as before, $p_i^!L_i$ is the vector bundle pullback.
%Exactness at $p_1^*L_1 + p_2^*L_2$ follows by definition.
%Let $((v_1, \alpha_1), v_0, (v_2, \alpha_2))$ be such that $((v_1, v_0, v_2), p_1^*\alpha_1 + p_2^*\alpha_2) = 0$.
%Then $(\alpha_1, 0, \alpha_2) = (-\c_{12}^*\xi, \sss^*\xi + \ttt^*\eta, -\c_{22}^*\eta)$ for some $(\xi, \eta) \in T^*\G\obj_2 \times T^*\G\obj_2$, which shows exactness at the middle term.
Since $U^\circ$, $R^\circ$, and $p_1^*L_1 + p_2^*L_2$ have constant rank, so does $p_1^!L_1 \oplus_{T\G\obj_2} T\G\arr_2 \oplus_{T\G\obj_2} p_2^!L_2$.
Since the latter is the kernel of a vector bundle homomorphism, it is smooth, and hence so is its image.
\end{proof}

\begin{lemma}\label{zidpz55t}
Part \ref{a09rl9m8} of Theorem \ref{rt0qkos3} holds, i.e.\ there is a short exact sequence
\[
\begin{tikzcd}[column sep=1.5em]
\qquad\quad 0 \arrow{r} & 
(\ker \aaa_{\C_1} \cap \ker \c_{1*}) \times (\ker \aaa_{\C_2} \cap \ker \c_{2*}) \arrow{r} & \ker \aaa_\C \cap \ker \c_* \arrow{r} & R^\circ
\arrow{r} & 0,
\end{tikzcd}
\]
of families of vector spaces over $\C\obj$.
%Let $R \s T\G\obj_2 \times T\G\obj_2$ be the image of \eqref{rf57it5e}.
%Then we have a sequence
%\begin{equation}\label{4on99zwo}
%0 \too (\ker \aaa_{\C_1} \cap \ker \c_1) \times (\ker \aaa_{\C_2} \cap \ker \c_2) \too \ker \aaa_\C \cap \ker \c \too R^\circ \too 0
%\end{equation}
%(where $c$ is defined in \eqref{qs4rx11f}) which is exact at the first two terms, and also at the third term if the maps $A_{\C_1} \to A_{\G_2} \leftarrow A_{\C_2}$ are transverse.
\end{lemma}

\begin{proof}
The first map is inclusion.
For the second map, let
\begin{equation}\label{hl1zyyry}
\ker \aaa_\C \cap \ker \c_* \too R^\circ,
\quad
(b_1, b_2) \mtoo (\IM_{\omega_2} \c_{12*} b_1, -\IM_{\omega_2} \c_{22*} b_2).
\end{equation}
To see that this is well-defined, let $(b_1, b_2) \in \ker \aaa_\C \cap \ker \c_*$ and let $\xi \coloneqq \IM_{\omega_2} \c_{12*} b_1$ and $\eta \coloneqq -\IM_{\omega_2} \c_{22*} b_2$.
By Lemma \ref{18e8u16h}, we have $(0, \c_{12}^*\xi) \in L_1$ and $(0, \c_{22}^*\eta) \in L_2$.
By Lemma \ref{7wn1m20c}\ref{xul5itsk}, $\sss^*\xi + \ttt^*\eta = i_{(\c_{12*}b_1)^L - (\c_{22*}b_2)^R}\omega_2 = 0$ since $\aaa_\C(b_1, b_2) = 0$.
Hence, $(\xi, \eta) \in R^\circ$.
Exactness at $\ker \aaa_\C \cap \ker \c_*$ follows from Lemma \ref{7wn1m20c}\ref{rwp7208a}.

Now, suppose now that $A_{\C_1} \to A_{\G_2} \leftarrow A_{\C_2}$ are transverse.
We show that \eqref{n0nqav1s} is surjective.
Let $(\xi, \eta) \in R^\circ$.
%We first claim that $(0, \xi), (0, \eta) \in \im(\aaa_{\G_2}, \IM_{\omega_2})$.
For all $b_1 \in A_{\C_1}$, we have $(\aaa b_1, \c_{11}^*\IM_{\omega_1} \c_{11*}b_1 - \c_{12}^*\IM_{\omega_2} \c_{12*} b_1) \in L_1$ by Lemma \ref{18e8u16h}.
Since also $(0, \c_{12}^*\xi) \in L_1$ and $L_1^\perp = L_1$, we have $\xi(\aaa_{\G_2} \c_{12*} b_1) = 0$.
Similarly, $\xi(\aaa_{\G_2} \c_{22*}b_2) = 0$ for all $b_2 \in A_{\C_2}$.
The transversality assumption then implies that $\aaa_{\G_2}^* \xi = 0$.
It follows (by Lemma \ref{7wn1m20c}\ref{rwp7208a}) that $(0, \xi) = -(\aaa_{\G_2} a, \IM_\omega a)$ for some $a \in A_{\G_2}$ and hence $((0, \c_{12}^*\xi), (0, a)) \in L_1 \times_{\c_1}(A_{\G_1} \times A_{\G_2})$.
%By transversality, every $a \in A_{\G_2}$ can be written $a = \c_{12*}b_1 + \c_{22*} b_2$ for some $b_i \in A_{\C_i}$.
%Now, we have $(\aaa b_1, \c_{11}^*\IM_{\omega_1} \c_{11*}b_1 - \c_{12}^*\IM_{\omega_2} \c_{12*} b_1) \in L_1$ and $(0, \c_{12}^*\xi) \in L_1$, so by pairing them, we get that $\xi(\aaa \c_{12*} b_1) = \xi(\c_{12*}\aaa b_1) = 0$.
%Similarly, $\xi(\aaa \c_{22*}b_2) = 0$, and hence $\xi(\aaa a) = \xi(\aaa \c_{12*}b_1 + \aaa \c_{22*} b_2) = 0$.
%It follows that $(0, \xi) \in \im(\aaa_{\G_2}, \IM_{\omega_2})$; so let $(0, \xi) = -(\aaa a, \IM_\omega a)$ for some $a \in A_{\G_2}$.
%Then $((0, \c_{12}^*\xi), (0, a)) \in L_1 \times_{\c_1}(A_{\G_1} \times A_{\G_2})$.
By the non-degeneracy of $L_1$, $((0, \c_{12}^*\xi), (0, a)) = ((\aaa b_1, \c_{11}^*\IM_{\omega_1}\c_{11*} b_1 - \c_{12}^* \IM_{\omega_2} \c_{12*} b_1), (\c_{11*} b_1, \c_{12*}b_1))$ for some $b_1 \in A_{\C_1}$.
It follows that $\xi = \IM_{\omega_2}\c_{12*} b_1$ for some $b_1 \in A_{\C_1}$ such that $\c_{11*}b_1 = 0$ and $\aaa b_1 = 0$.
Similarly, $\eta = -\IM_{\omega_2} \c_{22*} b_2$ for some $b_2 \in A_{\C_2}$ such that $\c_{23*} b_2 = 0$ and $\aaa b_2 = 0$.
By Lemma \ref{7wn1m20c}\ref{xul5itsk}, we have $0 = \sss^*\xi + \ttt^*\eta = i_{(\c_{12*}b_1)^L - (\c_{22*}b_2)^R}\omega_2 = 0$.
It follows that $(\c_{12*}b_1)^L - (\c_{22*}b_2)^R \in \ker \omega \cap \ker \sss_* \cap \ker \ttt_* = 0$.
Hence, $(b_1, b_2) \in \ker \aaa_\C \cap \ker \c$ and maps to $(\xi, \eta)$ under \eqref{hl1zyyry}.
\end{proof}

\begin{lemma}
Part \ref{k9kpqiad} of Theorem \ref{rt0qkos3} holds.
\end{lemma}

\begin{proof}
Cleanness of \eqref{rf57it5e} is equivalent to $R$ having constant rank, so $L$ is smooth by Lemma \ref{jbecqi1g}.
If \eqref{rf57it5e} is transverse and $L_i$ are strong, then $R^\circ = 0$ and $\ker \rho_{\C_i} \cap \ker \c_{i*} = 0$, so $\ker \rho_\C \cap \ker \c_* = 0$ by Lemma \ref{zidpz55t}.
\end{proof}

%\begin{remark}
%Example where transversality of $A_{\C_1} \to A_{\G_2} \leftarrow A_{\C_2}$ is necessary.
%\end{remark}

\section{Morita transfer of basic Dirac structures}
\label{eph7yukn}

We now review the notion of Morita equivalences and introduce a weak version suitable for applications to symplectic reduction and its generalizations.
We then explain how to transfer forms and Dirac structures under such equivalences.
This section is of independent interest.

\subsection{Morita equivalences}

\begin{definition}\label{xv4cb2ld}
A \defn{Morita morphism} %(see e.g.\ \cite{beh:04}, \cite[Section 2.2]{lau-sti-xu:09}, \cite[Section 2]{beh-xu:03}, or \cite[Definition 2.9]{beh-xu:11}) 
\cite{beh-xu:03,beh:04,lau-sti-xu:09,beh-xu:11} (also known as a \emph{surjective equivalence} \cite{hoy:13} or \emph{hypercover} \cite{zhu:09,cue-zhu:23})
is a morphism of Lie groupoids $f : \H \to \G$ such that $f\obj : \H\obj \to \G\obj$ is a surjective submersion and the diagram
\begin{equation}\label{59upu6zm}
\begin{tikzcd}
\H\arr \arrow{d}{(\sss,\ttt)} \arrow{r} & \G\arr \arrow{d}{(\sss,\ttt)} \\
\H\obj \times \H\obj \arrow{r} & \G\obj \times \G\obj
\end{tikzcd}
\end{equation}
is cartesian, i.e.\ $\H$ is isomorphic to the pullback groupoid $(f\obj)^*\G$.
A \defn{Morita equivalence} between Lie groupoids $\G_1$ and $\G_2$ is a span
\[
\begin{tikzcd}[column sep={4em,between origins},row sep={2em,between origins}]
& \H \arrow{dl} \arrow{dr} & \\
\G_1 & & \G_2
\end{tikzcd}
\]
of Morita morphisms for some Lie groupoid $\H$.
\end{definition}

\begin{remark}\label{nhqq70ch}
There is an equivalent way of defining Morita equivalences using the weaker notion of essential equivalences.
An \defn{essential equivalence} \cite{met:03} (also known as a \emph{Morita map} \cite{hoy-ort:20,hoy-fer:19} or \emph{weak equivalence} \cite{moe-mrc:03,hoy:13}) is a morphism of Lie groupoids $f : \H \to \G$ such that $\ttt \circ \pr_{\G\arr} : \H\obj \times_{f\obj,\sss} \G\arr \to \G\obj$ is a surjective submersion and \eqref{59upu6zm} is cartesian.
Morita morphisms are essential equivalences.
Conversely, if $f : \H \to \G$ is an essential equivalence, then there is a Morita equivalence
\[
\begin{tikzcd}[row sep=1em]%[column sep={5em,between origins},row sep={2.5em,between origins}]
& (\ttt \circ \pr_{\G\arr})^*\G \arrow{dl} \arrow{dr} & \\
\H \cong (f\obj)^*\G & & \G
\end{tikzcd}
\]
in the sense of Definition \ref{xv4cb2ld}, where the first map is given by
\[
(\H\obj \times_{\G\obj} \G\arr) \times_{\G\obj} \G\arr \times_{\G\obj} (\H\obj \times_{\G\obj} \G\arr)
\too
\H\obj \times_{\G\obj} \G \times_{\G\obj} \H\obj,
\quad
((x, a), g, (y, b)) \mtoo (x, b^{-1}ga, y).
\]
Hence, both notions produce the same equivalence relation on Lie groupoids.
%(which still implies that $(f\obj)^*\G$ is well-defined); this is sometimes called \emph{essential surjectivity}.
% also known as surjective equivalence \cite{hoy:13}
\end{remark}

\begin{remark}
Another equivalent approach to Morita equivalences is via bibundles \cite{ler:10,hil-ska:87}.
We will not follow this approach, but note that the equivalence between Morita morphisms and bibundles is proved, for example, in \cite[Proposition 2.4]{lau-sti-xu:09} and \cite[Theorem 2.2]{beh-xu:11}, and the equivalence between essential equivalences and bibundles is proved in \cite{ler:10}.
\end{remark}

%The idea behind weak Morita mosphisms is that this is equivalent to the condition that $[N/\H] \to [M/\G]$ is an isomorphism

%(sometimes called Hilsum--Skandalis bibundles \cite{ler:10} \cite{hil-ska:87}).
%The equivalence between Morita morphisms and bibundles is standard, and is proved in \cite[Proposition 2.4]{lau-sti-xu:09} and \cite[Theorem 2.2]{beh-xu:11}.
%The equivalence between weak Morita morphisms and bibundles is proved in \cite{ler:10}. of stacks \cite[Proposition 60]{met:03}.

Morita equivalent Lie groupoids have homeomorphic orbit spaces \cite[Theorem 4.3.1]{hoy:13}.
In fact, as pointed out by del Hoyo and Fernandes \cite[Proposition 6.1.1]{hoy-fer:19}, the Morita equivalence class of a Lie groupoid can be thought of as an ``enhanced'' version of the orbit space, also encoding information about the normal representation.
See also \cite{bur-hoy:23} and the reference therein.
Morita equivalent classes of Lie groupoids are also in natural one-to-one correspondence with isomorphism classes of differentiable stacks \cite{beh-xu:03,beh-xu:11,ler:10}; we will return to this in \S\ref{dy26su68}.

On the other hand, if $\G$ is a Lie groupoid whose orbit space $Q = \G\obj/\G\arr$ is a manifold, then $\G$ and $Q$ are Morita equivalent (where $Q$ is viewed as the trivial Lie groupoid $Q \tto Q$) if and only if $\G$ has trivial isotropy groups.
%Indeed, the obvious morphism $\G \to Q$ is not a Morita morphism (nor an essential equivalence) if $\G$ has non-trivial isotropy groups.
This is problematic for applications to symplectic reduction and its generalization since we want, for example, an action groupoid $G \ltimes \mu^{-1}(0)$ of a Hamiltonian $G$-space with moment map $\mu$ to be equivalent to the quotient $\mu^{-1}(0)/G$ so that $0$-shifted symplectic structures transfer to ordinary symplectic structures, even if the $G$-action is not free.
For this reason, we introduce the following.

\begin{definition}
A \defn{weak Morita morphism} is a morphism of Lie groupoids $f : \H \to \G$ such that $f\obj : \H\obj \to \G\obj$ and $\H\arr \to (f\obj)^*\G\arr$ are surjective submersions.
\end{definition}

In other words, if $f\obj : \H\obj \to \G\obj$ is a surjective submersion, then $f$ is a Morita morphism if the induced map $\H\arr \to (f\obj)^*\G\arr$ is an isomorphism and it is a weak Morita morphism if this map is only a surjective submersion.
The following basic observation will be useful.

\begin{lemma}\label{htlesxzs}
Let $f : \H \to \G$ be a weak Morita morphism.
For every $v \in T\H\obj$ and $a \in A_\G$ such that $f_* v = \aaa a$, there exists $b \in A_\H$ such that $f_*b = a$ and $\aaa b = v$.
If $f$ is a Morita morphism, the element $b$ is unique.
\end{lemma}

\begin{proof}
The Lie algebroid of the pullback of $\G$ by $f\obj$ is $T\H\obj \times_{T\G\obj} A_\G$ and the corresponding Lie algebroid morphism is given by $A_\H \to T\H\obj \times_{T\G\obj} A_\G$, $b \mto (\aaa b, f_*b)$.
The latter is surjective if $f$ is a weak Morita morphism and bijective if $f$ is a Morita morphism.
\end{proof}

%\begin{lemma}[Pullback by Morita morphisms]\label{ijs66i1n}
%Consider Lie groupoid morphisms
%\[
%\begin{tikzcd}
%& \C_2 \arrow{d} \\
%\C_1 \arrow{r}{\simeq} & \G
%\end{tikzcd}
%\]
%where $\simeq$ is a Morita morphism.
%Then the homotopy fibre product
%\[
%\begin{tikzcd}
%\C_1 \htimes_\G \C_2 \arrow{d} \arrow{r}{\simeq} & \C_2 \arrow{d} \\
%\C_1 \arrow{r}{\simeq} & \G
%\end{tikzcd}
%\]
%exists, and the top arrow $\simeq$ is also a Morita morphism.
%\end{lemma}
%
%\begin{proof}
%See \cite[Proposition 8.2]{bur-hoy:23} or \cite[Proposition 4.4.4]{hoy:13} or \cite[Proposition 9.22]{wat:22b}.
%\end{proof}
%

\subsection{Basic forms}

We now explain how to transfer differential forms through weak Morita morphisms.

\begin{definition}
A \defn{basic form} on a Lie groupoid $\G$ is a differential form $\beta$ on $\G\obj$ such that $\sss^*\beta = \ttt^*\beta$.
\end{definition}

The following result appears in \cite[Corollary 1.3]{wat:22} in the special case of a proper Lie groupoid.
If $\G$ has trivial isotropy groups, it also follows from Morita invariance of basic forms (see e.g.\ \cite[Proposition 5.3.12(ii) and Remark 5.3.16]{hof-sja:21} or \cite[Proposition 8.3]{pfl-pos-tan:14}).

\begin{lemma}\label{6sfaxsof}
Let $\G$ be a Lie groupoid whose orbit space $\G\obj / \G\arr$ has the structure of a smooth manifold such that the quotient map $\pi : \G\obj \to \G\obj / \G\arr$ is a smooth submersion.
Then a differential form $\beta$ on $\G\obj$ is basic if and only if $\beta = \pi^*\alpha$ for some $\alpha$ on $\G\obj / \G\arr$.
%i.e.\ we have an exact sequence
%\[
%\begin{tikzcd}
%\Omega^{\bullet}(\G\obj / \G\arr) \arrow{r}{\pi^*} 
%& \Omega^{\bullet}(\G\obj) \arrow{r}{\sss^* - \ttt^*}
%& \Omega^{\bullet}(\G\arr).
%\end{tikzcd}
%\]
\end{lemma}

\begin{proof}
Let $\beta$ be a basic $k$-form on $\G$ and let $Q \coloneqq \G\obj/\G\arr$.
By taking local sections $\sigma_i : U_i \to \G\obj$ of $\pi$ and restricting $\G$ to $\pi^{-1}(U_i)$, we may reduce to the case where there is a global section $\sigma : Q \to \G\obj$.
It then suffices to show that $\pi^*\sigma^*\beta = \beta$.
To this end, we first observe that the map
\begin{equation}\label{lzaj53oo}
\ttt \circ \pr_{\G\arr} : Q \times_{\sigma,\sss} \G\arr \too \G\obj
\end{equation}
is a surjective submersion (i.e.\ $Q \to \G$ is a weak Morita morphism).
Surjectivity follows from the fact that for all $p \in \G\obj$, we have $\pi(p) = \pi(\sigma(\pi(p)))$, so there exists $g \in \G\arr$ such that $\sss(g) = \sigma(\pi(p))$ and $\ttt(g) = p$.
For submersivity, we show that the nullity of its differential is $\dim(Q \times_{\sigma,\sss} \G\arr) - \dim \G\obj = \dim Q$.
Using left translations, the kernel of the differential at $(q, g)$ is isomorphic to $T_qQ \times_{\sigma_*, \aaa} A_\G|_{\sss(g)}$.
Since $\im \aaa = \ker \pi_*$, we have $T_{\sss(g)}\G\obj = \im \sigma_* \oplus \im \aaa$, and hence $T_qQ \times_{\sigma_*, \aaa} A_\G|_{\sss(g)} = \dim Q + \rk A_\G - \dim \G\obj = \dim Q$.
Hence, \eqref{lzaj53oo} is a surjective submersion.
Now, let $v \in (T\G\obj)^{\oplus k}$.
Then there exists $(u, w) \in T(Q \times_{\sigma,\sss} \G\arr)^{\oplus k}$ such that $\ttt_* w = v$.
It follows that $\beta(v) = \ttt^*\beta(w) = \sss^*\beta(w) = \sigma^*\beta(u) = \pi^*\sigma^*\beta(v)$.
\end{proof}

We now generalize the previous result to any weak Morita morphism.

\begin{proposition}\label{m9pdmd75}
Let $f : \H \to \G$ be a weak Morita morphism.
For every basic form $\beta$ on $\H$, there is a unique basic form $\alpha$ on $\G$ such that $\beta = f^*\alpha$.
%More generally, if $\beta$ is a differential form on $\H\obj$ such that $\ttt^*\beta = \sss^*\beta + f^*\omega$ for some multiplicative form $\omega$ on $\G$, then there is a unique form $\alpha$ on $\G\obj$ such that $f^*\alpha = \beta$.
\end{proposition}

\begin{proof}
Consider the \emph{submersion groupoid} $\mathcal{B} \coloneqq (\H\obj \times_{\G\obj} \H\obj \tto \H\obj)$ associated with the submersion $f\obj : \H\obj \to \G\obj$ \cite[Example 4.2]{bur-hoy:23} (also known as the \emph{banal groupoid} \cite{beh:04}).
%Denote its source and target maps by $\sss_{\mathcal{B}}$ and $\ttt_{\mathcal{B}}$ to distinguish them from those of $\H$.
Let $g : \mathcal{P} \to \G$ be the pullback of $\G$ by $f\obj : \H\obj \to \G\obj$.
Then there is a Lie groupoid morphism $\imath : \mathcal{B} \to \mathcal{P}$ covering the identity map on $\H\obj$ such that the composition $\jmath \coloneqq g \circ \imath : \mathcal{B} \to \G$ lies in the identity section of $\G$.
Also, by the definition of weak Morita morphisms, there is a surjective submersion $k : \H \to \P$ covering the identity map on $\H\obj$ and whose composition with $g$ is the map $f$.
In particular, we have a commutative diagram
\begin{equation}\label{dhi82hlj}
\begin{tikzcd}
\H \arrow[swap]{d}{k} \arrow{dr}{f} & \\
\mathcal{P} \arrow{r}{g} & \G \\
\mathcal{B}. \arrow{u}{\imath} \arrow[swap]{ur}{\jmath} &
\end{tikzcd}
\end{equation}
Now, $k^*\sss_{\mathcal{P}}^*\beta = \sss_\H^*\beta = \ttt_\H^*\beta = k^*\ttt_{\mathcal{P}}^*\beta$ and $k$ is a submersion, so $\sss_{\mathcal{P}}^*\beta = \ttt_{\mathcal{P}}^*\beta$.
It follows that $\sss_{\mathcal{B}}^*\beta = \imath^* \sss_{\mathcal{P}}^*\beta = \imath^* \ttt_{\mathcal{P}}^*\beta = \ttt_{\mathcal{B}}^*\beta$.
By Lemma \ref{6sfaxsof} applied to $\mathcal{B}$, there is a (unique) form $\alpha$ on $\G\obj$ such that $f^*\alpha = \beta$.
We have $f^*\sss^*_\G\alpha = \sss^*_\H\beta = \ttt^*_\H\beta = f^*\ttt^*_\G\alpha$ and $f$ is a submersion, so $\sss^*\alpha = \ttt^*\alpha$.
\end{proof}

\subsection{Basic Dirac structures}

We now give analogues of the previous results where differential forms are replaced by Dirac structures.

\begin{lemma}\label{7rjqnb2r}
Let $\G$ be a Lie groupoid whose orbit space $\G\obj / \G\arr$ has the structure of a smooth manifold such that the quotient map $\pi : \G\obj \to \G\obj / \G\arr$ is a smooth submersion.
Let $(L, \phi)$ be a Dirac structure on $\G\obj$ such that $\sss^*L = \ttt^* L$ and $\sss^*\phi = \ttt^*\phi$.
Then $\pi_*L$ is a Dirac structure on $\G\obj/\G\arr$ whose background 3-form $\zeta$ is characterized by $\pi^*\zeta = \phi$.
Moreover, $\pi^*\pi_*L = L$.
\end{lemma}

\begin{proof}
To show that $\pi_*L$ is a well-defined almost Dirac structure on $\G\obj/\G\arr$, it suffices to show (see e.g.\ \cite[Proposition 1.13]{bur:13}) that
\begin{enumerate}[label={(\arabic*)}]
\item
\label{2554a73s}
$L$ is $\pi$-invariant, i.e.\ $\pi_* L_x = \pi_* L_y$ for all $x, y \in \G\obj$ such that $\pi(x) = \pi(y)$, and
\item
\label{94odhqki}
$\ker \pi_* \cap \ker L$ has constant rank.
\end{enumerate}
To show \ref{2554a73s}, let $g \in \G\arr$ and let $x = \sss(g)$ and $y = \ttt(g)$.
Note that $L_x = \sss_* (\sss^*L)_g$ and $L_y = \ttt_* (\ttt^*L)_g$.
Since $\pi \circ \sss = \pi \circ \ttt$, it follows that $\pi_*L_x = \pi_*\sss_* (\sss^*L)_g = \pi_*\ttt_*(\ttt^*L)_g = \pi_*L_y$.
To show \ref{94odhqki}, we claim that $\ker \pi_* \s \ker L$.
Let $v \in \ker \pi_* = \im \aaa$, so that $v = \ttt_*w$ for some $w \in \ker \sss_*$.
Then $(w, 0) \in \sss^*L = \ttt^*L$, so $(w, 0) = (w, \ttt^*\alpha)$ for some $\alpha$ such that $(\ttt_* w, \alpha) \in L$.
It follows that $(v, 0) = (\ttt_* w, \alpha) \in L$, i.e.\ $v \in \ker L$.
Hence, $\ker \pi_* \cap \ker L = \ker \pi_*$ has constant rank, and $\pi_*L$ is an almost Dirac structure.
By Lemma \ref{6sfaxsof}, $\phi = \pi^*\zeta$ for a unique closed 3-form $\zeta$ on $\G\obj/\G\arr$.
The integrability of $\pi_*L$ with respect to $\zeta$ then follows by the same proof as in \cite[Proposition 1.13]{bur:13}.

For the last part, we first show that $\ker \pi_* \s \ker L$.
Let $v \in \ker \pi_*$, i.e.\ $v = \aaa a$ for some $a \in A_\G$.
Now, $\sss_* a = 0$, so $(a, 0) \in \sss^*L = \ttt^*L$ and hence $(\ttt_* a, 0) \in L$. 
It follows that $v = \ttt_* a \in \ker L$.
Hence, $\ker \pi_* \s \ker L$ and, taking $\perp$ on both sides, we have $L \s T\G\obj \oplus \im \pi^*$.
This implies that $L \s \pi^*\pi_*L$ and, by dimension count, equality holds.
%By dimension count, we have $\pi^*\pi_*L = L$ if and only if $L \s \pi^*\pi_*L$.
%Note that for $(v, \alpha) \in L$, we have $(v, \alpha) \in \pi^*\pi_*L$ if and only if $\alpha \in \im \pi^*$.
%Hence, $L \s \pi^*\pi_*L$ if and only if $L \s T\G\obj \oplus \im \pi^*$.
%Taking $\perp$ on both sides, we get $\ker \pi_* \s \ker L$.
\end{proof}

\begin{proposition}\label{9u2g4w2h}
Let $f : \H \to \G$ be a weak Morita morphism and $(L, \phi)$ a Dirac structure on $\H\obj$ such that
\[
\ttt^*\phi = \sss^*\phi + f^*\omega_1
\quad\text{and}\quad
\ttt^*L = \sss^*L + \Gamma_{f^*\omega_2}
\]
for some multiplicative forms $\omega_1$ and $\omega_2$ on $\G\arr$.
Then $f_*L$ is a Dirac structure on $\G\obj$ whose background $3$-form $\zeta$ is characterized by $f^*\zeta = \phi$.
Moreover, $f^*f_* L = L$.% if and only if $\ker f_* \s \ker L$.
\end{proposition}

\begin{proof}
We retain the notation in the proof of Proposition \ref{m9pdmd75}.
We have $k^*\ttt_{\mathcal{P}}^*L = \ttt_\H^*L = \sss_\H^*L + \Gamma_{f^*\omega_2} = k^*(\sss_{\mathcal{P}}^*L + \Gamma_{g^*\omega_2})$, and $k$ is a submersion, so $\ttt_{\mathcal{P}}^*L = \sss_{\mathcal{P}}^*L + \Gamma_{g^*\omega_2}$.
Then $\ttt_{\mathcal{B}}^*L = \imath^* \ttt^*_{\mathcal{P}} L = \imath^*(\sss^*_{\mathcal{P}} L + \Gamma_{g^*\omega_2}) = \sss_{\mathcal{B}}^*L + \Gamma_{\jmath^*\omega_2} = \sss_{\mathcal{B}}^*L$, since $\jmath$ has its image in the identity section and $\uuu^*\omega_2 = 0$ for a multiplicative form \cite[Lemma 3.1(i)]{bur-cra-wei-zhu:04}.
Similarly, $k^*\ttt_{\mathcal{P}}^*\phi = \ttt_\H^*\phi = \sss_\H^*\phi + f^*\omega_1 = k^*(\sss_{\mathcal{P}}^*\phi + g^*\omega_1)$ so $\ttt_{\mathcal{P}}^*\phi = \sss_{\mathcal{P}}^*\phi + g^*\omega_1$ and hence $\ttt_{\mathcal{B}}^*\phi = \imath^*\ttt^*_{\mathcal{P}} \phi = \imath^*(\sss^*_{\mathcal{P}}\phi + g^*\omega_1) = \sss_{\mathcal{B}}^*\phi + \jmath^*\omega_1 = \sss_{\mathcal{B}}^*\phi$.
The result then follows by Lemma \ref{7rjqnb2r} applied to $\mathcal{B}$.
\end{proof}

We can now prove the relationship between 0-shifted Poisson structures and ordinary Poisson structures mentioned in \S\ref{zaduohbe}.

\begin{corollary}\label{s7wgomy8}
Let $\G$ be a Lie groupoid whose orbit space $\G\obj/\G\arr$ has the structure of a smooth manifold such that the quotient map $\pi : \G\obj \to \G\obj/\G\arr$ is a smooth submersion.
Then there is a one-to-one correspondence between 0-shifted Poisson structures on $\G$ and Poisson structures on $\G\obj/\G\arr$ via pushfoward and pullback.
\end{corollary}

\begin{proof}
Let $L$ be a 0-shifted Poisson structure on $\G$.
By Lemma \ref{7rjqnb2r}, $\pi_*L$ is a Dirac structure on $\G\obj/\G\arr$ such that $\pi^*\pi_*L = L$.
To show that $\pi_*L$ is the graph of a Poisson structure, it suffices to show that $\ker \pi_* L = 0$.
Let $v \in \ker \pi_*L$.
Then $v = \pi_*w$ for some $w \in \ker L$.
But $\ker L = \im \aaa = \ker \pi_*$, so $v = 0$.
Conversely, if $\Gamma_\sigma$ is the graph of a Poisson structure on $\G\obj/\G\arr$ then $\pi^*\Gamma_\sigma$ is a 0-shifted Poisson structure.
Since $\pi^*\pi_*L = L$, these two operations are inverse to each other.
\end{proof}

\section{Morita transfer of quasi-symplectic structures}
\label{1vfh5k59}

We now recall how to transfer quasi-symplectic structures on Morita equivalent Lie groupoids.
The results in this section are mostly due to Xu \cite{xu:04}, with slight refinements.
We introduce them as they form the basis of our definition of 1-shifted symplectic structures on differentiable stacks in \S\ref{dy26su68} and will serve as a model for our discussion on Morita transfer of 1-shifted coisotropic structures.

Let us first recall the notion of symplectomorphisms of quasi-symplectic groupoids when viewed as presentations of 1-shifted symplectic stacks.
It is the direct analogue of the fact a map between symplectic manifolds is a symplectomorphism if and only if its graph is Lagrangian.

\begin{definition}\label{ie53b839}
A \defn{symplectic Morita equivalence} (see e.g.\ \cite[Definition 2.31]{cue-zhu:23}) between quasi-symplectic groupoids $\G_1$ and $\G_2$ is a Morita equivalence 
\[
\begin{tikzcd}[row sep={2em,between origins},column sep={4em,between origins}]
& \L \arrow[swap]{dl} \arrow{dr} & \\
\G_1 & & \G_2
\end{tikzcd}
\]
of the underlying Lie groupoids together with a 1-shifted Lagrangian structure on $\L \to \G_1 \times \G_2^-$. 
\end{definition}

\begin{remark}
Another popular definition is in terms of Hamiltonian bimodules \cite{xu:04} (see also \cite[Appendix A]{ale-mei:22}).
Both approaches are equivalent \cite[Theorem 2.37]{cue-zhu:23}.
\end{remark}

It is useful to observe that there is a redundancy in the definition, namely, the non-degeneracy condition of the 1-shifted Lagrangian structure is automatic:

\begin{lemma}\label{bczyu6nt}
Let $(\G_1, \omega_1, \phi_1)$ and $(\G_2, \omega_2, \phi_2)$ be quasi-symplectic groupoids and
\begin{equation}\label{zfkea3ml}
\begin{tikzcd}[row sep={2em,between origins},column sep={4em,between origins}]
& \L \arrow[swap]{dl}{\varphi_1} \arrow{dr}{\varphi_2} & \\
\G_1 & & \G_2
\end{tikzcd}
\end{equation}
a Morita equivalence of Lie groupoids together with a $2$-form $\gamma$ on $\L\obj$ such that
\begin{equation}\label{kvzg1hqz}
\varphi_1^*\omega_1 - \varphi_2^*\omega_2 = \ttt^*\gamma - \sss^*\gamma
\quad\text{and}\quad
\varphi_1^*\phi_1 - \varphi_2^*\phi_2 = -d\gamma.
\end{equation}
Then $\gamma$ is a 1-shifted Lagrangian structure on $\L \to \G_1 \times \G_2^{-}$ and hence \eqref{zfkea3ml} is a symplectic Morita equivalence.
\end{lemma}

\begin{proof}
By Lemma \ref{t7rjqxe5}, the non-degeneracy condition amounts to the statement that the map
\begin{align}
A_\L &\too \{(v, a_1, a_2) \in T\L\obj \oplus \varphi_1^*A_{\G_1} \oplus \varphi_2^*A_{\G_2} : \varphi_{1*} v = \aaa a_1, \varphi_{2*} v = \aaa a_2, i_v\gamma = \varphi_1^*\IM_{\omega_1}a_1 - \varphi_2^*\IM_{\omega_2}a_2\} \nonumber \\
\ell &\mtoo (\aaa \ell, \varphi_1 \ell, \varphi_2 \ell) \label{8pg0mp7q}
\end{align}
is bijective.
Let $(v, a_1, a_2)$ be in the codomain.
By Lemma \ref{htlesxzs}, there exists $\ell_1, \ell_2 \in A_\L$ such that $\aaa \ell_i = v$ and $\varphi_{i*} \ell_i = a_i$ for $i = 1, 2$.
We then get
\begin{equation}\label{ud2pxcmn}
i_v\gamma = \IM_{\varphi_1^*\omega_1} \ell_1 - \IM_{\varphi_2^*\omega_2} \ell_2.
\end{equation}
Also, by the first part of \eqref{kvzg1hqz}, we have
\begin{equation}\label{ol695sgs}
\IM_{\varphi_1^*\omega_1}\ell - \IM_{\varphi_2^*\omega_2}\ell = i_{\aaa \ell} \gamma
\end{equation}
for all $\ell \in A_\L$.
Applying \eqref{ol695sgs} to $\ell = \ell_1$ and subtracting to \eqref{ud2pxcmn}, we get $\IM_{\varphi_2^*\omega_2}(\ell_1 - \ell_2) = 0$.
Since also $\aaa(\ell_1 - \ell_2) = 0$ and $(\varphi_2^*\omega_2, \varphi_2^*\phi_2)$ is a quasi-symplectic structure \cite[Proposition 4.8]{xu:04}, we conclude from the non-degeneracy condition of $(\varphi_2^*\omega_2, \varphi_2^*\phi_2)$ that $\ell_1 = \ell_2$ (see Lemma \ref{7wn1m20c}\ref{rwp7208a}).
Hence, \eqref{8pg0mp7q} is surjective.
Injectivity follows from the fact that $\varphi_i$ are Morita morphisms (see Lemma \ref{htlesxzs}).
\end{proof}

%\begin{remark}
%If $(\G_i, \omega_i) \tto (M_i, \phi_i)$ for $i = 1, 2$ are quasi-symplectic groupoids and $(f, \mu) : (\G_1 \tto M_1) \to (\G_2 \tto M_2)$ is a Morita morphism such that $f^*\omega_2 = \omega_1$ and $\mu^*\phi_2 = \phi_1$, then they are Morita equivalence via the 1-shifted Lagrangian $\L = \G_1$ with 2-form $\gamma = 0$.
%(These are the \emph{stric morphisms} of \cite[Example 2.34]{cue-zhu:23}.)
%\end{remark}
%

%\begin{remark}
%This makes sense given that a map between symplectic manifolds is a symplectomorphism if and only if its graph is Lagrangian.
%\end{remark}
%
The \defn{gauge transformation} of a quasi-symplectic structure $(\omega, \phi)$ on a Lie groupoid $\G$ by a $2$-form $\gamma$ on $\G\obj$ is the new quasi-symplectic structure $(\omega + \sss^*\gamma - \ttt^*\gamma, \phi + d\gamma)$ \cite[\S4.1]{xu:04}.
%It is always a quasi-symplectic structure \cite[Proposition 4.6]{xu:04}.
We say that two quasi-symplectic structures are \defn{gauge equivalent} if they are related by a gauge transformation.
%Note that if $(\omega_1, \phi_1)$ and $(\omega_2, \phi_2)$ are gauge equivalent via a 2-form $\gamma$, then they are symplectically Morita equivalent in the sense of Definition \ref{ie53b839} via $(\G, \omega_1, \phi_1) \leftarrow (\G, \gamma) \to (\G, \omega_2, \phi_2)$.
We denote the set of gauge equivalence classes of quasi-symplectic structures on $\G$ by $\operatorname{qsymp}(\G)$.
%The converse also holds, as can be seen from the following result.

%It is also useful to observe that if $(\G, \omega, \phi)$ is a quasi-symplectic groupoid and $f : \H \to \G$ is a Morita morphism, then $(\H, f^*\omega, f^*\phi)$ is also a quasi-symplectic groupoid \cite[Proposition 4.8]{xu:04}.

\begin{theorem}[{Xu \cite{xu:04}; see also \cite[Theorem 3.12]{bon-cic-lau-xu:22}}]\label{jpmgv5l0}
Let
\begin{equation}\label{vwwcl5o8}
\begin{tikzcd}[row sep={2em,between origins},column sep={4em,between origins}]
& \L \arrow{dl} \arrow{dr} & \\
\G_1 & & \G_2
\end{tikzcd}
\end{equation}
be a Morita equivalence of Lie groupoids.
For every quasi-symplectic structure $(\omega_1, \phi_1)$ on $\G_1$, there is a unique (up to gauge transformation) quasi-symplectic structure $(\omega_2, \phi_2)$ on $\G_2$ such that \eqref{vwwcl5o8} has the structure of a symplectic Morita equivalence.
Moreover, this gives a bijection
\begin{equation}\label{a17x97sf}
\begin{tikzcd}
\operatorname{qsymp}(\G_1) \arrow{r}{\L} & \operatorname{qsymp}(\G_2).
\end{tikzcd}
\end{equation}
\end{theorem}

\begin{proof}
Let us explain how this statement can be inferred from \cite{xu:04}, since it does not appear explicitly in this form.
Let $\varphi_i : \L \to \G_i$ be the two maps in \eqref{vwwcl5o8}.
The compatibility condition of the quasi-symplectic structure $(\omega_1, \phi_1)$ (i.e.\ that $\omega_1$ is multiplicative, $d\omega_1 = \sss^*\phi_1 - \ttt^*\phi_1$, and $d\phi_1 = 0$) is the statement that it forms a 3-cocycle in the total de Rham complex of $\G_2$ \cite[\S2.1]{xu:04}.
Hence, it defines an element $[(\omega_1, \phi_1)]$ of the Lie groupoid cohomology $H^3(\G_1)$.
Since the maps in \eqref{vwwcl5o8} are Morita morphisms, they induce a diagram
\begin{equation}\label{e98tbiki}
\begin{tikzcd}[row sep={3em,between origins},column sep={5em,between origins}]
& H^3(\L) & \\
H^3(\G_1) \arrow{ur}{\varphi_1^*} & & H^3(\G_2) \arrow[swap]{ul}{\varphi_2^*}
\end{tikzcd}
\end{equation}
of isomorphisms of vector spaces.
It follows that there exists $[(\omega_2, \phi_2)] \in H^3(\G_2)$ such that $\varphi_1^*[(\omega_1, \phi_1)] = \varphi_2^*[(\omega_2, \phi_2)]$, i.e.\
\begin{equation}\label{wa7ybmiz}
(\varphi_2^*\omega_2, \varphi_2^*\phi_2) = (\varphi_1^*\omega_1 + \sss^*\gamma - \ttt^*\gamma, \varphi_1^*\phi_1 + d\gamma),
\end{equation}
for some $2$-form $\gamma$ on $\L\obj$.
By \cite[Proposition 4.8]{xu:04}, $(\varphi_1^*\omega_1, \varphi_1^*\phi_1)$ is a quasi-symplectic structure on $\L$.
By \cite[Proposition 4.6]{xu:04}, its gauge transformation by $\gamma$, which is $(\varphi_2^*\omega_2, \varphi_2^*\phi_2)$, is a quasi-symplectic structure on $\L$.
To see that $(\omega_2, \phi_2)$ is a quasi-symplectic structure on $\G_2$, let $v \in \ker \omega_2 \cap \ker \sss_* \cap \ker \ttt_*$.
Since $\varphi_2$ is a Morita morphism, there exists $w \in T\L$ such that $\varphi_{2*} w = v$ and $\sss_* w = \ttt_* w = 0$ (Lemma \ref{htlesxzs}).
Then $i_w \varphi_2^*\omega_2 = \varphi_2^* i_v\omega_2 = 0$, so $w = 0$ by the non-degeneracy of $(\varphi_2^*\omega_2, \varphi_2^*\phi_2)$, and hence $v = 0$.
The condition that $\dim \G_2\arr = 2 \dim \G_2\obj$ follows from the fact that $2 \dim \G_2\obj - \dim \G_2\arr$ is a Morita invariant (the ``dimension'' of the orbit space).
Therefore, $(\omega_2, \phi_2)$ is a quasi-symplectic structure.
By \eqref{wa7ybmiz} and Lemma \ref{bczyu6nt}, $\gamma$ is a 1-shifted Lagrangian structure on $\L \to \G_1 \times \G_2^-$.

To show that $(\omega_2, \phi_2)$ is unique up to gauge transformation, let $(\omega_2', \phi_2')$ be another quasi-symplectic structure on $\G_2$ such that \eqref{vwwcl5o8} has a 1-shifted Lagrangian structure.
It follows that $[(\omega_2, \phi_2)]$ and $[(\omega_2', \phi_2')]$ map to the same element of $H^3(\L)$ under \eqref{e98tbiki}.
They are then cohomologous, i.e.\ differ by a gauge transformation.
\end{proof}

Next, we study the extent to which the transfer map \eqref{a17x97sf} depends on the Morita equivalence $\L$.
We say that two Morita equivalences $\G_1 \leftarrow \L \to \G_2$ and $\G_1 \leftarrow \L' \to \G_2$ are \defn{equivalent} if there exists a 2-commutative diagram of Morita morphisms
\begin{equation}\label{poby95jm}
\begin{tikzcd}[row sep={4em,between origins},column sep={4em,between origins}]
& \L \arrow{dl} \arrow{dr} & \\
\G_1 & \widehat{\L} \arrow{u} \arrow{d} & \G_2 \\
& \L'. \arrow{ul} \arrow{ur} &
\end{tikzcd}
\end{equation}

In other words, they are equivalent if they present the same morphism of stacks; see \S\ref{dy26su68}.

\begin{proposition}\label{0cc6w7hy}
The transfer map \eqref{a17x97sf} depends only on the equivalence class of \eqref{vwwcl5o8}. %\eqref{vwwcl5o8}
%only up to equivalences. % as in \eqref{poby95jm}.
\end{proposition}

\begin{proof}
Let $(\omega_1, \phi_1)$ be a quasi-symplectic structure on $\G_1$, let $(\omega_2, \phi_2)$ be the quasi-symplectic structure (up to gauge equivalence) on $\G_2$ obtained from $\G_1 \leftarrow \L \to \G_2$, and $(\omega_2', \phi_2')$ the one obtained from $\G_1 \leftarrow \L' \to \G_2$ as in \eqref{poby95jm}.
Recall that for a Lie groupoid morphism $f : \H \to \G$, the induced map on Lie groupoid cohomology $f^* : H^{\bullet}(\G) \to H^{\bullet}(\H)$ is invariant under natural transformation of $f$ (see e.g.\ \cite{beh:04}).
Hence, passing to cohomology, \eqref{poby95jm} gives a \emph{commutative} diagram of vector space isomorphisms
\begin{equation}\label{s99zulcs}
\begin{tikzcd}[row sep={5em,between origins},column sep={5em,between origins}]
& H^3(\L) \arrow[from=dl] \arrow[from=dr] & \\
H^3(\G_1) & H^3(\widehat{\L}) \arrow[from=u] \arrow[from=d] & H^3(\G_2) \\
& H^3(\L') \arrow[from=ul] \arrow[from=ur]. &
\end{tikzcd}
\end{equation}
As we recalled in the proof of Theorem \ref{jpmgv5l0}, $(\G_1, \omega_1, \phi_1) \leftarrow \L \to (\G_2, \omega_2, \phi_2)$ has the structure of a symplectic Morita equivalence if and only if $[(\omega_1, \phi_1)]$ and $[(\omega_2, \phi_2)]$ are mapped to the same element in $H^3(\L)$.
It follows that the cohomology classes $[(\omega_2, \phi_2)]$ and $[(\omega_2', \phi_2')]$ are mapped to the same element of $H^3(\widehat{\L})$ in \eqref{poby95jm} and hence are equal.
In other words, $(\omega_2, \phi_2)$ and $(\omega_2', \phi_2')$ are cohomologous, i.e.\ gauge equivalent \cite[\S4.1]{xu:04}.
\end{proof}

%\begin{corollary}
%Two quasi-symplectic structures $(\omega_1, \phi_1)$ and $(\omega_2, \phi_2)$ on a Lie groupoid $\G$ are gauge equivalent if and only if there is a symplectic Morita equivalence
%\[
%\begin{tikzcd}[row sep={2em,between origins},column sep={4em,between origins}]
%& \L \arrow[swap]{dl}{\varphi_1} \arrow{dr}{\varphi_2} & \\
%(\G, \omega_1, \phi_1) & & (\G, \omega_2, \phi_2)
%\end{tikzcd}
%\]
%such that $\varphi_1$ and $\varphi_2$ are homotopic.\qed
%\end{corollary}
%
%\begin{proof}
%If $(\omega_1, \phi_1)$ and $(\omega_2, \phi_2)$ are related by a gauge transformation by $\gamma$, then we can form a symplectic Morita equivalence by taking $\L = \G$ with the 1-shifted Lagrangian structure $\gamma$ and $\varphi_i = \Id$.
%The converse follows from the uniqueness part of Theorem \ref{jpmgv5l0}.
%\end{proof}

Suppose now that we have two Morita equivalences of Lie groupoids
\begin{equation}\label{q554p0i4}
\begin{tikzcd}[row sep={2em,between origins},column sep={4em,between origins}]
& \L_1 \arrow{dl} \arrow{dr} & & \L_2 \arrow{dl}\arrow{dr} & \\
\G_1 & & \G_2 & & \G_3.
\end{tikzcd}
\end{equation}
By taking the homotopy fibre product, we obtain a new Morita equivalence of Lie groupoids \cite[Proposition 4.4.4]{hoy:13}
\begin{equation}\label{wf31d5cr}
\begin{tikzcd}[row sep={3em,between origins},column sep={6em,between origins}]
& \L_1 \htimes_{\G_2} \L_2 \arrow{dl} \arrow{dr} & \\
\G_1 & & \G_3.
\end{tikzcd}
\end{equation}
Moreover, if \eqref{q554p0i4} are symplectic Morita equivalences, then so is \eqref{wf31d5cr}:

\begin{proposition}[{See also \cite[Proposition 2.33]{cue-zhu:23}}] \label{9j2iully}
Suppose that the Lie groupoids $\G_i$ in \eqref{q554p0i4} have quasi-symplectic structure $(\omega_i, \phi_i)$ and both equivalences in \eqref{q554p0i4} are symplectic Morita equivalences.
Then \eqref{wf31d5cr} is also a symplectic Morita equivalence.
%More precisely, the 1-shifted Lagrangian structure on $\L_1 \htimes_{\G_2} \L_2 \to \G_1 \times \G_3^-$ is given by $p_1^*\gamma_1 + p_2^*\gamma_2 - p_0^*\omega_2$, where $(p_1, p_0, p_2) : \L_1 \htimes_{\G_2} \L_2 \to \L_1 \times \G_2 \times \L_2$ are the natural projections and $\gamma_i$ are the 1-shifted Lagrangian structures on \eqref{q554p0i4}.
\end{proposition}

\begin{proof}
This follows from Theorem \ref{rt0qkos3}\ref{k9kpqiad}; the transversality of \eqref{rf57it5e} is automatic since $p_T : L_i \to T\L\obj_i$ and $\c_{i2*} : T\L\obj_i \to T\G\obj_2$ are surjective.
\end{proof}

In particular, the respective transfer maps fit into a commutative diagram
\begin{equation}\label{p8gy6x7h}
\begin{tikzcd}
\operatorname{qsymp}(\G_1) \arrow{r}{\L_1} \arrow[swap]{dr}{\L_1 \htimes_{\G_2} \L_2} & \operatorname{qsymp}(\G_2) \arrow{d}{\L_2} \\
& \operatorname{qsymp}(\G_3).
\end{tikzcd}
\end{equation}
%which commutes by Proposition \ref{9j2iully}.

%\begin{corollary}
%Diagram \eqref{p8gy6x7h} commutes.\qed
%\end{corollary}
%

\section{Morita transfer of 1-shifted coisotropic structures}
\label{4pvyvpb2}

In this section, we prove that 1-shifted coisotropics transfer through Morita morphisms.
We first state the transfer theorem in \S\ref{mvhqjvjk}, then recall the notion of adjoint representation up to homotopy in \S\ref{f73gy0ci} to derive some useful results, and prove the theorem in \S\ref{vazibeld}.

\subsection{The transfer theorem}
\label{mvhqjvjk}

Let us extend the notion of symplectic Morita equivalence of quasi-symplectic groupoids to account for 1-coisotropics on them.

\begin{definition}\label{w43yacl2}
Let $\c_1 : \C_1 \to \G_1$ and $\c_2 : \C_2 \to \G_2$ be morphisms of Lie groupoids.
\begin{enumerate}[label=\textup{(\arabic*)}]
\item\label{pz58funr}
A \defn{Morita equivalence} between $\c_1$ and $\c_2$ is a 2-commutative diagram of Lie groupoid morphisms
\begin{equation}\label{27zunnqp}
\begin{tikzcd}[column sep={5em,between origins},row sep={2.5em,between origins}]
& \K \arrow[swap]{dl}{\psi_1} \arrow{dd}{g} \arrow{dr}{\psi_2} & \\
\C_1 \arrow[swap]{dd}{\c_1} \arrow[Rightarrow,shorten=14pt,swap]{dr}{\theta_1} &  & \C_2 \arrow{dd}{\c_2} \arrow[Rightarrow,shorten=14pt]{dl}{\theta_2} \\
& \L \arrow{dl}{\varphi_1} \arrow[swap]{dr}{\varphi_2} & \\
\G_1 & & \G_2,
\end{tikzcd}
\end{equation}
where $\psi_i$ and $\varphi_i$ are Morita morphisms and $\theta_i : \c_i\psi_i \Rightarrow \varphi_ig$ are natural transformations.
%It is a \defn{weak Morita equivalence} if $\varphi_1$ and $\varphi_2$ are Morita morphisms and $\psi_1$ and $\psi_2$ are weak Morita morphisms.

\item\label{0dpv2tg2}
If $(\G_1, \omega_1, \phi_1)$ and $(\G_2, \omega_2, \phi_2)$ are quasi-symplectic groupoids, then a \defn{symplectic Morita equivalence} between $\c_1$ and $\c_2$ is a Morita equivalence as in \ref{pz58funr} together with a 1-shifted Lagrangian structure $\gamma$ on $\L \to \G_1 \times \G_2^-$.

\item\label{bx4cvb7v}
If $\c_1$ and $\c_2$ are endowed with 1-shifted coisotropic structures $L_1$ and $L_2$, then a \defn{coisotropic Morita equivalence} is a symplectic Morita equivalence as in \ref{0dpv2tg2}
%together with a closed basic 2-form $\beta$ on $\K\obj$ 
such that
\begin{equation}\label{3ttcx243}
\psi_2^*L_2 = \psi_1^*L_1 + \Gamma_\delta,
\end{equation}
where
\begin{equation}\label{h2qt67r9}
\delta \coloneqq %\beta
- g^*\gamma + \theta_1^*\omega_1 - \theta_2^*\omega_2,
\end{equation}
is called the \defn{connecting form}.
\end{enumerate}
We also define \defn{weak Morita equivalences}, \defn{weak symplectic Morita equivalences}, and \defn{weak coisotropic Morita equivalences} as above except that $\psi_1$ and $\psi_2$ are weak Morita morphisms.
\end{definition}

The goal of this section is to prove the following result.

\begin{theorem}%[Morita transfer of 1-shifted coisotropic structures]
\label{vazibeld}
Let $\c_1 : \C_1 \to \G_1$ and $\c_2 : \C_2 \to \G_2$ be morphisms of Lie groupoids over quasi-symplectic groupoids together with a weak symplectic Morita equivalence \eqref{27zunnqp}.
Then for every 1-shifted coisotropic structure $L_1$ on $c_1$, there is a unique 1-shifted coisotropic structure $L_2$ on $c_2$ such that \eqref{27zunnqp} is a weak coisotropic Morita equivalence.
Moreover, this gives a one-to-one correspondence between 1-shifted coisotropic structures on $\c_1$ and those on $\c_2$.
If \eqref{27zunnqp} is strong, this restricts to a one-to-one correspondence between strong 1-shifted coisotropic structures on $\c_1$ and those on $\c_2$.
%Moreover, this restricts to a bijection on the set of strong coisotropic structures.
%\[
%\begin{tikzcd}
%\{\text{1-shifted coisotropic structures on $\c_1$}\} 
%\arrow[leftrightarrow]{r}{1:1}
%&
%\{\text{1-shifted coisotropic structures on $\c_2$}\}.
%\end{tikzcd}
%\]
\end{theorem}

%\begin{remark}
%The proof shows that this extends to a one-to-one correspondence between weak 1-shifted coisotropic structures on $c_1$ and those on $c_2$.
%\end{remark}

\begin{remark}
If we start with a 1-coisotropic $\c : (\C, L) \to \G$ and a symplectic Morita equivalence $\G \leftarrow \L \to \widetilde{\G}$, then we can transfer $(\C, \c, L)$ to a 1-coisotropic on $\widetilde{\G}$.
Indeed, we can take, for example, $\widetilde{\C} \coloneqq \C \htimes_{\G} \L$ and the symplectic Morita equivalence
\[
\begin{tikzcd}[column sep={4em,between origins},row sep={2em,between origins}]
& \C \htimes_{\G} \L \arrow[swap]{dl} \arrow{dd} \arrow{dr} & \\
\C \arrow[swap]{dd} \arrow[Rightarrow,shorten=8pt,swap]{dr} &  & \widetilde{\C} \arrow{dd} \\
& \L \arrow{dl} \arrow[swap]{dr} & \\
\G & & \widetilde{\G}.
\end{tikzcd}
\]
By Theorem \ref{vazibeld}, there is a 1-shifted coisotropic structure $\widetilde{L}$ on $\widetilde{\C} \to \widetilde{\G}$.
Moreover, it is easy to see using the results of this paper that $(\widetilde{\C}, \widetilde{\c}, \widetilde{L})$ is unique up to gauge equivalences.
In other words, if $(\widetilde{\C}_1, \widetilde{\c}_1, \widetilde{L}_1)$ and $(\widetilde{\C}_2, \widetilde{\c}_2, \widetilde{L}_2)$ are two 1-coisotropics on $\widetilde{\G}$ with coisotropic Morita equivalences from $(\C, \c, L)$ covering $\G \leftarrow \L \to \widetilde{\G}$, then there is a 2-commutative diagram
\[
\begin{tikzcd}[column sep={3em,between origins},row sep={3em,between origins}]
& \K \arrow[swap]{dl}{\psi_1} \arrow{dr}{\psi_2} & \\
\C_1 \arrow[swap]{dr}{\c_1} \arrow[Rightarrow,shorten=13pt]{rr}{\theta} & & \C_2 \arrow{dl}{\c_2} \\
& \G, &
\end{tikzcd}
\]
where $\psi_i$ are Morita morphisms, and a closed basic 2-form $\beta$ on $\K\obj$ such that $\psi_2^*L_2 = \psi_1^*L_1 + \Gamma_{\beta + \theta^*\omega}$.
We leave the details to the reader since we do not need this result.
A special case appears in \cite[\S4.3]{xu:04}.
\end{remark}

\subsection{Adjoint representation up to homotopy and natural transformations}\label{f73gy0ci}

Recall that the tangent complex of the differentiable stack induced by a Lie groupoid $\G$ is the 2-term complex given by the anchor map $\aaa_\G : A_\G \to T\G\obj$ in degrees $-1$ and $0$ (see e.g.\ \cite[\S1.2.3]{cal:21}).
In particular, a morphism of Lie groupoids $f : \H \to \G$ induces a chain map from $\aaa_\H$ to $\aaa_\G$.
Suppose now that we have two such morphisms related by a natural transformation
\begin{equation}\label{k97rggdc}
\begin{tikzcd}[row sep=0pt]
&\null \arrow[Rightarrow,shorten=2pt]{dd}{\theta} & \\
\H \arrow[bend left]{rr}{f} \arrow[bend right, swap]{rr}{g} & & \G \\
&\null &
\end{tikzcd}
\end{equation}
%This homotopy will be fixed throughout this subsection.
and consider the induced chain maps
\begin{equation}\label{7yktpks9}
\begin{tikzcd}
A_\H \arrow{r}{\aaa_\H} \arrow[shift right=1,swap]{d}{f_*} \arrow[shift left=1]{d}{g_*} & T\H\obj \arrow[shift right=1,swap]{d}{f_*} \arrow[shift left=1]{d}{g_*} \\
A_\G \arrow{r}{\aaa_\G} & T\G\obj.
\end{tikzcd}
\end{equation}
Since $f$ and $g$ present the same map of differentiable stacks, we expect $f_*$ and $g_*$ to be chain homotopic.
This is not true in the obvious naive sense since the vector bundle homomorphisms $f_*$ and $g_*$ do not cover the same map and hence land on different fibres.
On the other hand, by choosing a Ehresmann connection on $\G$ and using Abad--Crainic's adjoint representation up to homotopy \cite{aba-cra:13}, we can connect those fibres and obtain the appropriate chain homotopy.
The goal of this subsection is to construct this chain homotopy and derive some useful identities.

%and cthe associated adjoint representation up to homotopy \cite[\S2.4]{aba-cra:13}.

%The goal of this section is to explain how $\theta$ induces a chain homotopy between these two chain maps.
%The issue is that, given a base point $x \in \H\obj$, the chain maps will send $(A_\H)_x$ and $T_x\H\obj$ to fibres over different points $f(x)$ and $g(x)$, and hence cannot be directly compared.
%We thus need a way to connect different fibres over related points, and this will come from the notion of Ehresmann connections on $\G$ and the associated adjoint representation up to homotopy \cite[\S2.4]{aba-cra:13}.
%Along the way, we will also obtain various useful identities relating the adjoint representation, natural transformations, and multiplicative forms.
%

Let us first briefly recall the adjoint representation up to homotopy \cite[\S2.4]{aba-cra:13}.
Let $\G$ be a Lie groupoid.
An \defn{Ehresmann connection} on $\G$ is a right splitting of the short exact sequence
\[
\begin{tikzcd}[column sep=3em]
0 \arrow{r} & \ttt^* A_\G \arrow{r}{R} & T\G\arr \arrow{r}{\sss_*} & \sss^*T\G\obj \arrow{r} & 0,
\end{tikzcd}
\]
where $R$ is right translation.
Given an Ehresmann connection $\tau$, we denote by $\dconn$ the corresponding left splitting:
\[
\begin{tikzcd}[column sep=3em]
0 \arrow{r} & \ttt^* A_\G \arrow{r}{R} & T\G\arr \arrow{r}{\sss_*} \arrow[bend left]{l}{\dconn} & \sss^*T\G\obj \arrow{r} \arrow[bend left]{l}{\conn}   & 0.
\end{tikzcd}
\]
Note that
\begin{equation}\label{bymi4ryy}
\dconn(v) = R_{g^{-1}}(v - \tau_g \sss_* v),
\end{equation}
for all $v \in T_g\G\arr$.
An Ehresmann connection always exists \cite[Lemma 2.9]{aba-cra:13}, and we fix one for the rest of this subsection.
Then $\conn$ induces quasi-actions of $\G$ on $T\G\obj$ and $A_\G$, given by
%\footnote{There is a sign difference with \cite{aba-cra:13} due to the fact that they seem to be using a convention where $\sss(a^L_g) = \aaa a$ for $a \in A_\G$, while we have $\sss(a^L_g) = \sss L_g(a - \uuu \ttt a) = \sss(a - \uuu \ttt a) = 0 - \sss \uuu \ttt a = -\aaa a$.}
\begin{align*}
\qad_g &: T_{\sss(g)}\G\obj \too T_{\ttt(g)}\G\obj, \quad \qad_g v = \ttt_*(\conn_g(v)) \\
\qad_g &: (A_\G)_{\sss(g)} \too (A_\G)_{\ttt(g)}, \quad\qad_g a = \dconn(a^L_g),
\end{align*}
for all $g \in \G$, and called the \defn{adjoint representation}.
It depends on the choice of connection (up to homotopy), but we omit it from the notation for conciseness.
Note that $\qad_g a$ for $a \in (A_\G)_{\sss(g)}$ is also characterized by
\begin{equation}\label{jifax513}
(\qad_ga)^R_g = a^L_g + \conn_g \aaa a.
\end{equation}
In particular, we see that
\begin{equation}\label{gs6260jy}
\ker \qad_g \cap \ker \aaa = 0, \quad \text{for all } g \in \G.
\end{equation}
%This formula will be useful at times.
The anchor map is equivariant with respect to the adjoint representation, i.e.
\[
\qad_g \aaa a = \aaa \qad_g a,
\]
for all $g \in \G$ and $a \in (A_\G)_{\sss(g)}$.
The extent to which these quasi-actions are genuine actions is encoded by the \defn{basic curvature}, $\curv \in \Gamma(\G\comp{2}; \Hom(\sss^*T\G\obj, \ttt^*A_\G))$, defined by
\[
\curv(g, h)(v) = \dconn(\mmm_*(\conn_g\qad_hv, \conn_hv)) \in (A_\G)_{\ttt(g)},
\]
for $(g, h) \in \G\comp{2}$ and $v \in T_{\sss(h)}\G\obj$.
Indeed, for all $(g, h) \in \G\comp{2}$, $v \in T_{\sss(h)}\G\obj$, and $a \in (A_\G)_{\sss(h)}$, we have \cite[Proposition 2.15]{aba-cra:13}
\begin{align}
\qad_g \qad_h(v) - \qad_{gh}(v) &= \aaa(\curv(g, h)(v)) \\
\qad_g \qad_h(a) - \qad_{gh}(a) &= \curv(g, h)(\aaa(a)). \label{6ohv7hix}
\end{align}

Consider a natural transformation of Lie groupoid morphisms as in \eqref{k97rggdc}.
Note that for all $v \in T_x\H\obj$, we have $\theta_*v \in T_{\theta(x)}\G\arr$ and hence $\dconn \theta_*v \in (A_\G)_{\ttt(\theta(x))} = (A_\G)_{g(x)}$.
Therefore, we have a vector bundle homomorphism
\[
\dot{\theta} \coloneqq \dconn \circ \theta_* : T\H\obj \too g^*A_\G,
\]
that we call the \defn{differential} of the natural transformation $\theta$.
It provides a chain homotopy for \eqref{7yktpks9} in the following sense.

%\begin{lemma}\label{9lcikpiv}
%Let $\G \tto M$ be a Lie groupoid and $u \in T_{g^{-1}}\G$, $v \in T_g\G$ such that $\sss u = \ttt v$ and $\sss v = 0$.
%Then $\m(u, v) = R_g(u - iv)$.
%\end{lemma}
%
%\begin{proof}
%Let $w = R_{g^{-1}}v \in T_{\ttt(g)}\G$.
%Then $\m(u, v) = \m(u, R_g w) = R_g \m(u, w)$, so it suffices to show that $\m(u, w) = u - i R_g w$.
%Note that $\sss(u - i R_g w) = \sss u - \ttt w = \sss u - \ttt v = 0$.
%Hence, $\m(u, w) = \m((u - iR_gw) + iR_gw, w) = \m(u - iR_gw, 0) + \m(iR_gw, w) = u - iR_gw$, since the map $h \mto \m(iR_gh, h) = g^{-1}$ is constant.
%\end{proof}

\begin{proposition}\label{uipubhok}
For all $v \in T_x\H\obj$ and $b \in (A_\H)_x$ we have
\begin{align}
g_* v - \qad_{\theta(x)} f_* v &= \aaa \dot{\theta}(v) \label{w6a0pud8} \\
g_* b - \qad_{\theta(x)} f_* b &= \dot{\theta} \aaa(b). \label{44x2e7y8}
\end{align}
\end{proposition}

\begin{proof}
We have
\[
\aaa \dot{\theta} v = \ttt_*(\theta_* v - \conn_{\theta(x)} f_* v) = g_* v - \qad_{\theta(x)} f_* v,
\]
which proves \eqref{w6a0pud8}.
Now, differentiating the identity
\[
g(h) = \theta(\ttt(h)) \cdot f(h) \cdot \theta(\sss(h))^{-1}, \quad (\text{for } h \in \H),
\]
and using \eqref{4ksoptjs}, we get that for all $b \in (A_\H)_x$,
\begin{align*}
g_*b &= R_{\theta(x)^{-1}} \mmm_*(\theta \ttt_* b, f_*b) \\
&= R_{\theta(x)^{-1}}(\theta \ttt_* b + (f_*b)^L_{\theta(x)}) \\
&= R_{\theta(x)^{-1}}(\theta \ttt_* b - \conn_{\theta(x)} f \ttt_* b) + R_{\theta(x)^{-1}}((f_*b)^L_{\theta(x)} + \conn_{\theta(x)} \ttt f_* b) \\
&= \dconn(\theta \ttt_* b) + \dconn((f_*b)^L_{\theta(x)}) \\
&= \dot{\theta}\aaa b + \qad_{\theta(x)} f_*b,
\end{align*}
which proves \eqref{44x2e7y8}.
\end{proof}

The following two lemmas will be used in the proof of the transfer theorem.

\begin{lemma}\label{9lc8lqnv}
%Consider a homotopy of Lie groupoid morphisms as in \eqref{k97rggdc}, let $\conn$ be a connection on $\G$ with adjoint representation $\qad$, and 
Let $\omega$ be a multiplicative 2-form on $\G$.
The following hold.
%and consider the 2-form $\tau^*\omega \in \Gamma(\sss^*\Lambda^2 T^*M)$ given by $(\tau^*\omega)_g(u, v) = \omega_g(\tau_g u, \tau_g v)$ for $u, v \in T_{\sss(g)}M$ and view it as a map $\tau^*\omega : \sss^*TM \to \sss^*T^*M$.
%Then
%\[
%\qad_g^* \IM_\omega \qad_g = \IM_\omega + \tau^* \omega \aaa : A_\G|_{\sss(g)} \too T^*_{\sss(g)}M,
%\]
%i.e.\
\begin{enumerate}[label={\textup{(\arabic*)}}]

\item
\label{nxi6yds0}
We have
\[
\ip{\IM_\omega \qad_g a, \qad_g v} = \ip{\IM_\omega a, v} + \omega(\tau_g \aaa a, \tau_g v),
\]
for all $a \in (A_\G)_{\sss(g)}$ and $v \in T_{\sss(g)}\G\obj$.

\item
\label{b91k4xix}
We have
\begin{align*}
\theta^*\omega(v, w) &= \ip{\IM_\omega \dot{\theta} v, \qad_{\theta(x)} f w} - \ip{\IM_\omega \dot{\theta} w, \qad_{\theta(x)} f v} + \ip{\IM_\omega \dot{\theta} v, \aaa \dot{\theta} w} + \omega(\tau_{\theta(x)} f v, \tau_{\theta(x)} f w),
\end{align*}
for all $v, w \in T_x\G\obj$.
\end{enumerate}
\end{lemma}

\begin{proof}
\ref{nxi6yds0}
Since $(0, a, a^L_g)$ and $(\tau_g v, \uuu_* v, \tau_g v)$ are tangent to the graph of the multiplication of $\G$ at $(g, \uuu_{\sss(g)}, g)$, we have $\ip{\IM_\omega a, v} = \omega(a, \uuu_* v) = \omega(a^L_g, \tau_g v) = \omega((\qad_ga)^R_g - \tau_g \aaa a, \tau_g v)$, where we used \eqref{jifax513} for the last equality.
Now, $(\qad_g a, 0, (\qad_ga)^R_g)$ and $(\uuu_* \qad_g v, \tau_g v, \tau_g v)$ are also tangent to the graph of the multiplication at $(\uuu_{\ttt(g)}, g, g)$, so $\omega((\qad_ga)^R_g, \tau_gv) = \omega(\qad_g a, \uuu_* \qad_g v) = \ip{\IM_\omega\qad_g a, \qad_g v}$.

\ref{b91k4xix}
First note that
\[
\ip{\IM_\omega \dot{\theta} u, \aaa \dot{\theta} v} = \omega(\dot{\theta} u, \uuu \ttt_* \dot{\theta} v) = \omega(\dot{\theta} u, \dot{\theta} v)
\]
since $\ttt_*(\dot{\theta} v - \uuu \ttt_* \dot{\theta} v) = 0$ and $\ker \sss_*$ and $\ker \ttt_*$ are orthogonal \cite[Lemma 3.1(ii)]{bur-cra-wei-zhu:04}.
Now,
\[
\omega(\theta_* u - \tau_{\theta(x)} f_* u, \theta_* v - \tau_{\theta(x)} f_* v) = \omega(R_{\theta(x)^{-1}}(\theta_* u - \tau_{\theta(x)} \sss \theta_* u), R_{\theta(x)^{-1}}(\theta_* v - \tau_{\theta(x)}\sss \theta_* v)) = \omega(\dot{\theta}u, \dot{\theta}v),
\]
by multiplicativity of $\omega$.
It follows that
\begin{equation}\label{1ahtvr5p}
\omega(\theta_* u, \theta_* v) - \omega(\tau_{\theta(x)} f_* u, \theta_* v) - \omega(\theta_* u, \tau_{\theta(x)} f_* v) + \omega(\tau_{\theta(x)} f_* u, \tau_{\theta(x)} f_* v) = \ip{\IM_\omega \dot{\theta} u, \aaa \dot{\theta} v}.
\end{equation}
Note that for every $w_1, w_2 \in \ker (\sss_*)_g$, $(R_{g^{-1}}w_1, 0, w_1)$ and $(\uuu \ttt_* w_2, w_2, w_2)$ are tangent to the graph of the multiplication of $\G$ at $(\uuu_{\ttt(g)}, g, g)$, and hence $\omega(R_{g^{-1}}w_1, \uuu \ttt_* w_2) = \omega(w_1, w_2)$.
Therefore,
\begin{align*}
\ip{\IM_\omega \dot{\theta} v, \qad_{\theta(x)} f_* u} &= \omega(\dot{\theta} v, \uuu \ttt_* \tau_{\theta(x)} f_* u) \\
&= \omega(R_{\theta(x)^{-1}}(\theta_* v - \tau_{\theta(x)} \sss \theta_* v), \uuu \ttt_* \tau_{\theta(x)} f_* u) \\
&= \omega(\theta_* v - \tau_{\theta(x)} f_* v, \tau_{\theta(x)} f_* u).
\end{align*}
Hence,
\begin{align*}
\ip{\IM_\omega \dot{\theta} v, \qad_{\theta(x)} f_* u} - \ip{\IM_\omega \dot{\theta} u, \qad_{\theta(x)} f_* v}
&= \omega(\theta_* v - \tau_{\theta(x)} f_* v, \tau_{\theta(x)} f_* u) - \omega(\theta_* u - \tau_{\theta(x)} f_* u, \tau_{\theta(x)} f_* v) \\
&= \omega(\theta_* v, \conn_{\theta(x)}f_* u) - \omega(\theta_* u, \conn_{\theta(x)} f_* v) - \omega(\conn_{\theta(x)}f_* v, \conn_{\theta(x)}f_* u) \\
&\qquad  + \omega(\conn_{\theta(x)} f_* u, \conn_{\theta(x)} f_* v) \\
&= \ip{\IM_\omega \dot{\theta} u, \aaa \dot{\theta} v} - \omega(\theta_* u, \theta_* v) + \omega(\tau_{\theta(x)} f_* u, \tau_{\theta(x)} f_* v),
\end{align*}
by \eqref{1ahtvr5p}.
\end{proof}

\begin{lemma}\label{1g34n1g6}
Let $\theta : f \Rightarrow g$ be a natural transformation as in \eqref{k97rggdc} and let $\eta : g \Rightarrow f$ be its inverse.
Then
\[
\dot{\theta} v + \qad_{\theta(x)} \dot{\eta} v + \curv(\theta(x), \eta(x))(g_* v) = 0,
\]
for all $v \in T\H\obj$.
\end{lemma}

\begin{proof}

We have
\begin{align*}
\curv(\theta(x), \eta(x))(g_* v)
&=
\dconn \mmm_*(\conn_{\theta(x)} \qad_{\eta(x)} g_* v, \conn_{\eta(x)} g_* v) \\
&=
\mmm_*(\conn_{\theta(x)} \qad_{\eta(x)} g_* v, \conn_{\eta(x)} g_* v)
- \conn_{\uuu_{g(x)}} \sss \mmm_*(\conn_{\theta(x)} \qad_{\eta(x)} g_* v, \conn_{\eta(x)} g_* v) \\
&= \mmm_*(\conn_{\theta(x)} \qad_{\eta(x)} g_* v, \conn_{\eta(x)} g_* v) - \uuu g_* v \\
&= \mmm_*(\tau_{\theta(x)} \qad_{\eta(x)} g_* v - \theta_* v, \tau_{\eta(x)} g_* v - \eta_* v),
\end{align*}
where for the last equality we used that $\mmm_*(\theta, \eta) = \uuu \ttt \theta_* = \uuu g_*$.
Since
\[
\dot{\eta} v = \dconn \eta_* v = R_{\eta(x)^{-1}}(\eta_* v - \tau_{\eta(x)} \sss \eta_* v) = R_{\theta(x)}(\eta_* v - \tau_{\eta(x)} g_* v),
\]
we have
\begin{align}
(\curv(\theta(x), \eta(x))(g_* v))^R_{\theta(x)}
&=
\mmm_*(\tau_{\theta(x)} \qad_{\eta(x)} g_* v - \theta_* v, R_{\theta(x)}(\tau_{\eta(x)} g_* v - \eta_* v)) \nonumber \\
&=
\mmm_*(\tau_{\theta(x)} \qad_{\eta(x)} g_* v - \theta_* v, -\dot{\eta} v) \nonumber \\
&= \tau_{\theta(x)} \qad_{\eta(x)} g_* v - \theta_* v - (\dot{\eta}v)^L_{\theta(x)}, \label{ltsxzp3c}
\end{align}
using \eqref{4ksoptjs} for the last step.
Now, by \eqref{bymi4ryy} and \eqref{jifax513}, we have
\begin{align}
(\dot{\theta} v + \qad_{\theta(x)} \dot{\eta} v)^R_{\theta(x)} &= \theta_* v - \tau_{\theta(x)} f_* v + (\dot{\eta} v)^L_{\theta(x)} +  \tau_{\theta(x)} \aaa \dot{\eta} v \nonumber \\
&= \theta_* v - \tau_{\theta(x)} (f_* v - \aaa \dot{\eta} v) + (\dot{\eta} v)^L_{\theta(x)} \nonumber \\
&= \theta_* v - \tau_{\theta(x)} \qad_{\eta(x)} g_* v + (\dot{\eta} v)^L_{\theta(x)}, \label{g3rv27ht}
\end{align}
where the last equality follows from Proposition \ref{uipubhok}.
Comparing \eqref{ltsxzp3c} and \eqref{g3rv27ht} gives the result.
\end{proof}

\subsection{Proof of Theorem \ref{vazibeld}}
\label{e1wxyjzw}

We now prove Theorem \ref{vazibeld} through the following five steps.
Let $(\G_1, \omega_1, \phi_1)$ and $(\G_2, \omega_2, \phi_2)$ be quasi-symplectic groupoids and let $\c_1 : \C_1 \to \G_1$ and $\c_2 : \C_2 \to \G_2$ be morphisms of Lie groupoids.
Suppose that we have a weak symplectic Morita equivalence between $\c_1$ and $\c_2$ as in \eqref{27zunnqp} with 1-shifted Lagrangian structure $\gamma$, i.e.\ (see Lemma \ref{bczyu6nt})
\begin{align}
-d\gamma &= \varphi_1^*\phi_1 - \varphi_2^*\phi_2 \\
\ttt^*\gamma - \sss^*\gamma &= \varphi_1^*\omega_1 - \varphi_2^*\omega_2. \label{rqchlevu}
\end{align}
The natural transformations $\theta_i : \c_i\psi_i \Rightarrow \varphi_i g$ are smooth maps $\theta_i : \K\obj \to \G\arr_i$ satisfying
\[
\sss \circ \theta_i = \c_i \circ \psi_i, \quad \ttt \circ \theta_i = \varphi_i \circ g,
\]
and
\begin{equation}\label{hoe86awi}
\theta_i(\ttt(a)) \cdot \c_i(\psi_i(a)) = \varphi_i(g(a)) \cdot \theta_i(\sss(a)),
\quad
\text{for all } a \in \K.
\end{equation}
Let $L_1$ be a 1-shifted coisotropic structure on $\c_1$ and let
\begin{equation}\label{x4j85ou2}
\delta \coloneqq %\beta 
- g^*\gamma + \theta_1^*\omega_1 - \theta_2^*\omega_2.
\end{equation}

\begin{step}
The pushforward
\begin{equation}\label{ddcou5n6}
L_2 \coloneqq \psi_{2*}(\psi_1^*L_1 + \Gamma_{\delta})
\end{equation}
is a well-defined Dirac structure on $\C\obj_2$ with background 3-form $\c_2^*\phi_2$.
Moreover,
\begin{equation}\label{o7uodz2w}
\psi_2^*L_2 = \psi_1^*L_1 + \Gamma_\delta.
\end{equation}
\end{step}

\begin{proof}
We apply Proposition \ref{9u2g4w2h} to the Dirac structure $L_0 \coloneqq \psi_1^*L_1 + \Gamma_{\delta}$ on $\K\obj$ (the pullback by a submersion is always well-defined \cite[Proposition 1.10]{bur:13}).
Note that the background 3-form of $L_0$ is
\begin{align}
\psi_1^*\c_1^*\phi_1 - d\delta &= \psi_1^*\c_1^*\phi_1 + g^*d\gamma - \theta_1^*d\omega_1 + \theta_2^*d\omega_2 \nonumber \\
&= \psi_1^*\c_1^*\phi_1 - g^*(\varphi_1^*\phi_1 - \varphi_2^*\phi_2) - \theta_1^*(\sss^*\phi_1 - \ttt^*\phi_1) + \theta_2^*(\sss^*\phi_2 - \ttt^*\phi_2) \nonumber \\
&= \psi_2^*\c_2^*\phi_2. \label{hpgol1ec}
\end{align}
We have $\sss^*\psi_2^*\c_2^*\phi_2 - \ttt^*\psi_2^*\c_2^*\phi_2 = \psi_2^*\c_2^*(\sss^*\phi_2 - \ttt^*\phi_2) = \psi_2^*\c_2^*d\omega_2$, and $\c_2^*d\omega_2$ is multiplicative.
It then suffices to show that 
\begin{equation}\label{0kxyfhp7}
\ttt^*L_0 = \sss^*L_0 + \Gamma_{\psi_2^*\c_2^*\omega_2}.
\end{equation}
To do so, first note that by Lemma \ref{wip6t5yh}\ref{u0vnf6px}, we have
\begin{equation}\label{poq9p95s}
g^*\varphi_i^*\omega_i - \psi_i^* \c_i^*\omega_i = \ttt^*\theta_i^*\omega_i - \sss^*\theta_i^*\omega_i,
\end{equation}
for $i = 1, 2$.
By \eqref{rqchlevu}, it follows that
\begin{align}
\ttt^*\delta - \sss^*\delta
&= -g^*(\ttt^*\gamma - \sss^*\gamma) + (\ttt^*\theta_1^*\omega_1 - \sss^*\theta_1^*\omega_1) - (\ttt^*\theta_2^*\omega_2 - \sss^*\theta_2^*\omega_2) \nonumber \\
&= -g^*(\varphi_1^*\omega_1 - \varphi_2^*\omega_2) + (g^*\varphi_1^*\omega_1 - \psi_1^*\c_1^*\omega_1) - (g^*\varphi_2^*\omega_2 - \psi_2^*\c_2^*\omega_2) \nonumber \\
&= \psi_2^*\c_2^*\omega_2 - \psi_1^*\c_1^*\omega_1. \label{xlnikjej}
\end{align}
Hence, $\ttt^*L_0 = \psi_1^* \ttt^*L_1 + \Gamma_{\ttt^*\delta} = \psi_1^* \sss^*L_1 + \Gamma_{\psi_1^*\c_1^*\omega_1 + \ttt^*\delta} = \sss^*L_0 + \Gamma_{\psi_1^*\c_1^*\omega_1 + \ttt^*\delta - \sss^*\delta} = \sss^*L_0 + \Gamma_{\psi_2^*\c_2^*\omega_2}$, proving \eqref{0kxyfhp7}.
By Proposition \ref{9u2g4w2h}, $L_2 \coloneqq \psi_{2*} L_0$ is a Dirac structure with a background 3-form $\eta_2$ satisfying $\psi_2^*\eta_2 = \psi_1^*\c_1^*\phi_1 - d\delta$.
By \eqref{hpgol1ec}, $\psi_2^*\c_2^*\phi_2 = \psi_2^*\eta_2$, so $\c_2^*\phi_2 = \eta_2$.
By the last part of Proposition \ref{9u2g4w2h}, we have $\psi_2^*L_2 = L_0 = \psi_1^*L_1 + \Gamma_\delta$.
\end{proof}

\begin{step}
The Dirac structure $L_2$ satisfies the compatibility condition $\ttt^*L_2 = \sss^*L_2 + \Gamma_{\c_2^*\omega_2}$.
\end{step}

\begin{proof}
%By the last part of Proposition \ref{9u2g4w2h}, we have $\psi_2^*L_2 = L_0$, i.e.\
%To show this, let $v \in \ker \psi_{2*}$.
%Since $\K \to \C_2$ is a weak Morita morphism, there exists $k \in A_\K$ such that $\aaa k = v$ and $\psi_2 k = 0$ (Lemma \ref{htlesxzs}).
%By Lemma \ref{9lc8lqnv}\ref{n3zv04a7}, we have $\IM_{g^*\varphi^*\omega_2}k - \IM_{\psi_2^*\c_2^*\omega_2}k = i_{\aaa k} \theta_2^*\omega_2$.
%Since $\psi_2 k = 0$ and $\aaa k = v$, this implies that
%\begin{equation}\label{6mahoqda}
%(v, i_v\theta_2^*\omega_2) = (\aaa k, g^*\IM_{\psi_2^*\c_2^*\omega_2} gk).
%\end{equation}
%Note that by \eqref{0kxyfhp7} and \eqref{poq9p95s}, we have
%\begin{equation}\label{runsqk80}
%\ttt^*(L_0 + \Gamma_{\theta_2^*\omega_2}) = \sss^*(L_0 + \Gamma_{\theta_2^*\omega_2}) + \Gamma_{g^*\varphi_2^*\omega_2}.
%\end{equation}
%We may then apply Lemma \ref{18e8u16h} to $g : \K \to \L$ and $L_0 + \Gamma_{\theta_2^*\omega_2}$.
%This shows that the right-hand side of \eqref{6mahoqda} is in $L_0 + \Gamma_{\theta_2^*\omega_2}$, and hence $(v, 0) \in L_0$.
%This proves \eqref{xvqu3pw2}, and hence that $\psi_2^*L_2 = L_0$, i.e.\
By \eqref{0kxyfhp7}, we have
\[
\psi_2^*\ttt^*L_2 
= \ttt^* \psi_2^*L_2
= \ttt^*L_0 
= \sss^*L_0 + \Gamma_{\psi_2^*\c_2^*\omega_2}
= \psi_2^*(\sss^*L_2 + \Gamma_{\c_2^*\omega_2}).
\]
Since $\psi_2$ is a submersion, this implies that $\ttt^*L_2 = \sss^*L_2 + \Gamma_{\c_2^*\omega_2}$.
\end{proof}

\begin{step}
The Dirac structure $L_2$ is a 1-shifted coisotropic structure on $\c_2$.
\end{step}

\begin{proof}
It remains to show the non-degeneracy condition.
Let $((v_2, \alpha_2), a_2) \in L_2 \times_{\c_2} A_{\G_2}$, i.e.\  $\c_{2*} v_2 = \aaa a_2$ and $\alpha_2 = \c_2^* \IM_{\omega_2} a_2$.
Since $L_2 = \psi_{2*}(\psi_1^*L_1 + \Gamma_{\delta})$, there exists $(v_1, \alpha_1) \in L_1$ and $v \in T\K\obj$ such that $v_2 = \psi_{2*} v$, $v_1 = \psi_{1*} v$, and $\psi_2^*\alpha_2 = \psi_1^*\alpha_1 + i_v \delta$.
By \eqref{rqchlevu}, we have
\begin{equation}\label{5b9ebasj}
\varphi_1^* \IM_{\omega_1} \varphi_{1*} \ell - \varphi_2^* \IM_{\omega_2} \varphi_{2*} \ell = i_{\aaa \ell} \gamma
\end{equation}
for all $\ell \in A_\L$.
Let $x \in \K\obj$ be the point over which $v$ lives.
Note that by Proposition \ref{uipubhok}, we have
\[
\aaa(\qad_{\theta_2(x)}a_2 + \dot{\theta}_2 v) = \qad_{\theta_2(x)} \c_2 \psi_{2*} v + \aaa \dot{\theta}_2 v = \varphi_2 g_* v.
\]
Since $\varphi_2$ is a weak Morita morphism, Lemma \ref{htlesxzs} shows that there exists $\ell \in A_\L$ such that $\aaa \ell = g_* v$ and $\varphi_{2*} \ell = \qad_{\theta_2(x)} a_2 + \dot{\theta}_2 v$.
Consider the inverse natural transformation $\theta_0 \coloneqq \iii \circ \theta_1 : \varphi_1 g \Rightarrow \c_1 \psi_1$.
%Since $\psi_1$ is a weak Morita morphism, there is a closed basic form $\beta_2$ on $\C_1\obj$ such that $\beta = \psi_1^*\beta_1$ (Proposition \ref{m9pdmd75}).
Let $a_1 \coloneqq \qad_{\theta_0(x)} \varphi_{1*} \ell + \dot{\theta}_0 v \in A_{\G_1}$.
We claim that $((v_1, \alpha_1), a_1) \in L_1 \times_{\c_1} A_{\G_1}$, i.e.\
\begin{enumerate}[label={(\arabic*)}]
\item \label{dcf3w1kn}
$\c_{1*} v_1 = \aaa a_1$
\item \label{eqh6buv2}
$\alpha_1 = \c_1^*\IM_{\omega_1}a_1$.
\end{enumerate}
We have $\aaa a_1 = \qad_{\theta_0(x)} \varphi_1 g_* v + \aaa \dot{\theta}_0 v = \c_1 \psi_{1*} v = \c_{1*} v_1$, which proves \ref{dcf3w1kn}.
%So
%\[
%\varphi_i \ell = \qad_{\theta_i(x)} a_i + \dot{\theta}_i v,\quad
%\text{for }i = 1, 2.
%\]
We now show \ref{eqh6buv2}.
Since $\psi_1$ is a submersion, it suffices to show that $\psi_1^*\c_1^*\IM_{\omega_1} a_1 = \psi_1^*\alpha_1$.
But $\psi_1^*\alpha_1 = \psi_2^*\alpha_2 - i_v\delta = \psi_2^*\c_2^* \IM_{\omega_2} a_2 - i_v\delta$, so we need to show that
\begin{equation}\label{eacl1j7c}
\psi_1^*\c_1^*\IM_{\omega_1} a_1 = \psi_2^*(\c_2^*\IM_{\omega_2} a_2) - i_v \delta.
\end{equation}
Let $w \in T_x\K\obj$.
We have
\begin{align*}
\ip{\psi_1^*\c_1^*\IM_{\omega_1} a_1, w}
&= \ip{\IM_{\omega_1} a_1, \c_1 \psi_{1*} w} \\
&= \ip{\IM_{\omega_1}(\qad_{\theta_0(x)}\varphi_{1*} \ell + \dot{\theta}_0 v), \qad_{\theta_0(x)} \varphi_1 g_* w + \aaa \dot{\theta}_0 w} \\
&= \ip{\IM_{\omega_1} \varphi_{1*} \ell, \varphi_1 g_* w} + \omega_1(\tau_{\theta_0(x)} \varphi_1 g_* v, \tau_{\theta_0(x)} \varphi_1 g_* w) \\
&\qquad + \ip{\IM_{\omega_1} \dot{\theta}_0 v, \qad_{\theta_0(x)} \varphi_1 g_* w + \aaa \dot{\theta}_0 w} + \ip{\IM_{\omega_1} \qad_{\theta_0(x)} \varphi_{1*} \ell, \aaa \dot{\theta}_0 w}
\end{align*}
by Lemma \ref{9lc8lqnv}\ref{nxi6yds0}.
Noting that
\[
\ip{\IM_{\omega_1} \qad_{\theta_0(x)} \varphi_{1*} \ell, \aaa \dot{\theta}_0 w} = -\ip{\IM_{\omega_1} \dot{\theta}_0 w, \aaa \qad_{\theta_0(x)} \varphi_{1*} \ell} = -\ip{\IM_{\omega_1} \dot{\theta}_0 w, \qad_{\theta_0(x)}\varphi_1 g_* v},
\]
and applying Lemma \ref{9lc8lqnv}\ref{b91k4xix}, we get that
\[
\ip{\psi_1^* \c_1^* \IM_{\omega_1} a_1, w} = \ip{\IM_{\omega_1} \varphi_{1*} \ell, \varphi_1 g_* w} + \theta_0^*\omega_1(v, w) = \ip{\IM_{\omega_1} \varphi_{1*} \ell, \varphi_1 g_* w} - \theta_1^*\omega_1(v, w),
\]
i.e.\
\begin{equation}\label{tgi4nkk9}
\psi_1^*\c_1^*\IM_{\omega_1} a_1 = g^* \varphi_1^* \IM_{\omega_1} \varphi_{1*} \ell - i_v \theta_1^*\omega_1.
\end{equation}
Similarly, 
%Lemma \ref{9lc8lqnv}\ref{nxi6yds0} implies that
%\begin{align*}
%\ip{\psi_2^*\c_2^*\IM_{\omega_2} a_2, w} 
%&= \ip{\IM_{\omega_2} a_2, \c_2 \psi_2 w} \\
%&= \ip{\IM_{\omega_2} \qad_{\theta_2(x)} a_2, \qad_{\theta_2(x)} \c_2 \psi_2 w} - \omega_2(\tau_{\theta_2(x)} \c_2 \psi_2 v, \tau_{\theta_2(x)} \c_2 \psi_2 w) \\
%&= \ip{ \IM_{\omega_2}(\varphi_2 \ell - \dot{\theta}_2 v), \varphi_2 g w - \aaa \dot{\theta}_2 w} - \omega_2(\tau_{\theta_2(x)} \c_2 \psi_2 v, \tau_{\theta_2(x)} \c_2 \psi_2 w).
%\end{align*}
%Expanding everything out and using Lemma \ref{9lc8lqnv}\ref{b91k4xix} again, we get that
\begin{equation}\label{br8hdzve}
\psi_2^*\c_2^*\IM_{\omega_2} a_2 = g^* \varphi_2^* \IM_{\omega_2} \varphi_{2*} \ell - i_v\theta_2^*\omega_2.
\end{equation}
Combining \eqref{tgi4nkk9} and \eqref{br8hdzve} with \eqref{5b9ebasj}, we get
\begin{align*}
\psi_1^*\c_1^* \IM_{\omega_1} a_1 &= g^* \varphi_1^* \IM_{\omega_1} \varphi_{1*} \ell - i_v \theta_1^*\omega_1 \\
&= g^*(\varphi_2^* \IM_{\omega_2} \varphi_{2*} \ell + i_{g_* v} \gamma) - i_v \theta_1^*\omega_1 \\
&= \psi_2^* \c_2^*\IM_{\omega_2} a_2 + i_v \theta_2^*\omega_2 + i_v g^*\gamma - i_v \theta_1^*\omega_1 \\
&= \psi_2^* \c_2^* \IM_{\omega_2} a_2 - i_v \delta,
\end{align*}
which proves \eqref{eacl1j7c}, and hence \ref{eqh6buv2}.
Therefore, by the non-degeneracy condition of the 1-shifted coisotropic structure $L_1$, there exists $h_1 \in A_{\C_1}$ such that $((\aaa h_1, \c_1^* \IM_{\omega_1} \c_{1*} h_1), \c_{1*} h_1) = ((v_1, \alpha_1), a_1)$.
By Lemma \ref{htlesxzs}, there exists $k \in A_\K$ such that $\aaa k = v$ and $\psi_{1*} k = h_1$.
By Proposition \ref{uipubhok}, we have $\varphi_1 g_* k = \dot{\theta}_1 \aaa k + \qad_{\theta_1(x)} \c_1 \psi_{1*} k = \dot{\theta}_1 v + \qad_{\theta_1(x)} a_1$.
By \eqref{6ohv7hix} and Lemma \ref{1g34n1g6}, we have
\begin{align*}
\dot{\theta}_1 v + \qad_{\theta_1(x)} a_1 &= \dot{\theta}_1 v + \qad_{\theta_1(x)} \qad_{\theta_0(x)} \varphi_{1*} \ell + \qad_{\theta_1(x)} \dot{\theta}_0 v \\
&= \dot{\theta}_1 v + K(\theta_1(x), \theta_0(x)) g_* v + \varphi_{1*} \ell + \qad_{\theta_1(x)} \dot{\theta}_0 v \\
&= \varphi_{1*} \ell.
\end{align*}
It follows that $\varphi_1 g_* k = \varphi_{1*} \ell$.
Since also $\aaa g_* k = g_* v = \aaa \ell$, we have $g_* k = \ell$. %***requires $\varphi_1$ to be a Morita morphism***.
Let $h_2 \coloneqq \psi_{2*} k$.
Then
\begin{equation}\label{ehqr2ibo}
\aaa h_2 = \psi_{2*} v = v_2.
\end{equation}
Also, by Proposition \ref{uipubhok},
\[
\qad_{\theta_2(x)}a_2 + \dot{\theta}_2 v = \varphi_{2*} \ell = \varphi_2 g_* k = \dot{\theta}_2\aaa k + \qad_{\theta_2(x)} \c_2 \psi_{2*} k = \dot{\theta}_2v + \qad_{\theta_2(x)} \c_{2*} h_2,
\]
so $\qad_{\theta_2(x)}a_2 = \qad_{\theta_2(x)} \c_{2*} h_2.$
Since also $\aaa a_2 = \aaa \c_{2*} h_2$, \eqref{gs6260jy} implies that
\begin{equation}\label{gfre77wt}
a_2 = \c_{2*} h_2
\end{equation}
Equations \eqref{ehqr2ibo} and \eqref{gfre77wt} then show that $L_2$ is a coisotropic structure.
\end{proof}

\begin{step}
If $L_1$ is a strong 1-shifted coisotropic structure and $\psi_1$ and $\psi_2$ are Morita morphisms, then $L_2$ is also strong.
\end{step}

\begin{proof}
Let $h_2 \in A_{\C_2}$ be such that $\aaa h_2 = 0$ and $\c_{2*} h_2 = 0$.
We need to show that $h_2 = 0$.
By Lemma \ref{htlesxzs}, there exists $k \in A_\K$ such that $\psi_{2*} k = h_2$ and $\aaa k = 0$.
Moreover, by Proposition \ref{uipubhok}, $\varphi_2 g_* k = \Ad_{\theta_2(x)} \c_2 \psi_{2*} k + \dot{\theta}_2 \aaa k = 0$.
Since also $\aaa g_* k = g_* \aaa k = 0$ and $\varphi_2$ is a Morita morphism, we have $g_* k = 0$.
Let $h_1 \coloneqq \psi_{1*} k \in A_{\C_1}$.
Then $\aaa h_1 = \psi_{1*} \aaa k = 0$.
Also, $\qad_{\theta_1(x)}\c_{1*} h_1 = \varphi_1 g_* k - \dot{\theta}_1 \aaa k = 0$ and $\aaa \c_{1*} h_1 = 0$ so $\c_{1*} h_1 = 0$ by \eqref{gs6260jy}.
By non-degeneracy of the coisotropic structure on $\C_1 \to \G_1$, we have $h_1 = 0$.
It follows that $\psi_{1*} k = 0$ and $\aaa k = 0$.
Since $\psi_1 : \K \to \C_1$ is a Morita morphism, this implies that $k = 0$ and hence $h_2 = 0$.
It follows that $L_2$ is a strong 1-shifted coisotropic structure on $\c_2$.
%Moreover, by \eqref{o7uodz2w}, the diagram \eqref{27zunnqp} is a coisotropic Morita equivalence.
\end{proof}

\begin{step}
The map $L_1 \mto L_2$ is a bijection from the set of 1-shifted coisotropic structures on $\c_1$ to those on $\c_2$.
\end{step}

\begin{proof}
The same symplectic Morita equivalence \eqref{27zunnqp} gives a map $L_2 \mto L_1$ in the other direction.
The compatibility equation \eqref{o7uodz2w} together with the fact that $\psi_i$ are surjective submersions show that they are inverse to each other.
\end{proof}

%\begin{corollary}
%Let $\G$ be a quasi-symplectic groupoid and
%\[
%\begin{tikzcd}[row sep = 6pt]
%\widetilde{\C} \arrow{dr} \arrow{dd} \\
% & \G \\
%\C \arrow{ur}
%\end{tikzcd}
%\]
%a commutative diagram of Lie groupoid morphisms such that the vertical arrow is a Morita morphism.
%Then coisotropic structures on $\C \to \G$ are in one-to-one correspondence with coisotropic structures on $\widetilde{\C} \to \G$ via pullback and pushfoward.
%\end{corollary}
%
%\begin{corollary}\label{j4b1f2jk}
%Let $\G \tto M$ be a 1-shifted symplectic groupoid, $\H \tto N$ a Lie groupoid, and $(\G \times \H) \ltimes N$ an action such that $N/\H$ is a manifold.
%Then 1-shifted coisotropic structures on the morphism $(\G \times \H) \ltimes N \to \G$ are in one-to-one correspondence with 1-shifted coisotropic structures on $\G \ltimes N/\H \to \G$.
%\end{corollary}
%
%\begin{proof}
%Let $Q = N/\H$ with quotient map $\pi : N \to Q$.
%We have a strong Morita morphism $(\G \times \H) \ltimes N \to \G \ltimes Q$ comuting with the maps to $\G$.
%\end{proof}
%

\section{Functorial properties of the transfer map}
\label{r9gv7nax}

We now study the extent to which the correspondence in Theorem \ref{vazibeld} depends on the choice of symplectic Morita equivalence \eqref{27zunnqp} and prove that it is compatible with the composition of Morita equivalences.
This will be used in \S\ref{dy26su68} for the definition of 1-shifted coisotropics on morphisms of differentiable stacks.
The results of this section will be at the level of gauge equivalence classes of 1-shifted coisotropic structures, which we define next.

\subsection{Gauge equivalence of 1-shifted coisotropic structures}
We have the following notion of internal symmetries of 1-coisotropics.

%\begin{proposition}\label{lmn5w4xw}
%Let $(\C, \c, L)$ be a 1-coisotropic.
%For every closed basic 2-form $\beta$ on $\C\obj$ (i.e.\ $d\beta = 0$ and $\sss^*\beta = \ttt^*\beta$), $(\C, \c, L + \Gamma_\beta)$ is another 1-coisotropic.
%\end{proposition}
%
%***This is false: e.g.\ take a 1-Lagrangian over a point, i.e.\ $\omega$ is a closed basic form on $\C\obj$ such that $\im \rho_\C = \ker \omega$.
%Let $\beta = -\omega$.***
%
%\begin{proof}
%The compatibility condition is clear.
%For the non-degeneracy condition, it suffices to show that $(L + \Gamma_\beta) \times_\c A_\G = L \times_\c A_\G$ so that \eqref{0v1fz7h6} is unchanged.
%Let $((v, \alpha + i_v\beta), a) \in (L + \Gamma_\beta) \times_\c A_\G$, i.e.\ $(v, \alpha) \in L$, $\c_* v = \aaa a$ and $\alpha + i_v\beta = \c^*\IM_\omega a$.
%Since $\beta$ is basic, we have $i_v\beta = i_{\ttt_*a}\beta = \uuu^*\ttt^*i_{\ttt_* a}\beta = \uuu^* \sss^*i_{\sss_* a} \beta = 0$, so indeed $((v, \alpha + i_v\beta), a) = ((v, \alpha), a) \in L \times_\c A_\G$.
%The converse follows by the same argument.
%\end{proof}

\begin{definition}\label{x88vq0lb}
Two 1-shifted coisotropic structures $L_1$ and $L_2$ on $\c : \C \to \G$ are \defn{gauge equivalent} if $L_2 = L_1 + \Gamma_\beta$ for some closed basic $2$-form $\beta$ on $\C$.
%(i.e.\ $d\beta = 0$ and $\sss^*\beta = \ttt^*\beta$).
%Two 1-shifted coisotropic structures related by a gauge transformation are called \defn{gauge equivalent}.
\end{definition}

\begin{remark}
In particular, this gives a notion of gauge equivalence of $0$-shifted Poisson structures (Proposition \ref{q2gecljb}).
If $\G$ is a 0-shifted Poisson Lie groupoid such that $\G\obj/\G\arr$ is a manifold, then (by Corollary \ref{s7wgomy8} and Lemma \ref{6sfaxsof}) this coincides with the standard notion of gauge equivalence of Poisson structures on $\G\obj/\G\arr$ \cite{sev-wei:01,bur:05}.
\end{remark}

%Note that gauge equivalence is indeed an equivalence relation.
Note that if $L$ is a 1-shifted coisotropic structure and $\beta$ is any closed basic $2$-form on $\C\obj$, then $L + \Gamma_\beta$ satisfies the compatibility condition \ref{0k2j1uwv} in Definition \ref{gm52z93m}.
Hence, $L + \Gamma_\beta$ is a 1-shifted coisotropic structure for $\beta$ small enough.
For arbitrary $\beta$, $L + \Gamma_\beta$ does not necessarily satisfy the non-degeneracy condition \ref{3bw7gcnj}, but a sufficient condition is when it comes from an internal symmetry of the morphism $\c$:

\begin{proposition}
Let $(\C, \c, L)$ be a 1-coisotropic on $(\G, \omega, \phi)$ and $\theta : \c \Rightarrow \c$ a natural transformation.
Then $(\C, \c, L + \Gamma_{\theta^*\omega})$ is a 1-coisotropic on $\G$.
\end{proposition}

\begin{proof}
Consider the symplectic Morita equivalence
\[
\begin{tikzcd}[column sep={5em,between origins},row sep={2.5em,between origins}]
& \C \arrow[swap]{dl}{\Id} \arrow{dd}{\c} \arrow{dr}{\Id} & \\
\C \arrow[swap]{dd}{\c} \arrow[Rightarrow,shorten=14pt,swap]{dr}{\theta} &  & \C \arrow{dd}{\c}  \\
& \G \arrow{dl}{\Id} \arrow[swap]{dr}{\Id} & \\
\G & & \G
\end{tikzcd}
\]
and apply Theorem \ref{vazibeld}. 
\end{proof}

The following sufficient condition for gauge equivalence will be useful.

\begin{lemma}\label{k1r3zchl}
%Let $(\G, \omega, \phi)$ be a quasi-symplectic groupoid and $\c : \C \to \G$ a Lie groupoid morphism.
Let $L_1$ and $L_2$ be 1-shifted coisotropic structures on $\c : \C \to \G$ and suppose that there is a weak coisotropic Morita equivalence
\[
\begin{tikzcd}[column sep={4em,between origins},row sep={2em,between origins}]
& \K \arrow[swap]{dl}{\psi_1} \arrow{dd}{g} \arrow{dr}{\psi_2} & \\
\C \arrow[swap]{dd}{\c} \arrow[Rightarrow,shorten=10pt,swap]{dr}{\theta_1} &  & \C \arrow{dd}{\c} \arrow[swap,Rightarrow,shorten=10pt,swap]{dl}{\theta_2} \\
& \L \arrow{dl}{\varphi_1} \arrow[swap]{dr}{\varphi_2} & \\
\G & & \G
\end{tikzcd}
\]
such that $\psi_1$ and $\psi_2$ are homotopic (i.e.\ related by a natural transformation).
Then $L_1$ and $L_2$ are gauge equivalent.
\end{lemma}

\begin{proof}
%The forward implication follows by taking $\K = \C$ and $\L = \G$.
Let $\eta : \psi_1 \Rightarrow \psi_2$ be a natural transformation and let $\delta \coloneqq - g^*\gamma + \theta_1^*\omega - \theta_2^*\omega$ be the connecting form.
Note that
\begin{align*}
\ttt^*\delta - \sss^*\delta
&= g^*(\sss^*\gamma - \ttt^*\gamma) + (g^*\varphi_1^*\omega - \psi_1^*\c^*\omega) + (\psi_2^*\c^*\omega - g^*\varphi_2^*\omega) \\
&= \psi_2^*\c^*\omega - \psi_1^*\c^*\omega \\
&= \ttt^*\eta^*\c^*\omega - \sss^*\eta^*\c^*\omega,
\end{align*}
so $\delta - \eta^*\c^*\omega$ is basic.
Moreover, by \eqref{hpgol1ec}, we have
\[
d\delta = \psi_1^*\c^*\phi - \psi_2^*\c^*\phi = \eta^*\sss^*\c^*\phi - \eta^*\ttt^*\c^*\phi = \eta^*\c^*d\omega
\]
so $\delta - \eta^*\c^*\omega$ is closed.
It follows that $\delta - \eta^*\c^*\omega = \psi_2^*\alpha$ for some closed basic form $\alpha$ on $\C$ (Proposition \ref{m9pdmd75}).
Therefore,
\begin{align*}
\psi_2^*L_2 &= \psi_1^*L_1 + \Gamma_{\delta} = \eta^* \sss^* L_1 + \Gamma_{\delta} = \eta^*(\ttt^* L_1 - \Gamma_{\c^*\omega}) + \Gamma_{\delta} = \psi_2^*L_1 + \Gamma_{\delta - \eta^*\c^*\omega} = \psi_2^*(L_1 + \Gamma_\alpha).
\end{align*}
Since $\psi_2$ is a submersion, we have $L_2 = L_1 + \Gamma_\alpha$.
\end{proof}

%For example, gauge equivalence can arise through internal symmetries of the map $\c : \C \to \G$, i.e.\ if $\theta : \c \Rightarrow \c$ is a homotopy, then $\beta \coloneqq \theta^*\omega$ is a closed basic form.

%\begin{definition}
%Let $(\G, \omega, \phi)$ be a quasi-symplectic groupoid and $c : \C \to \G$ a Lie groupoid morphism.
%Two coisotropic structures $L_1$ and $L_2$ on $c$ are \defn{equivalent} if there exists a closed basic form $\beta$ on $\H\obj$ (i.e.\ $d\beta = 0$ and $\sss^*\beta = \ttt^*\beta$) such that $L_2 = L_1 + L_\beta$.
%\end{definition}
%
%
%\begin{corollary}
%Let $c : (\C, L_\C) \to (\G, \omega, \phi)$ be a shifted coisotropic and $\theta : c \Rightarrow c$ a natural transformation.
%Then $L_\C + L_{\theta^*\omega}$ is another coisotropic structure.
%\end{corollary}
%
%\begin{example}
%If there is a natural transformation 
%\[
%\begin{tikzcd}
%\H \arrow[bend right,swap]{rr}{f} \arrow[bend left]{rr}{f}
%& \rotatebox[origin=c]{270}{$\Rightarrow$}\;\theta
%& \G,
%\end{tikzcd}
%\]
%then $\beta \coloneqq \theta^*\omega$ is closed and basic by Lemma \ref{9lc8lqnv}\ref{u0vnf6px}.
%\end{example}
%
\subsection{The transfer map}
For a quasi-symplectic groupoid $\G$ and a morphism of Lie groupoids $c : \C \to \G$, let $\coiso(c)$ be the set of 1-shifted coisotropic structures on $c$.
Theorem \ref{vazibeld} shows that for a weak symplectic Morita equivalence $\ME$ 
\begin{equation}\label{57mf6lqw}
\begin{tikzcd}[column sep={5em,between origins},row sep={2.5em,between origins}]
& \K \arrow[swap]{dl}{\psi_1} \arrow{dd}{g} \arrow{dr}{\psi_2} & \\
\C_1 \arrow[swap]{dd}{c_1} \arrow[Rightarrow,shorten=14pt,swap]{dr}{\theta_1} &  & \C_2 \arrow{dd}{c_2} \arrow[Rightarrow,shorten=14pt]{dl}{\theta_2} \\
& \L \arrow{dl}{\varphi_1} \arrow[swap]{dr}{\varphi_2} & \\
\G_1 \arrow[rr, yshift=-1em, no head, decorate, decoration={brace,mirror}, swap, "\ME" {yshift=-4pt}] & & \G_2,
\end{tikzcd}
\end{equation}
there is a bijection
\begin{equation}\label{pn3lblh9}
\transfer_{\ME} : \coiso(c_1) \too \coiso(c_2).
\end{equation}
%This map depends of course on more data than what is shown, but this consice notation will suffice to distinguish them in the following discussion.
Let
\[
\coisoeq(c) \coloneqq \coiso(c) / {\sim} 
\]
be the set of gauge equivalence classes of 1-shifted coisotropic structures (Definition \ref{x88vq0lb}).
We will show that $\transfer_{\ME}$ descends to a bijection
\begin{equation}\label{1lj7fbbs}
\transfereq_\ME : \coisoeq(c_1) \too \coisoeq(c_2),
\end{equation}
called the \defn{transfer map}, and study its dependency on $\ME$.

\begin{proposition}
The bijection \eqref{pn3lblh9} descends to a bijection \eqref{1lj7fbbs}, and the latter is independent of the choices of natural transformations $\theta_i$ and 1-shifted Lagrangian structure on $\L \to \G_1 \times \G_2^-$.
\end{proposition}

\begin{proof}
%We retain the notation in \S\ref{e1wxyjzw}.
%the proof of Theorem \ref{vazibeld}.
Let us first show that the map $\transfer_{\ME}$ descends to gauge equivalence classes.
Let $L_1$ be a 1-shifted coisotropic structure on $\c_1$ and $\beta_1$ a closed basic 2-form on $\C_1$ such that $L_1 + \Gamma_{\beta_1}$ is a 1-shifted coisotropic structure.
By Proposition \ref{m9pdmd75}, there is a unique closed basic 2-form $\beta_2$ on $\C_2$ such that $\psi_1^*\beta_1 = \psi_2^*\beta_2$.
%It suffices to show that $\transfer_{\ME}(L_1 + \Gamma_{\beta_1}) = \transfer_{\ME}(L_1) + \Gamma_{\beta_2}$.
Let $\delta$ be the connecting form of $\ME$.
Then $\transfer_{\ME}(L_1 + \Gamma_{\beta_1}) = \psi_{2*}(\psi_1^*(L_1 + \Gamma_{\beta_1}) + \Gamma_\delta) = \psi_{2*}(\psi_1^*L_1 + \Gamma_\delta) + \Gamma_{\beta_2} = \transfer_{\ME}(L_1) + \Gamma_{\beta_2}$.
It follows that $\transfer_\M$ descends to a bijection $\transfereq_\M$.

%To show that $\transfereq_\ME$ does not depend on the choice of closed basic form $\beta$, let $\widetilde{\ME}$ be is as $\ME$ but with a different closed basic form $\widetilde{\beta}$.
%Let $L_1 \in \coiso(\c_1)$, $L_2 = \transfer_{\ME}(L_1)$, and $\widetilde{L}_2 = \transfer_{\widetilde{\ME}}$.
%Now, $\widetilde{\beta} - \beta$ is basic, so, by Proposition \ref{m9pdmd75}, there is a closed basic 2-form $\beta_2$ on $\C_2\obj$ such that $\psi_2^*\beta_2 = \widetilde{\beta} - \beta$.
%It follows that $\psi_2^*\widetilde{L}_2 = \psi_1^*L_1 + \Gamma_{\widetilde{\delta}} = \psi_1^*L_1 + \Gamma_{\delta + \widetilde{\beta} - \beta} = \psi_2^*(L_2 + \Gamma_{\beta_2})$.
%Since $\psi_2$ is a submersion, we have $\widetilde{L}_2 = L_2 + \Gamma_{\beta_2}$ and hence they are gauge equivalent.

Now, let $\widetilde{\ME}$ be as in \eqref{57mf6lqw} but with possibly different natural transformations $\widetilde{\theta}_i$ and 1-shifted Lagrangian structure $\widetilde{\gamma}$.
%We show that $\transfereq_\ME = \transfereq_{\widetilde{\ME}}$.
%We now show that if $\widetilde{\theta}_i$ are other choices of natural transformations in \eqref{27zunnqp} and $\widetilde{L}_2$ is the corresponding 1-shifted coisotropic structure on $\C_2 \to \G_2$, then $L_2$ and $\widetilde{L}_2$ are gauge equivalent.
Let $L_1 \in \coiso(\c_1)$, $L_2 \coloneqq \transfer_\ME(L_1)$, $\widetilde{L}_2 \coloneqq \transfer_{\widetilde{\ME}_2}(L_1)$, and $\widetilde{\delta}$ the connecting form of $\widetilde{\ME}$.
We have $\psi_2^*L_2 = \psi_1^*L_1 + \Gamma_{\delta}$ and $\psi_2^*\widetilde{L}_2 = \psi_1^*L_1 + \Gamma_{\widetilde{\delta}}$, so $\psi_2^*L_2 = \psi_2^*\widetilde{L}_2 + \Gamma_{\zeta}$, where $\zeta \coloneqq \delta - \widetilde{\delta} = g^*(\widetilde{\gamma} - \gamma) + \theta_1^*\omega_1 - \widetilde{\theta}_1^*\omega_1 + \widetilde{\theta}_2^*\omega_2 - \theta_2^*\omega_2$.
Note that $d(\widetilde{\theta}_i^*\omega_i - \theta_i^*\omega_i) = \widetilde{\theta}_i^*(\sss^*\phi_i - \ttt^*\phi_i) - \theta_i^*(\sss^*\phi_i - \ttt^*\phi_i) = 0$ for $i = 1, 2$, and $d(\widetilde{\gamma} - \gamma) = 0$, so $\zeta$ is closed.
By Lemma \ref{wip6t5yh}\ref{u0vnf6px}, $\zeta$ is also basic.
We then have $\zeta = \psi_2^*\beta_2$ for some closed basic 2-form $\beta_2$ on $\C_2$ (Proposition \ref{m9pdmd75}).
It follows that $\psi_2^*L_2 = \psi_2^*(\widetilde{L}_2 + \Gamma_{\beta_2})$ and hence $L_2 = \widetilde{L}_2 + \Gamma_{\beta_2}$ since $\psi_2$ is a submersion.
%It follows that $\transfer_{\ME}$ descends to gauge equivalence classes.
%By symmetry of the argument, the inverse of $\transfer_{\ME}$ also descends, and hence we have a bijection on gauge equivalence classes.
%Now, let $\widetilde{\ME}$ be
%Finally, suppose that we have another 1-shifted Lagrangian structure $\widetilde{\gamma}$ on $\L \to \G_1 \times \G_2^-$.
%Let $\widetilde{L}_2$ be the 1-shifted coisotropic structure on $c_2$ obtained from $L_1$ by this new correspondence.
%Let $\zeta \coloneqq \gamma - \widetilde{\gamma}$.
%Since $\ttt^*\gamma - \sss^*\gamma = \varphi_1^*\omega_1 - \varphi_2^*\omega_2 = \ttt^*\widetilde{\gamma} - \sss^*\widetilde{\gamma}$, we have $\ttt^*\zeta = \sss^*\zeta$.
%Also $d\zeta = d\gamma - d\widetilde{\gamma} = 0$.
%By Proposition \ref{m9pdmd75}, $g^*\zeta = \psi_2^*\beta_2$ for some closed basic 2-form $\beta_2$ on $\C\obj_2$.
%Moreover, $\psi_2^*\widetilde{L}_2 = \psi_1^*L_1 + \Gamma_{\beta -g^*\widetilde{\gamma} + \theta_1^*\omega_1 - \theta_2^*\omega_2} = \psi_1^*L_1 + \Gamma_\delta + \Gamma_{g^*\zeta} = \psi_2^*L_2 + \Gamma_{\psi_2^*\beta_2} = \psi_2^*(L_2 + \Gamma_{\beta_2})$.
%It follows that $\widetilde{L}_2 = L_2 + \Gamma_{\beta_2}$.
\end{proof}

\subsection{Functoriality}

Since the transfer map $\transfereq_\ME$ is independent of the choice of natural transformations, we now drop them from our notation in diagrams.
We must, however, keep in mind that diagrams are only 2-commutative.
%We may also assume from now on without loss of generality that the closed basic forms $\beta$ in coisotropic Morita equivalences are trivial.

%We now prove functoriality of the transfer map $\transfereq$.
Suppose that we have two symplectic Morita equivalences $\ME_1$ and $\ME_2$:
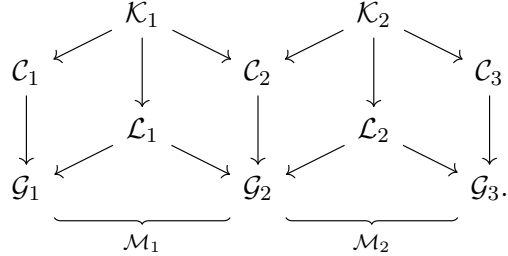
\[
\begin{tikzcd}[column sep={4em,between origins},row sep={2em,between origins}]
     & \K_1 \arrow{dl}\arrow{dd}\arrow{dr} &      & \K_2 \arrow{dl}\arrow{dd}\arrow{dr}     &   \\
\C_1 \arrow{dd} &  & \C_2 \arrow{dd} &  & \C_3 \arrow{dd} \\
& \L_1 \arrow{dl}\arrow{dr} & & \L_2 \arrow{dl}\arrow{dr} & \\
\G_1 \arrow[rr, yshift=-1em, no head, decorate, decoration={brace,mirror}, swap, "\ME_1" {yshift=-4pt}] &      & \G_2 \arrow[rr, yshift=-1em, no head, decorate, decoration={brace,mirror}, swap, "\ME_2" {yshift=-4pt}] &      & \G_3.
\end{tikzcd}
\]
By Proposition \ref{9j2iully}, $\L_1 \htimes_{\G_2} \L_2$ is a 1-shifted Lagrangian correspondence from $\G_1$ to $\G_3$.
The fibre product $\K_1 \htimes_{\C_2} \K_2$ is also a Morita equivalence between $\C_1$ and $\C_3$ \cite[Proposition 4.4.4]{hoy:13}.
Hence, we have a composition $\ME_2 \circ \ME_1$ of symplectic Morita equivalences
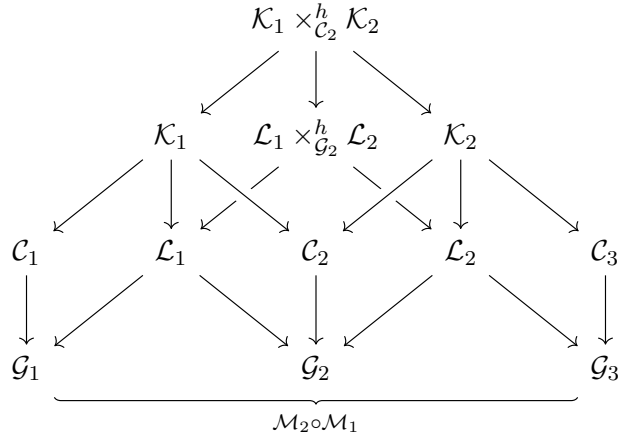
\[
\begin{tikzcd}[column sep={5em,between origins},row sep={4em,between origins}]
& & \K_1 \htimes_{\C_2} \K_2 \arrow{dl}\arrow{d}\arrow{dr} & &\\
     & \K_1 \arrow{dl}\arrow{d} &  \L_1 \htimes_{\G_2} \L_2 \arrow{dl}\arrow{dr}   & \K_2 \arrow{d}\arrow{dr} &       \\
\C_1 \arrow{d} & \L_1 \arrow{dl}\arrow{dr} & \C_2 \arrow{d} \arrow[from=ul, crossing over] \arrow[from=ur, crossing over] & \L_2 \arrow{dl}\arrow{dr} & \C_3 \arrow{d} \\
\G_1 \arrow[rrrr, yshift=-1em, no head, decorate, decoration={brace,mirror}, swap, "\ME_2 \circ \ME_1" {yshift=-4pt}] &      & \G_2 &      & \G_3
\end{tikzcd}
\]

\begin{proposition}\label{3ni3b2b5}
We have $\transfereq_{\ME_2 \circ \ME_1} = \transfereq_{\ME_2} \circ \transfereq_{\ME_1}$.
\end{proposition}

\begin{proof}
Let $\widehat{\L} = \L_1 \htimes_{\G_2} \L_2$ and $\widehat{\K} = \K_1 \htimes_{\C_2} \K_2$ and introduce the labels
\begin{equation}\label{04g8a4em}
\begin{tikzcd}[column sep={7em,between origins},row sep={5em,between origins}]
& & \widehat{\K} \arrow[swap]{dl}{\pr_{\K_1}}\arrow{d}{\widehat{g}}\arrow{dr}{\pr_{\K_2}} & &\\
     & \K_1 \arrow[swap]{dl}{\psi_{11}}\arrow[swap]{d}{g_1} &  \widehat{\L} \arrow[swap, "\pr_{\L_1}"{pos=0.15,outer sep=-2pt}]{dl}\arrow["\pr_{\L_2}"{pos=0.15,outer sep=-2pt}]{dr}   & \K_2 \arrow{d}{g_2}\arrow{dr}{\psi_{23}}  &      \\
\C_1 \arrow{d}{c_1} & \L_1 \arrow{dl}{\varphi_{11}}\arrow[swap]{dr}{\varphi_{12}} & \C_2 \arrow{d}{c_2} \arrow[from=ul, crossing over, swap, "\psi_{12}"{pos=0.85,outer sep=-2pt}] \arrow[from=ur, crossing over, "\psi_{22}"{pos=0.85,outer sep=-2pt}] & \L_2 \arrow{dl}{\varphi_{22}}\arrow[swap]{dr}{\varphi_{23}} & \C_3 \arrow{d}{c_3} \\
\G_1 &      & \G_2 &      & \G_3.
\end{tikzcd}
\end{equation}
Let us denote the natural transformations in \eqref{04g8a4em} by
\[
\begin{tikzcd}[column sep={5em,between origins},row sep={2.5em,between origins}]
     & \K_1 \arrow{dl}\arrow{dd}\arrow{dr} &      & \K_2 \arrow{dl}\arrow{dd}\arrow{dr}     &   \\
\C_1 \arrow{dd} \arrow[Rightarrow,shorten=12pt]{dr}{\theta_{11}} &  & \C_2 \arrow{dd} \arrow[Rightarrow,shorten=12pt,swap]{dl}{\theta_{12}} \arrow[Rightarrow,shorten=12pt]{dr}{\theta_{22}} &  & \C_3 \arrow{dd} \arrow[Rightarrow,shorten=12pt,swap]{dl}{\theta_{23}} \\
& \L_1 \arrow{dl}\arrow{dr} & & \L_2 \arrow{dl}\arrow{dr} & \\
\G_1  &      & \G_2  &      & \G_3.
\end{tikzcd}
\]
Recall that the natural transformations making the squares defining the homotopy fibre products $\widehat{\K}$ and $\widehat{\L}$ 2-commute are the projections
\[
\eta \coloneqq \pr_{\C_2} : \psi_{12} \pr_{\K_1} \Rightarrow \psi_{22} \pr_{\K_2}
\quad\text{and}\quad
\xi \coloneqq \pr_{\G_2} : \varphi_{12} \pr_{\L_1} \Rightarrow \varphi_{22} \pr_{\L_2}.
\]
Note that
\[
\xi \circ \widehat{g} = \c_2 \circ \eta.
\]
The two vertical squares starting at the top of \eqref{04g8a4em} are commutative, i.e.
\[
g_i \circ \pr_{\K_i} = \pr_{\L_i} \circ \widehat{g},
\]
for $i = 1, 2$.
The morphisms in the composition $\M_2 \circ \M_1$ are given by $\widehat{\psi}_1 \coloneqq \psi_{11} \circ \pr_{\K_1}$, $\widehat{\psi}_2 \coloneqq \psi_{23} \circ \pr_{\K_2}$, $\widehat{\varphi}_1 \coloneqq \varphi_{11} \circ \pr_{\L_1}$, and $\widehat{\varphi}_2 \coloneqq \varphi_{23} \circ \pr_{\L_2}$.
Hence, the two natural transformations in $\M_2 \circ \M_1$ are
\[
\widehat{\theta}_1 \coloneqq \theta_{11} \circ \pr_{\K_1}\obj : \c_1\widehat{\psi}_1 \Rightarrow \widehat{\varphi}_1 \widehat{g}
\quad\text{and}\quad
\widehat{\theta}_2 \coloneqq \theta_{23} \circ \pr_{\K_2}\obj : \c_2 \widehat{\psi}_2 \Rightarrow \widehat{\varphi}_2 \widehat{g}.
\]
By Theorem \ref{rt0qkos3}, the 1-shifted Lagrangian structure on $\widehat{\L} \to \G_1 \times \G_3^-$ is given by 
\[
\widehat{\gamma} \coloneqq \pr_{\L_1}^*\gamma_1 + \pr_{\L_2}^*\gamma_2 - \xi^*\omega_2.
\]
The connecting forms of $\M_1$, $\M_2$ and $\M_2 \circ \M_1$ are given by
\begin{align*}
\delta_1 &\coloneqq -g_1^*\gamma_1 + \theta_{11}^*\omega_1 - \theta_{12}^*\omega_2 \\
\delta_2 &\coloneqq -g_2^*\gamma_2 + \theta_{22}^*\omega_2 - \theta_{23}^*\omega_3 \\
\widehat{\delta} &\coloneqq -\widehat{g}^*\widehat{\gamma} + \widehat{\theta}_1^*\omega_1 - \widehat{\theta}_2^*\omega_3,
\end{align*}
respectively.
It follows that
\begin{align}
\widehat{\delta}
&= -\pr_{\K_1}^*g_1^*\gamma_1 - \pr_{\K_2}^*g_2^*\gamma_2 + \widehat{g}^*\xi^*\omega_2 + \pr_{\K_1}^* \theta_{11}^*\omega_1 - \pr_{\K_2}^* \theta_{23}^*\omega_3 \nonumber \\
&= \pr_{\K_1}^*\delta_1 + \pr_{\K_2}^*\delta_2 + \eta^*c_2^*\omega_2 + \pr_{\K_1}^*\theta_{12}^*\omega_2 - \pr_{\K_2}^* \theta_{22}^*\omega_2 \nonumber \\
&= \pr_{\K_1}^*\delta_1 + \pr_{\K_2}^*\delta_2 + \eta^*c_2^*\omega_2 + \zeta, \label{3vlt24os}
\end{align}
where
\[
\zeta \coloneqq \pr_{\K_1}^*\theta_{12}^*\omega_2 - \pr_{\K_2}^* \theta_{22}^*\omega_2.
\]
Then, by repeated use of Lemma \ref{wip6t5yh}\ref{u0vnf6px}, we have
\begin{align*}
\sss^*\zeta &= \pr_{\K_1}^*(\ttt^*\theta_{12}^*\omega_2 + \psi_{12}^*c_2^*\omega_2 - g_1^*\varphi_{12}^*\omega_2) - \pr_{\K_2}^*(\ttt^*\theta_{22}^*\omega_2 + \psi_{22}^*c_2^*\omega_2 - g_2^*\varphi_{22}^*\omega_2) \\
&= \ttt^*\zeta + \sss^*\eta^*c_2^*\omega_2 - \ttt^*\eta^*c_2^*\omega_2 - \widehat{g}^*\pr_{\L_1}^*\varphi_{12}^*\omega_2 + \widehat{g}^* \pr_{\L_2}^* \varphi_{22}^*\omega_2 \\
&= \ttt^* \zeta + \sss^*\eta^*c_2^*\omega_2 - \ttt^*\eta^*c_2^*\omega_2 + \ttt^*\widehat{g}^*\xi^*\omega_2 - \sss^*\xi^*\omega_2\\
&= \ttt^*\zeta.
\end{align*}
Moreover,
\begin{align*}
d\zeta &= \pr_{\K_1}^*\theta_{12}^*(\sss^*\phi_2 - \ttt^*\phi_2) - \pr_{\K_2}^*\theta_{22}^*(\sss^*\phi_2 - \ttt^*\phi_2) \\
&= \pr_{\K_1}^*(\psi_{12}^*c_2^* \phi_2 - g_1^* \varphi_{12}^*\phi_2) - \pr_{\K_2}^*(\psi_{22}^*c_2^*\phi_2 - g_2^*\varphi_{22}^*\phi_2) \\
&= \sss^*\eta^* c_2^*\phi_2 - \widehat{g}^*\pr_{\L_1}^*\varphi_{12}^*\phi_2 - \ttt^*\eta^*c_2^*\phi_2 + \widehat{g}^* \pr_{\L_2}^* \varphi_{22}^* \phi_2 \\
&= \sss^*\eta^*c_2^*\phi_2 - \sss^* \widehat{g}^* \xi^* \phi_2 - \ttt^* \eta^* c_2^*\phi_2 + \ttt^* \widehat{g}^* \xi^* \phi_2 \\
&= 0.
\end{align*}
By Proposition \ref{m9pdmd75}, there is a basic closed 2-form $\beta$ on $\C_3$ such that $\zeta = \widehat{\psi}_2^*\beta = \pr_{\K_2}^* \psi_{23}^*\beta$.

Let $L_1 \in \coiso(c_1)$, $L_2 \coloneqq \transfer_{\M_1}(L_1) \in \coiso(c_2)$, $L_3 \coloneqq \transfer_{\M_2}(L_2) \in \coiso(c_3)$, and $\widehat{L}_3 \coloneqq \transfer_{\M_2 \circ \M_1} \in \coiso(c_3)$.
We need to show that $\widehat{L}_3$ and $L_3$ are gauge equivalent.
We have
\begin{align}
\psi_{11}^*L_1 &= \psi_{12}^*L_2 - \Gamma_{\delta_1} \label{dlkc6umi} \\
\psi_{22}^*L_2 &= \psi_{23}^*L_3 - \Gamma_{\delta_2} \label{teurg3x6} \\
\pr_{\K_1}^* \psi_{11}^*L_1 &= \pr_{\K_2}^* \psi_{23}^* \widehat{L}_3 - \Gamma_{\widehat{\delta}}. \label{upfw4j2f}
\end{align}
Now, \eqref{teurg3x6} implies that
\begin{align*}
\pr_{\K_1}^* \psi_{12}^*L_2
&= \eta^* \sss^* L_2 \\
&= \eta^* (\ttt^*L_2 - L_{c_2^*\omega_2}) \\
&= \pr_{\K_2}^*\psi_{22}^*L_2 - \Gamma_{\eta^*c_2^*\omega_2} \\
&= \pr_{\K_2}^*\psi_{23}^*L_3 + \Gamma_{-\pr_{\K_2}^*\delta_2 - \eta^*c_2^*\omega_2}.
\end{align*}
Making use of \eqref{dlkc6umi}, we then get
\[
\pr_{\K_1}^*\psi_{11}^*L_1 = \pr_{\K_2}^*\psi_{23}^*L_3 + \Gamma_{-\pr_{\K_1}^*\delta_1 - \pr_{\K_2}^*\delta_2 - \eta^*c_2^*\omega_2}.
\]
Inserting this in \eqref{upfw4j2f} and using \eqref{3vlt24os}, we get
\[
\pr_{\K_2}^*\psi_{23}^*\widehat{L}_3
=
\pr_{\K_2}^*\psi_{23}^*L_3 + \Gamma_{\widehat{\delta} - \pr_{\K_1}^*\delta_1 - \pr_{\K_2}^*\delta_2 - \eta^*c_2^*\omega_2}
=
\pr_{\K_2}^*\psi_{23}^*L_3 + \Gamma_{\zeta}
=
\pr_{\K_2}^*\psi_{23}^*(L_3 + \Gamma_\beta),
\]
and hence $\widehat{L}_3 = L_3 + \Gamma_\beta$.
\end{proof}

Finally, the dependency of $\transfereq_\ME$ on $\ME$ is expressed by the following result.

\begin{proposition}\label{kvvrwful}
Let $\c_i : \C_i \to \G_i$ for $i = 1, 2$ be Lie groupoid morphisms, where $\G_i$ are quasi-symplectic groupoids.
Suppose that we have two symplectic Morita equivalences
\[
\begin{tikzcd}[column sep={5em,between origins},row sep={2.5em,between origins}]
& \K_i \arrow[swap]{dl} \arrow{dd} \arrow{dr} & \\
\C_1 \arrow[swap]{dd}{c_1} &  & \C_2 \arrow{dd}{c_2} \\
& \L_i \arrow{dl} \arrow[swap]{dr} & \\
\G_1 \arrow[rr, yshift=-1em, no head, decorate, decoration={brace,mirror}, swap, "\ME_i" {yshift=-4pt}] & & \G_2,
\end{tikzcd}
\qquad
\text{for $i = 1, 2$},
\]
related by a 2-commutative diagram of the form
\begin{equation}\label{dkhiewxt}
\begin{tikzcd}[column sep={2.5em,between origins},row sep={3.5em,between origins}]
& & & & \K_2 \arrow{dllll} \arrow{drr} \arrow{dd} & & \\
\C_1 \arrow{dd} & & & \widetilde{\K} \arrow{ur} \arrow{dl} & & & \C_2 \arrow{dd} \\
& & \K_1 \arrow{ull} & & \L_2 \arrow{dllll} \arrow{drr} & & \\
\G_1 & & & \widetilde{\L} \arrow[from=uu,crossing over] \arrow{ur} \arrow{dl} & & & \G_2 \\
& & \L_1 \arrow{ull} \arrow[from=uu,crossing over] \arrow{urrrr},
\arrow[from=3-3,to=2-7,crossing over]{urrrr} & & &
\end{tikzcd}
\end{equation}
%\begin{equation}\label{dkhiewxt}
%\begin{tikzcd}[column sep={5em,between origins},row sep={3em,between origins}]
%& & & \K_2 \arrow{dlll} \arrow{dr} \arrow{dd} & \\
%\C_1 \arrow{dd} & & \widetilde{\K} \arrow{ur} \arrow{dl} & & \C_2 \arrow{dd} \\
%& \K_1 \arrow{ul} & & \L_2 \arrow{dlll} \arrow{dr} & \\
%\G_1 & & \widetilde{\L} \arrow[from=uu,crossing over] \arrow{ur} \arrow{dl} & & \G_2 \\
%& \L_1 \arrow{ul} \arrow[from=uu,crossing over] \arrow{urrr},
%\arrow[from=3-2,to=2-5,crossing over]{urrr} & & &
%\end{tikzcd}
%\end{equation}
where $\widetilde{\K} \to \K_i$ and $\widetilde{\L} \to \L_i$ are Morita morphisms.
Then $\transfereq_{\ME_1} = \transfereq_{\ME_2}$.
\end{proposition}

\begin{proof}
This follows from Lemma \ref{k1r3zchl} and Proposition \ref{3ni3b2b5}.
\end{proof}

\section{Hamiltonian actions of quasi-symplectic groupoids as 1-coisotropics}\label{616chpl6}

As observed by Calaque \cite[Example 1.31]{cal:15}, given a quasi-symplectic groupoid $\G$ acting on a manifold $M$, a 1-shifted Lagrangian structure on the projection $\G \ltimes M \to \G$ is precisely a 2-form on $M$ making the $\G$-action Hamiltonian in the sense of Xu \cite{xu:04}.
We now generalize this to a correspondence between 1-shifted coisotropic structures and Hamiltonian actions on Dirac manifolds.

%It is well-known that if a quasi-symplectic groupoid $\G \coloneqq (G \tto M)$ acts on a map $J : N \to M$, then giving $N$ the structure of a Hamiltonian $\G$-space (in the sense of \cite{xu:04}) is equivalent to giving $[N/\G] \to [M/\G]$ a 1-shifted Lagrangian structure.

\begin{definition}[{cf.\ \cite[Theorem 4.7]{bur-cra:05}}]\label{vofxgsye}
Let $(\G, \omega, \phi)$ be a quasi-symplectic groupoid acting on a Dirac manifold $(M, L, \eta)$ with moment map $\moment : M \to \G\obj$.
The action is \defn{Hamiltonian} if the following two conditions hold.
\begin{itemize}
\item
\defn{Compatibility.}
We have $\moment^*\phi = \eta$ and $\mathsf{a}^*L = \pr_M^*L + \Gamma_{\pr_\G^*\omega}$, where $\mathsf{a} : \G \ltimes M \to M$ is the action map and $\pr_\G$ and $\pr_M$ are the natural projections on $\G \ltimes M$.
\item
\defn{Non-degeneracy.}
$\ker \moment_* \cap \ker L = 0$.
\end{itemize}
\end{definition}

\begin{remark}\
\begin{itemize}
\item
The compatibility condition implies that the moment map $\moment : M \to \G\obj$ is a forward Dirac map with respect to the Dirac structure on $M$ induced by the quasi-symplectic structure (this follows from Lemma \ref{18e8u16h}).
The non-degeneracy condition is then the condition that $\moment$ is a strong Dirac map \cite{ale-bur-mei:09,bur-cra:05} (\S\ref{xlzo381e}).
\item
If $L$ is the graph of a 2-form, we recover Xu's notion of Hamiltonian $\G$-spaces \cite[\S3.1]{xu:04}.
\item
A Hamiltonian space for the trivial groupoid is a Poisson manifold.
\end{itemize}
\end{remark}

\begin{proposition}\label{ww8x8m9h}
Let $(\G, \omega, \phi)$ be a quasi-symplectic groupoid acting on a smooth manifold $M$. % with moment map $J : M \to \G\obj$.
Then a Dirac structure on $M$ making the $\G$-action Hamiltonian is precisely a 1-shifted coisotropic structure on the projection $\G \ltimes M \to \G$.
%\begin{equation}\label{n928xt75}
%\begin{tikzcd}
%\G \ltimes M  \arrow[shift right=1, swap]{d} \arrow[shift left=1]{d} \arrow{r}{\pr_G} &  G \arrow[shift right=1, swap]{d} \arrow[shift left=1]{d} \\
%N \arrow{r}{J} & M.
%\end{tikzcd}
%\end{equation}
\end{proposition}

\begin{proof}
Let $\mu : M \to \G\obj$ be the moment map and let $(L, \eta)$ be a Dirac structure on $M$.
The compatibility condition for $(L, \eta)$ to be a 1-shifted coisotropic structure on $\G \ltimes M \to \G$ is the same as the compatibility condition in Definition \ref{vofxgsye}.
It remains to show that the non-degeneracy conditions are equivalent.
First note that $A_{\G \ltimes M} = A_\G \ltimes M = \moment^*A_\G$.
The map \eqref{0v1fz7h6} is then given by
\begin{equation}\label{6z83mp46}
\moment^*A_\G \too \{((v, \alpha), a) \in L \times A_\G : \moment_* v = \aaa a, \alpha = \moment^*\IM_\omega a\}, \quad a \mtoo ((\act_* a, \moment^* \IM_\omega a), a),
\end{equation}
where $\act_* : \mu^*A_\G \to TM$ is the infinitesimal action.
We need to show that \eqref{6z83mp46} is surjective if and only if $\ker \moment_* \cap \ker L = 0$.
Suppose that $\ker \moment_* \cap \ker L = 0$ and let $((v, \alpha), a)$ be in the codomain of \eqref{6z83mp46}.
By viewing $a$ as an element of $\moment^*A_\G$, we have $(\act_* a, \moment^*\IM_\omega a) \in L$.
Since also $(v, \alpha) \in L$, we have $(v - \act_* a, 0) \in L$, so $v - \act_* a \in \ker \moment_* \cap \ker L = 0$.
Hence, \eqref{6z83mp46} is surjective.
Conversely, suppose that \eqref{6z83mp46} is surjective and let $v \in \ker \moment_* \cap \ker L$.
Then $((v, 0), 0)$ is in the codomain of \eqref{6z83mp46} so, by surjectivity, we have $v = \act_* 0 = 0$.
\end{proof}

%\begin{remark}
%The specialisation of Proposition \ref{ww8x8m9h} to Lagrangians appears (without proof) in \cite[Example 1.31]{cal:15}.
%Namely, a 1-shifted Lagrangian structure on $\G \ltimes M \to M$ is precisely a 2-form on $M$ making the $\G$-action Hamiltonian in the sense of Xu \cite{xu:04}.
%\end{remark}
%

By applying our results on intersections and Morita transfer of 1-coisotropics, we get:

\begin{theorem}%[Dirac reduction along a 1-coisotropic]
\label{ytrgmucn}
Let $\G$ and $\H$ be quasi-symplectic groupoids such that $\G \times \H$ acts on a Dirac manifold $M$ in a Hamiltonian way with moment map
\[
(\mu, \nu) : M \too \G\obj \times \H\obj.
\]
Let $\c : \C \to \G$ be a 1-coisotropic.
If the intersection $\C\obj \times_{\G\obj} M$ is clean and the quotient
\[
M \sll{\c} \G \coloneqq (\C\obj \times_{\G\obj} M) / \C
\]
has the structure of a smooth manifold such that the quotient map is a smooth submersion, then $M \sll{\c} \G$ has the structure of a Hamiltonian $\H$-space.
If $\C$ acts locally freely on $\C\obj \times_{\G\obj} M$, then the fibre product $\C\obj \times_{\G\obj} M$ is transverse.
%If $\c' : \C' \to \G$ is a 1-coisotropic equivalent to $\c$, then $M \sll{\c'} \G \cong M \sll{\c} \G$ as Hamiltonian $\H$-spaces.
\end{theorem}

\begin{proof}
We apply the theorem on strong intersections of 1-shifted coisotropic correspondences (Theorem \ref{2jy36i7i}) to
\[
\begin{tikzcd}[column sep={4em,between origins},row sep={4em,between origins}]
& \mathcal{C}^- \arrow{dr} \arrow{dl} & & (\G \times \H) \ltimes M \arrow{dl} \arrow{dr} & \\
\{\star\} & & \G & & \H^-.
\end{tikzcd}
\]
Note that the corresponding maps of Lie algebroids $A_{\C} \to A_\G \leftarrow A_{(\G \times \H) \ltimes M}$ are transverse since the latter is surjective.
Let $Z \coloneqq \C\obj \times_{\G\obj} M$.
Then the strong fibre product of $\C$ and $(\G \times \H) \ltimes M$ over $\G$ is the action groupoid $(\C \times \H) \ltimes Z$, where $\C \times \H$ acts on $Z$ via $(a, h) \cdot (x, p) = (\ttt(a), (\c(a), h) \cdot p)$ with moment map $(x, p) \mto (x, \nu(p))$, for $(x, p) \in Z$ and $(a, h) \in \C \times \H$.
The short exact sequence in Part \ref{u56ix79e} of Theorem \ref{2jy36i7i} then reduces to
\[
0 \too 0 \too \ker \mathsf{a}_{\C*} \too R^\circ \too 0,
\]
where $\mathsf{a}_{\C*} : A_\C \to TZ$ is the infinitesimal action of $\C$ on $Z$.
Since $Z \to Z/\C$ is a submersion, $\mathsf{a}_{\C*}$ has constant rank and hence so does $R$.
It follows that $\pr_{\C^0}^*L_\C + \pr_M^*L_M$ (where $L_\C$ and $L_M$ are the Dirac structures on $\C\obj$ and $M$ respectively) is a smooth Dirac structure on $Z$ and a 1-shifted coisotropic structure on $(\C \times \H) \ltimes Z \to \H$.
Now, we have a commutative diagram
\[
\begin{tikzcd}[row sep={4em,between origins},column sep={4em,between origins}]
(\C \times \H) \ltimes Z \arrow{rr} \arrow{dr} & & \H \ltimes Z/\C \arrow{dl} \\
& \H &
\end{tikzcd}
\]
of Lie groupoid morphisms.
We claim that the horizontal arrow is a weak Morita morphism.
This reduces to the statement that the map
\[
\H \ltimes Z \too Z/\C, \quad (h, (x, p)) \mtoo [(x, h \cdot p)]
\]
is a surjective submersion.
This follows from the fact that the action map $\H \ltimes Z \to Z$ and the quotient map $\pi : Z \to Z/\C$ are surjective submersions. %(it is the target map of the action groupoid) 
By Theorem \ref{vazibeld}, $\pi_*(\pr_{\C\obj}^*L_\C + \pr_M^*L_M)$ is a 1-shifted coisotropic structure on $\H \ltimes Z/\C \to \H$.
By Proposition \ref{ww8x8m9h}, it gives $Z/\C$ the structure of a Hamiltonian $\H$-space.
\end{proof}

In particular, by using the 1-Lagrangian $\G|_\O \to \G$ induced by an orbit (Proposition \ref{3pl9nrs8}), we get:

\begin{corollary}\label{2wut6us3}
Let $\G$ and $\H$ be quasi-symplectic groupoids such that $\G \times \H$ acts on a Dirac manifold $M$ in a Hamiltonian way with moment map $(\mu, \nu) : M \to \G\obj \times \H\obj$.
Let $\O \s \G\obj$ be an orbit.
If
\[
M \sll{\O} \G \coloneqq \mu^{-1}(\O) / (\G|_\O)
\]
is a manifold, then it has the structure of a Hamiltonian $\H$-space.
\end{corollary}

\begin{example}\
\begin{itemize}
\item
In the special case where $M$ is non-degenerate, Corollary \ref{2wut6us3} recovers Xu's reduction \cite[Theorem 3.18]{xu:04}.
In particular, it includes Marsden--Weinstein--Meyer symplectic reduction \cite{mar-wei:74,mey:73} (when $\G = T^*G$ is the cotangent groupoid of a Lie group $G$), quasi-Hamiltonian reduction \cite{ale-mal-mei:98} (when $\G = G \times G$ is the double), Mikami--Weinstein reduction \cite{mik-wei:88} (when $\G$ is a symplectic groupoid), and Lu's reduction \cite{lu:91} (when $\G$ is the integration of the dual $G^*$ of a Poisson-Lie group).

\item
When $\G$ is source-simply-connected, Corollary \ref{2wut6us3} recovers the Dirac reduction of Bursztyn--Crainic \cite{bur-cra:05}.
In particular, it includes quasi-Poisson reduction \cite{ale-kos-mei:02}.

\item
When $M$ is a symplectic manifold, $\c : \C \hookrightarrow \G$ is a subgroupoid with a trivial 1-shifted Lagrangian structure $\gamma = 0$, and $\H = \{\star\}$, Theorem \ref{ytrgmucn} recovers the notion of symplectic reduction along a submanifold introduced in \cite{cro-may:22}.
In particular, it includes symplectic cutting \cite{ler:95,ler-mei-tol-woo:98}, symplectic implosion \cite{gui-jef-sja:02}, preimages of Poisson transversals \cite{fre-mar:17}, and the Ginzburg--Kazhdan construction \cite{gin.kaz:23} of the Moore--Tachikawa topological quantum field theory \cite{moo-tac:12}.

\item
When $\c : \C \hookrightarrow \G$ is a subgroupoid with a 1-shifted Lagrangian structure, Theorem \ref{ytrgmucn} gives a global version of the notion of reduction of strong Dirac maps introduced in \cite{bal-may:22}.
In particular, it includes group-valued implosion \cite{hur-jef-sja:06}, multiplicative Whittaker reduction \cite{bal:22}, quasi-Poisson reduction relative to a subgroup \cite{lib-sev:15,sev:15}, and a multiplicative version of the Moore--Tachikawa varieties \cite{moo-tac:12}.
\end{itemize}

\end{example}

\section{Shifted coisotropic structures for differentiable stacks}\label{dy26su68}

We now define 1-shifted coisotropic structures on morphisms of differentiable stacks.
Let $\mathbf{DiffStack}$ be the 2-category of differentiable stacks \cite{beh-xu:11}, i.e.\ presentable stacks over the site of smooth manifolds with the Grothendieck topology of open covers.
%Let $\mathbf{Lie}[\E^{-1}]$ be the bicategory of Lie groupoids localised at \essequivs.
%See e.g.\ \cite[Appendix A]{ber-ler:20} for a brief review of bicategories.
Recall that there is an equivalence of bicategories, called the \defn{classifying functor}
\begin{equation}\label{ifpavzpw}
\mathbf{B} : \mathbf{Lie}[\E^{-1}] \too \mathbf{DiffStack}
\end{equation}
where $\mathbf{Lie}[\E^{-1}]$ is the bicategory of Lie groupoids localised at essential equivalences \cite{pro:96}.
That is, an object of $\mathbf{Lie}[\E^{-1}]$ is a Lie groupoid, a 1-morphism from $\H$ to $\G$ is a span of Lie groupoid morphisms
\begin{equation}\label{skxte4yr}
\begin{tikzcd}[row sep={2em,between origins},column sep={4em,between origins}]
& \C \arrow[swap]{dl}{\simeq} \arrow{dr} & \\
\H & & \G
\end{tikzcd}
\end{equation}
where $\simeq$ is an essential equivalence, and a 2-morphism between $\H \overset{\simeq}{\leftarrow} \C_i \to \G$ for $i = 1, 2$ is a 2-commutative diagram
\begin{equation}\label{hx1x0jqk}
\begin{tikzcd}
& \C_1 \arrow[swap]{dl}{\simeq} \arrow{dr} & \\
\H & \K \arrow{u}{\simeq} \arrow[swap]{d}{\simeq} & \G \\
& \C_2. \arrow{ul}{\simeq} \arrow{ur}
\end{tikzcd}
\end{equation}
This equivalence is discussed in more details, for example, in \cite[Theorem 2.3]{hep:09}, \cite[Theorem I.2.1]{car:11}, \cite[\S4.2]{hof-sja:21}, and \cite{vil:18} (see also \cite{beh-xu:11, blo:08, ler:10, met:03, noo:05}.)
In particular, essential equivalences of Lie groupoids are mapped to weak equivalences in $\mathbf{DiffStack}$ (i.e.\ morphisms with an inverse up to 2-morphism).
By Remark \ref{nhqq70ch}, we may replace essential equivalences with Morita morphisms.

%is an equivalence of bicategories [some references] (see also \cite{hof-sja:21} which also uses this equivalence).
%This is a standard result that is used in many other works, such as \cite{hof-sja:21,hof:20} or \cite[Theorem 2.3]{hep:09}.
%As explained in \cite{hep:09}, ``Pronk's proof is only given for groupoids with source and target map \'etale, and for stacks admitting an \'etale atlas, but the same techniques used there give the full result above.'', so we can do as \cite{hep:09} and attribute the result to Pronk \cite[Corollary 7]{pro:96}.
%
A \defn{presentation} of a differentiable stack $\mathbf{X}$ is a Lie groupoid $\G$ together with a weak equivalence $\mathbf{B}\G \to \mathbf{X}$ in $\mathbf{DiffStack}$.
The fact that \eqref{ifpavzpw} is an equivalence of bicategories implies that every differentiable stack has a presentation and that two Lie groupoids present the same stack if and only if they are Morita equivalent.

Let us now recall the definition 1-shifted symplectic structures on differentiable stacks via Lie groupoids (see also, e.g., \cite{get:14,cal:21,cue-zhu:23,mag-tor-vit:23}).
%\cite[Definition 5.4]{mag-tor-vit:23}).
Let $\mathbf{X}$ be a differentiable stack.
Define an equivalence relation on presentations of $\mathbf{X}$ by quasi-symplectic groupoids as follows.
We say that $(\G_1, \omega_1, \phi_1) \sim (\G_2, \omega_2, \phi_2)$ if there is a symplectic Morita equivalence
\begin{equation}\label{b22utjui}
\begin{tikzcd}[row sep={2em,between origins},column sep={4em,between origins}]
& \L \arrow{dl} \arrow{dr} & \\
\G_1 & & \G_2
\end{tikzcd}
\end{equation}
such that the induced diagram of weak equivalences in $\mathbf{DiffStack}$
%\begin{equation}\label{7t9qv00r}
%\begin{tikzcd}[column sep={4em,between origins},row sep={4em,between origins}]
%\B\G_1 \arrow{dr} \arrow{rr}{\B\L} & & \B\G_2 \arrow{dl} \\
%& \mathbf{X}
%\end{tikzcd}
%\end{equation}
\begin{equation}\label{o1mxhizj}
\begin{tikzcd}[column sep={3em,between origins},row sep={3em,between origins}]
&\B\L \arrow{dl} \arrow{dr} & \\
\B\G_1 \arrow{dr} & & \B\G_2 \arrow{dl} \\
& \mathbf{X} &
\end{tikzcd}
\end{equation}
2-commutes.
Reflexivity and symmetry are clear, and transitivity follows by Proposition \ref{9j2iully}.

\begin{definition}\label{s5okb7r6}
A \defn{1-shifted symplectic structure} on a differentiable stack $\mathbf{X}$ is an equivalence class $\mathbf{\Omega}$ of presentations of $\mathbf{X}$ by quasi-symplectic groupoids.
A \defn{symplectic presentation} of a 1-shifted symplectic differentiable stack $(\mathbf{X}, \mathbf{\Omega})$ is a representative of $\mathbf{\Omega}$.
\end{definition}

To show that this genuinely defines a structure on the differentiable stack, % independent of the equivalence \eqref{ifpavzpw}, 
we need to show that choosing a 1-shifted symplectic structure on $\mathbf{X}$ does not single out any particular subclass of presentations of $\mathbf{X}$.
This is the content of the following proposition.
%Here, a \defn{gauge transformation} of a quasi-symplectic groupoid $(\G, \omega, \phi)$ consists of modifying the structure as $(\omega + \sss^*\gamma - \ttt^*\gamma, \phi + d\gamma)$ for a 2-form $\gamma$ on $\G\obj$.
%A gauge transformation is always a quasi-symplectic groupoid, and two quasi-symplectic structures on the same groupoid are gauge equivalent if and only if they are symplectically Morita equivalent.

\begin{proposition}\label{6x6gdfs5}
Let $(\mathbf{X}, \mathbf{\Omega})$ be a 1-shifted symplectic differentiable stack.
For any presentation $\G$ of $\mathbf{X}$, there is a quasi-symplectic structure $(\omega, \phi)$ on $\G$ such that $(\G, \omega, \phi) \in \mathbf{\Omega}$.
Moreover, $(\omega, \phi)$ is unique up to gauge transformation.
\end{proposition}

\begin{proof}
%We will denote quasi-symplectic structures by a single symbol $\Omega = (\omega, \phi)$ to simplify the notation.
Let $(\G_1, \omega_1, \phi_1) \in \mathbf{\Omega}$ and let $\G_2$ be a presentation of $\mathbf{X}$.
%We need to show that there is a quasi-symplectic structure $\Omega_2$ on $\G_2$ together with a symplectic Morita equivalence $(\G_1, \Omega_1) \leftarrow (\L, \gamma) \to (\G_2, \Omega_2)$, and that if $\Omega_2'$ is any other quasi-symplectic structure with this property, then $\Omega_2$ and $\Omega_2'$ are cohomologous.
We have weak equivalences $\B\G_1 \simeq \mathbf{X} \simeq \B \G_2$, so there is a Morita equivalence \eqref{b22utjui} such that \eqref{o1mxhizj} 2-commutes.
By Theorem \ref{jpmgv5l0}, there is a quasi-symplectic structure $(\omega_2, \phi_2)$ on $\G_2$ such that \eqref{b22utjui} has the structure of a symplectic Morita equivalence.
Therefore, $(\G_2, \omega_2, \phi_2) \in \boldsymbol{\Omega}$.
%Then $H^3(\G_1) \cong H^3(\L_1) \cong H^3(\G_2)$, so there exists $[\Omega_2] \in H^3(\G_2)$ such that their pullback to $H^3(\L_1)$ are cohomologous.
%It follows that $\Omega_2$ is a quasi-symplectic structure (why non-degenerate?) on $\G_2$ and there is a 1-shifted Lagrangian structure $\gamma_1$ on $\L_1$ such that $\L_1$ is a symplectic Morita equivalence.
%This shows existence.
Suppose that there is another quasi-symplectic structure $(\omega_2', \phi_2')$ on $\G_2$ and a symplectic Morita equivalence $(\G_1, \omega_1, \phi_1) \leftarrow \L' \to (\G_2, \omega_2', \phi_2')$ such that \eqref{o1mxhizj} with $\L'$ in place of $\L$ 2-commutes.
In particular, there is a 2-commutative diagram of weak equivalences in $\mathbf{DiffStack}$
\begin{equation}\label{ufgrigp0}
\begin{tikzcd}[row sep={3em,between origins},column sep={3em,between origins}]
&\B\L \arrow{dl} \arrow{dd} \arrow{dr} & \\
\B\G_1 & & \B\G_2 \\
&\B\L' \arrow{ul} \arrow{ur}. &
\end{tikzcd}
\end{equation}
It follows that there is a 2-commutative diagram of Morita morphisms as in \eqref{poby95jm} and hence $(\omega_2, \phi_2)$ and $(\omega_2', \phi_2')$ are gauge equivalent by Proposition \ref{0cc6w7hy}.
%\begin{equation}\label{9zi1a5br}
%\begin{tikzcd}[row sep={4em,between origins},column sep={4em,between origins}]
%& \L \arrow{dl} \arrow{dr} & \\
%\G_1 & \widehat{\L} \arrow{u} \arrow{d} & \G_2 \\
%& \L', \arrow{ul} \arrow{ur} &
%\end{tikzcd}
%\end{equation}
%for some Lie groupoid $\widehat{\L}$.
%Recall that for a Lie groupoid morphism $f : \G \to \H$, the induced map on Lie groupoid cohomology $f^* : H^{\bullet}(\H) \to H^{\bullet}(\G)$ is invariant under homotopy of $f$ (see e.g.\ \cite{beh:04}).
%Hence, passing to cohomology, \eqref{9zi1a5br} gives a \emph{commutative} diagram of vector space isomorphisms
%\begin{equation}\label{s99zulcs}
%\begin{tikzcd}[row sep={5em,between origins},column sep={5em,between origins}]
%& H^3(\L) \arrow[from=dl] \arrow[from=dr] & \\
%H^3(\G_1) & H^3(\widehat{\L}) \arrow[from=u] \arrow[from=d] & H^3(\G_2) \\
%& H^3(\L') \arrow[from=ul] \arrow[from=ur]. &
%\end{tikzcd}
%\end{equation}
%As we recalled in the proof of Theorem \ref{jpmgv5l0}, $(\G_1, \omega_1, \phi_1) \leftarrow \L \to (\G_2, \omega_2, \phi_2)$ has the structure of a symplectic Morita equivalence if and only if $[(\omega_1, \phi_1)]$ and $[(\omega_2, \phi_2)]$ are mapped to the same element in $H^3(\L)$.
%It follows that the cohomology classes $[(\omega_2, \phi_2)]$ and $[(\omega_2', \phi_2')]$ are mapped to the same element of $H^3(\widehat{\L})$ in \eqref{s99zulcs} and hence are equal.
%In other words, $(\omega_2, \phi_2)$ and $(\omega_2', \phi_2')$ are cohomologous, i.e.\ gauge equivalent \cite[\S4.1]{xu:04}.
\end{proof}

\begin{remark}\label{kel26qh6}
An alternative approach to defining structures on differentiable stacks is to focus only on stacks of the form $\B\G$ (i.e.\ not just weakly equivalent but \emph{equal} to $\B\G$) and define the structure as a Morita invariant equivalence class of objects on $\G$ (see e.g.\ \cite{hoy-fer:19,cra-mes:19,wei:09,tu-xu-lau:04}).
With this approach, we would define a 1-shifted symplectic structure on $\B\G$ as a gauge equivalence class of quasi-symplectic structures on $\G$.
See also Remark \ref{3cbowb68}.
%We prefer the above approach since it is more general and yields to a cleaner description of structures on \emph{morphisms} of stacks; see Remark \ref{3cbowb68}.
%Both approaches are equivalent, but the latter has the disadvantage of not directly defining the structure on an arbitrary differentiable stack $\mathbf{X}$, and becomes less transparent (to the author) when constructing structures on \emph{morphisms} of stacks, such as coisotropic structures; see Remark \ref{3cbowb68}.
%so we prefer the more general one in Definition \ref{s5okb7r6}.
\end{remark}

We now emulate the above discussion to define 1-shifted coisotropic structures on morphisms of differentiable stacks.
A \defn{presentation} of a morphism of differentiable stacks $\mathbf{c} : \mathbf{C} \to \mathbf{X}$ is a Lie groupoid morphism $\c : \C \to \G$ together with a 2-commutative diagram
\[
\begin{tikzcd}
\mathbf{B}\C \arrow{r}{\simeq} \arrow[swap]{d}{\B \c} & \mathbf{C} \arrow{d}{\mathbf{\c}} \\
\mathbf{B}\G \arrow{r}{\simeq} & \mathbf{X},
\end{tikzcd}
\]
in $\mathbf{DiffStack}$, where $\simeq$ are weak equivalences.
Every morphism of differentiable stacks has a presentation.
Moreover, for any presentation $\G$ of $\mathbf{X}$, there is a presentation of $\mathbf{\c}$ whose codomain is $\G$.

Let $(\mathbf{X}, \mathbf{\Omega})$ be a 1-shifted symplectic differentiable stack and $\mathbf{\c} : \mathbf{C} \to \mathbf{X}$ a morphism of differentiable stacks.
A \defn{symplectic presentation} of $\mathbf{c}$ is a symplectic presentation $(\G, \omega, \phi)$ of $(\mathbf{X}, \mathbf{\Omega})$ together with a Lie groupoid morphism $\c : \C \to \G$ presenting $\mathbf{c}$.
There is a natural equivalence relation on the set of pairs $(\c, L)$, where $\c$ is a symplectic presentation of $\mathbf{c}$ and $L$ is a 1-shifted coisotropic structure on $c$.
Namely, $(\c_1, L_1) \sim (\c_2, L_2)$ if they are coisotropically Morita equivalent as in Definition \ref{w43yacl2}\ref{bx4cvb7v} via a 2-commutative diagram
\begin{equation}\label{r5rwl431}
\begin{tikzcd}[column sep={4em,between origins},row sep={2em,between origins}]
& \K \arrow[swap]{dl} \arrow{dd} \arrow{dr} & \\
\C_1 \arrow[swap]{dd}{\c_1} &  & \C_2 \arrow{dd}{\c_2} \\
& \L \arrow{dl} \arrow[swap]{dr} & \\
\G_1 & & \G_2,
\end{tikzcd}
\end{equation}
such that the induced diagram of weak equivalences in $\mathbf{DiffStack}$
\begin{equation}\label{o6kq4ddf}
\begin{tikzcd}[column sep={4em,between origins},row sep={1.7em,between origins}]
& \B\K \arrow[swap]{dl} \arrow{ddd} \arrow{drr} & & \\
\B\C_1 \arrow[swap]{ddd} \arrow[crossing over]{drr} &  &  & \B\C_2 \arrow{ddd} \arrow{dl} \\
&  & \mathbf{C} &  \\
& \B\L \arrow{dl} \arrow[swap]{drr} & & \\
\B\G_1 \arrow{drr} & & & \B\G_2 \arrow{dl} \\
& & \mathbf{X} \arrow[from=uuu,crossing over] &
\end{tikzcd}
\end{equation}
2-commutes.
Reflexivity and symmetry are clear, and transitivity follows from Proposition \ref{3ni3b2b5}.

\begin{definition}
Let $(\mathbf{X}, \mathbf{\Omega})$ be a 1-shifted symplectic differentiable stack and $\mathbf{c} : \mathbf{C} \to \mathbf{X}$ a morphism of differentiable stacks.
A \defn{1-shifted coisotropic structure} on $\mathbf{c}$ is an equivalence class $\mathbf{L}$ of pairs $(\c, L)$, where $c$ is a symplectic presentation of $\mathbf{c}$ and $L$ is a 1-shifted coisostropic structure on $c$.
\end{definition}

As for 1-shifted symplectic structures, choosing a 1-shifted coisotropic structure on $\mathbf{c}$ does not single out a particular subclass of presentations of $\mathbf{c}$, and hence is genuinely a structure on the morphism of stacks:

\begin{proposition}
Let $(\mathbf{X}, \mathbf{\Omega})$ be a 1-shifted symplectic differentiable stack, $\mathbf{c} : \mathbf{C} \to \mathbf{X}$ a morphism of differentiable stacks, and $\mathbf{L}$ a 1-shifted coisotropic structure on $\mathbf{c}$.
For any symplectic presentation $c$ of $\mathbf{\c}$, there is a 1-shifted coisotropic structure $L$ on $\c$ such that $(\c, L) \in \mathbf{L}$.
Moreover, $L$ is unique up to gauge equivalence.
\end{proposition}

\begin{proof}
For existence, let $\c_1 : (\C_1, L_1) \to (\G_1, \omega_1, \phi_1)$ be a representative of $\mathbf{L}$ and $\c_2 : \C_2 \to (\G_2, \omega_2, \phi_2)$ a symplectic presentation of $\mathbf{c}$.
%We must show that there is a coisotropic structure $L_2$ on $\c_2$ such that $(\c_1, L_1) \sim (\c_2, L_2)$.
Since $\B \c_1$ and $\B \c_2$ are both weakly equivalent to $\mathbf{C} \to \mathbf{X}$, there is a diagram of Morita equivalences of the form \eqref{r5rwl431} such that \eqref{o6kq4ddf} 2-commutes.
%We claim that we can assume, without loss of generality, that $\L$ is a symplectic Morita equivalence.
Moreover, since $(\G_1, \omega_1, \phi_1)$ and $(\G_2, \omega_2, \phi_2)$ are symplectic presentations of $(\mathbf{X}, \mathbf{\Omega})$, there is a symplectic Morita equivalence $\G_1 \leftarrow \L' \to \G_2$ such that \eqref{o1mxhizj} with $\L'$ in place of $\L$ 2-commutes.
It follows that there is a 2-commutative diagram as in \eqref{ufgrigp0} and hence a 2-commutative diagram of Morita morphisms of the form \eqref{poby95jm}.
Note that $\G_1 \leftarrow \widehat{\L} \to \G_2$ is also a symplectic Morita equivalence.
Thus, by considering the fibre product $\K \htimes_\L \widehat{\L}$, we have a symplectic Morita equivalence between $\c_1$ and $\c_2$.
%
%Consider the 2-commutative diagram
%\[
%\begin{tikzcd}[column sep={4em,between origins},row sep={4em,between origins}]
%& & \K \arrow{dll} \arrow{dd} \arrow{drr} \\
%\C_1 \arrow{dd} & & & \K \htimes_\L \hat{\L} \arrow{ul} & \C_2 \arrow{dd} \\
%& & \L \arrow{drr} \arrow{dll} \\
%\G_1 & & & \hat{\L} \arrow{ul} \arrow{dl} \arrow[from=uu, crossing over] & \G_2 \\
%& & \L' \arrow{ull} \arrow{urr}
%\end{tikzcd}
%\]
%Then $\hat{\L}$ in \eqref{9zi1a5br} is a symplectic Morita equivalence between $\G_1$ and $\G_2$ and the fibre product $\K \htimes_\L \hat{\L}$ (Lemma \ref{ijs66i1n}) is a Morita equivalence between $\C_1$ and $\C_2$.
%Therefore, replacing $\K \to \L$ with $\K \htimes_\L \hat{\L} \to \hat{\L}$ if necessary, we may assume, without loss of generality, that $\L$ in \eqref{fvhn3xkk} is a symplectic Morita equivalence.
By Theorem \ref{vazibeld}, there is a coisotropic structure $L_2$ on $\c_2$ such that $(\c_1, L_1) \sim (\c_2, L_2)$, i.e.\ $(\c_2, L_2) \in \mathbf{L}$.
%In the notation of Section \ref{zpetash3}, $L_2 = \transfer_\ME(L_1)$, where $\ME$ is the diagram \eqref{fvhn3xkk}.

For uniqueness, let $\c : \C \to (\G, \omega, \phi)$ be a symplectic presentation of $\mathbf{c}$ and suppose that there are two coisotropic structures $L$ and $L'$ on $\c$ such that $(\c, L) \sim (\c, L')$.
We then have a symplectic Morita equivalence
\[
\begin{tikzcd}[column sep={4em,between origins},row sep={2em,between origins}]
& \K \arrow[swap]{dl}{\psi_1} \arrow{dd} \arrow{dr}{\psi_2} & \\
\C \arrow[swap]{dd}{\c} &  & \C \arrow{dd}{\c} \\
& \L \arrow{dl}{\varphi_1} \arrow[swap]{dr}{\varphi_2} & \\
\G & & \G,
\end{tikzcd}
\]
such that
\[
\begin{tikzcd}[column sep={4em,between origins},row sep={2em,between origins}]
& \B\K \arrow[swap]{dl} \arrow{dd} \arrow{dr} & \\
\B\C \arrow[swap]{dd} \arrow[equal,crossing over]{rr}  &  & \B\C \arrow{dd} \\
& \B\L \arrow{dl} \arrow[swap]{dr} & \\
\B\G \arrow[equal]{rr} & & \B\G  \\
\end{tikzcd}
\]
2-commutes.
It follows that $\psi_1$ and $\psi_2$ are homotopic.
By Lemma \ref{k1r3zchl}, $L$ and $L'$ are gauge equivalent.
\end{proof}

\begin{remark}\label{3cbowb68}
Continuing with Remark \ref{kel26qh6}, the reader who prefers to work only with quotient stacks may define 1-shifted coisotropic structures as follows.
In this approach, we consider the category induced by the bicategory $\mathbf{Lie}[\E^{-1}]$, i.e.\ a morphism $\mathbf{c} : \B\H \to \B\G$ is an equivalence class of spans \eqref{skxte4yr}, called \emph{fractions}, where two fractions are equivalent if they are related by a 2-commutative diagram \eqref{hx1x0jqk} \cite{bur-hoy:23, hoy:13}. %, pro:96}.
%Two fractions $\H \overset{\simeq}{\leftarrow} \C_1 \to \G$ and $\H \overset{\simeq}{\leftarrow} \C_2 \to \G$ are equivalent if there exists a 2-commutative diagram
%\begin{equation}\label{lsqygsj6}
%\begin{tikzcd}[row sep={4em,between origins},column sep={4em,between origins}]
%& \C_1 \arrow[swap]{dl}{\simeq} \arrow{dr} & \\
%\H & \K \arrow{u}{\simeq} \arrow[swap]{d}{\simeq} & \G \\
% & \C_2, \arrow{ul}{\simeq} \arrow{ur} &
%\end{tikzcd}
%\end{equation}
%where $\simeq$ are Morita morphisms.
Suppose that $\B\G$ is a 1-shifted symplectic stack in the sense of Remark \ref{kel26qh6}, i.e.\ $\G$ is quasi-symplectic.
We define a 1-shifted coisotropic structure on $\mathbf{\c}$ as an equivalence class of triples consisting of a fraction $\H \overset{\simeq}{\leftarrow} \C \to \G$ representing $\mathbf{c}$, a quasi-symplectic structure $(\omega, \phi)$ on $\G$ representing the 1-shifted symplectic structure on $\B\G$, and a 1-shifted coisotropic structure $L$ on $\C \to (\G, \omega, \phi)$.
Two such structures $\H \overset{\simeq}{\leftarrow} (\C_i, L_i) \to (\G, \omega_i, \phi_i)$ are equivalent if there is a 2-commutative diagram \eqref{hx1x0jqk} whose right-most square gives a coisotropic Morita equivalence
\[
\begin{tikzcd}[column sep={4em,between origins},row sep={2em,between origins}]
& \K \arrow[swap]{dl} \arrow{dd} \arrow{dr} & \\
\C_1 \arrow[swap]{dd} &  & \C_2 \arrow{dd} \\
& \G \arrow{dl} \arrow[swap]{dr} & \\
\G & & \G,
\end{tikzcd}
\]
as in Definition \ref{w43yacl2}\ref{bx4cvb7v}.
Proposition \ref{3ni3b2b5} shows that this is indeed an equivalence relation, Theorem \ref{vazibeld} shows that the equivalence class is Morita invariant, and Proposition \ref{kvvrwful} shows that it is independent of the choice of fraction.%, i.e.\ does not single out any particular class of fractions representating $\mathbf{c} : \B\H \to \B\G$.
\end{remark}

%\begin{theorem}[Properties of coisotropic structures]\
%\begin{enumerate}
%\item
%Let $(\mathbf{X}, \mathbf{\Omega})$ be a 1-shifted symplectic differentiable stack and $(\mathbf{C}_1, \mathbf{L}_1) \to \mathbf{X}$ and $(\mathbf{C}_2, \mathbf{L}_2) \to \mathbf{X}$ two 1-shifted coisotropics.
%If $\mathbf{C}_1 \times_{\mathbf{X}} \mathbf{C}_2$ is a stack, then it is a 0-shifted Poisson stack.
%\item
%The identity morphism $\mathbf{X} \to \mathbf{X}$ has a canonical coisotropic structure.
%\end{enumerate}
%\end{theorem}
%
%\begin{remark}
%So, we need to define $0$-shifted Poisson stacks (cf.\ \cite[Remark 4.2]{mag-tor-vit:23}) and explain the relationship between fibre products of stacks and homotopy fibre products of Lie groupoids, including a definition of transverse intersection of morphisms of stacks and how it relates to the case of Lie groupoids...
%The fibre product of stacks is described in \cite[p.\ 5]{blo:08}.
%\end{remark}
%

\bibliographystyle{plain}
\bibliography{coisotropics}
\end{document}